\documentclass[10pt]{amsart}

\usepackage[latin2]{inputenc}
\usepackage{amsmath}
\usepackage{bbm}
\usepackage{graphicx}
\usepackage{amssymb}
\usepackage{esint}
\usepackage{tikz}
\usepackage[dvipsnames]{xcolor}
\usepackage{floatrow}
\usepackage{color}
\usepackage{amsthm}
\usepackage{stackengine}
\usepackage{scalerel}
\usepackage{epsfig}
\usepackage[english]{babel}
\usepackage{mathtools}
\newtheorem{theorem}{Theorem}
\newtheorem{proposition}[theorem]{Proposition}
\newtheorem{lemma}[theorem]{Lemma}
\newtheorem{corollary}[theorem]{Corollary}

\usepackage[colorlinks=true,linkcolor=RoyalBlue,anchorcolor=black,citecolor=Green,filecolor=black,menucolor=black,runcolor=black,urlcolor=black]{hyperref}
\usepackage[page,toc,titletoc,title]{appendix}
\usepackage{accents}

    \providecommand{\PfStep}[2]{\noindent \ifnum\value{#1}=1\else\fi{\sc Step }\arabic{#1}\label{#2}\refstepcounter{#1}.} 

\newtheorem*{theorem*}{Theorem}
\newtheorem*{theorem***}{Theorem 1 \& Theorem 2 of \cite{GNO_final}}
\newtheorem*{theorem****}{Theorem 2 of \cite{GNO5}}

\def\Xint#1{\mathchoice
{\XXint\displaystyle\textstyle{#1}}%
{\XXint\textstyle\scriptstyle{#1}}%
{\XXint\scriptstyle\scriptscriptstyle{#1}}%
{\XXint\scriptscriptstyle\scriptscriptstyle{#1}}%
\!\int}
\def\XXint#1#2#3{{\setbox0=\hbox{$#1{#2#3}{\int}$ }
\vcenter{\hbox{$#2#3$ }}\kern-.6\wd0}}

\def\fint{\Xint-}

\allowdisplaybreaks[2]

\newcommand{\supp}{\operatorname{supp}}

\newcommand{\diam}{\operatorname{diam}}
\newcommand{\eps}{\varepsilon}
\newcommand{\RVE}{{\operatorname{RVE}}}
\newcommand{\dist}{\operatorname{dist}}

\newcommand{\loc}{\rm{loc}}

\newcommand{\R}{\mathbb{R}}
\newcommand{\OO}{\mathcal{O}}

\newcommand{\dd}{{\textrm{d}}}
\newcommand{\eg}{e.\,g.}
\newcommand{\as}{a.\,s.\ }
\newcommand{\ie}{i.\,e.\ }
\newcommand{\aen}{a.\,e.\ }
\newcommand{\Wlog}{W.\,l.\,o.\,g.\ }
\newcommand{\wolog}{w.\,l.\,o.\,g.\ }

\newcommand{\domain}{{\mathcal{O}}}

\renewcommand{\vec}[1]{{\mathbf{#1}}}

\newcommand{\en}[1]{\left< #1 \right>} % ensemble 
\providecommand{\dx}{\, \mathrm{d} x}

\newcommand*{\Hd}{\mathbb{H}^d_+}
\newcommand{\Rd}{\mathbb{R}}
\newcommand{\C}{\mathcal{C}}

\usepackage{textcomp}

\numberwithin{equation}{section}

\newcommand{\transpose}{*}

\begin{document}

\title{Boundary Layer Estimates in Stochastic Homogenization}

\author{Peter Bella}
\thanks{\hspace{-3mm}TU Dortmund, DE-44227, Dortmund, Germany}
\author{Julian Fischer}
\thanks{\hspace{-3mm}IST Austria, AT-3400, Klosterneuburg, Austria}
\author{Marc Josien}
\thanks{\hspace{-3mm}CEA, DES, IRESNE, DEC, SESC, LMCP, Cadarache, F-13108, Saint-Paul-Lez-Durance, France}
\author{Claudia Raithel}
\thanks{\hspace{-3mm}TU Wien, AT-1040, Wien, Austria}

%\address{TU Dortmund, DE-44227, Dortmund, Germany}
%\address{IST Austria, AT-3400, Klosterneuburg, Austria}
%\address{CEA, DES, IRESNE, DEC, SESC, LMCP, Cadarache, F-13108, Saint-Paul-Lez-Durance, France}
%\address{TU Wien, AT-1040, Wien, Austria}

\begin{abstract}
We prove quantitative decay estimates for the boundary layer corrector in stochastic homogenization in the case of a half-space boundary. Our estimates are of optimal order and show that the gradient of the boundary layer corrector features nearly fluctuation-order decay; its expected value decays even one order faster. As a corollary, we deduce estimates on the accuracy of the representative volume element (RVE) method for the computation of effective coefficients: in $d\geq 3$ dimensions our understanding of the decay of boundary layers enables us to justify an improved formula for the RVE method, based on a combination of oversampling with the Hill-Mandel condition.
\end{abstract}

\maketitle

\vspace{-8.2mm}

\tableofcontents

\section{Introduction}

In the past two decades an extensive theory of quantitative homogenization of random media has been developed. For the linear elliptic PDE
\begin{align}
\label{LinearElliptic0}
-\nabla \cdot \big(a(\tfrac{\cdot}{\eps})\nabla u_\eps\big) &= f && \text{on }\mathbb{R}^d
\end{align}
with a stationary random coefficient field $a$ with (suitably defined) decay of correlations, optimal-order convergence estimates in $\eps$ to the solution $\bar u$ to a constant-coefficient effective equation $-\nabla \cdot (\bar a \nabla \bar u) = f$ have been obtained \cite{ArmstrongKuusiMourratNew,GNO1,GloriaOttoNearOptimal}.
However, most of the existing stochastic homogenization literature has focused on the homogenization problem on the entire space $\mathbb{R}^d$.

Both in periodic and stochastic homogenization, it is not surprising that the presence of a domain boundary may lead to intricate effects and substantial new challenges: In the periodic setting, the domain boundary breaks the periodic structure, while in stochastic homogenization, the domain boundary breaks the stationarity (invariance with respect to spatial shifts) of the problem. Nevertheless, for the Dirichlet problem with sufficiently regular boundary data the rate of convergence of $u_\eps$ towards $\bar u$ in the interior of the domain is not affected by the presence of a domain boundary; for a scalar PDE this follows from the maximum principle, while for systems we refer to \cite{AvellanedaLinCPAM,AvellanedaLinBoundary}.
However, already in the periodic case imposing boundary data that oscillates on the $\eps$-scale --a problem that naturally arises in higher-order quantitative homogenization theories-- leads to highly challenging new effects \cite{AllaireAmar,ArmstrongKuusiMourratPrange,GerardVaretMasmoudi}, as nontrivial interactions of the oscillating boundary data with the oscillating coefficient field emerge. To the best of our knowledge, the corresponding problem in stochastic homogenization -- the fluctuating Dirichlet problem for a random elliptic operator -- has remained largely unexplored.

In the present work, we are concerned with the stochastic homogenization of the fluctuating Dirichlet problem for the elliptic PDE \eqref{LinearElliptic0} on a half-space $\Hd$. We prove optimal-order estimates on the decay of the boundary layer corrector. Our results have interesting applications concerning the accuracy of the approximation of the effective coefficient via the representative volume element method (RVE method).

\smallskip\smallskip\smallskip
\noindent
\textbf{The fluctuating/oscillating Dirichlet problem in homogenization.}
Before stating our main results, let us briefly discuss how the oscillating/fluctuating Dirichlet problem naturally emerges in the construction of higher-order homogenization theories.
In order to obtain a higher-order quantitative homogenization theory for elliptic PDEs on bounded domains, it is important to understand the impact of boundary layers on the solution: For instance, for periodic coefficient fields $a(\cdot)=a(\cdot+k), k\in \mathbb{Z}^d$, with $a=a^\transpose$, solutions to the linear elliptic PDE on the whole-space
\begin{equation}
\label{LinearElliptic}
\begin{aligned}
-\nabla \cdot \big(a(\tfrac{\cdot}{\eps})\nabla u_\eps\big) &= \nabla \cdot f && \text{on }\mathbb{R}^d
\end{aligned}
\end{equation}
may be approximated by the solution to a constant coefficient effective equation $-\nabla\cdot (\bar a \nabla \bar u)=\nabla \cdot f$ up to an error of second order; i.e.
\begin{align}
\label{ErrorBoundFullSpace}
||u_\eps-\bar u||_{H^{-1}(B_1)}
\lesssim
\eps^2,
\end{align}
provided that the right-hand side $f$ is sufficiently regular \cite{BensoussanLionsPapanicolaou,AllaireAmar,BellaFehrmanFischerOtto}. In contrast, the analogous estimate \emph{fails} for the Dirichlet problem on bounded domains, even for the homogeneous Dirichlet problem and even strictly in the interior of the domain: On bounded domains the estimate $||u_\eps-\bar u||_{L^p}\lesssim \eps$ cannot be improved by passing to weak norms.
The underlying reason for this failure is boundary layer effects that provide an $\eps$-order contribution to the \emph{effective boundary data} for $\bar u$ \cite{AllaireAmar,ArmstrongKuusiMourratPrange,GerardVaretMasmoudi}.

On a mathematical level, the failure of the derivation of estimates like \eqref{ErrorBoundFullSpace} in the case of the Dirichlet problem on domains may be seen as follows:
In order to prove an estimate like \eqref{ErrorBoundFullSpace}, one attempts to approximate $u_\eps$ by the so-called two-scale expansion; written up to second order, it reads for periodic and symmetric coefficient fields $a=a^\transpose$
\begin{align}
\label{TwoScaleExpansion}
\bar u(x) + \eps \phi_i(\tfrac{x}{\eps}) \partial_i \bar u(x) + \eps^2 \psi_{ij}(\tfrac{x}{\eps})\partial_i \partial_j \bar u(x),
\end{align}
where $\phi_i$ and $\psi_{ij}$ are the so-called first-order and second-order homogenization correctors.
The homogenization (cell-)corrector $\phi_i$ is defined as the $1$-periodic solution to the PDE $-\nabla \cdot (a(e_i+\nabla \phi_i))=0$; note that it is unique up to additive constants. If we attempt to apply the two-scale expansion ansatz \eqref{TwoScaleExpansion} on a domain $\domain\subset \mathbb{R}^d$, we immediately observe that it fails to attain the correct boundary data --which would be $\bar u$ on the domain boundary. While correcting for the second-order $O(\eps^2)$ boundary contribution in \eqref{TwoScaleExpansion} is not necessary in order to obtain a second-order homogenization error estimate, the first-order contribution $\eps \phi_i(\tfrac{x}{\eps}) \partial_i \bar u(x)$ must be corrected for. In other words, it is necessary to subtract a term of the form $\eps v(x)$ with $v$ solving
\begin{equation}
\begin{aligned}
-\nabla \cdot \big(a\big(\tfrac{\cdot}{\eps}\big)\nabla v\big)&=0 &&\text{in }\domain,
\\
v&=\phi_i(\tfrac{\cdot}{\eps})\partial_i \bar u &&\text{on }\partial\domain
\end{aligned}
\end{equation}
from the ansatz \eqref{TwoScaleExpansion}. We remark that the above equation is an example of the so-called \emph{oscillating Dirichlet problem}: The boundary data oscillates on the same scale as the coefficient field. As it turns out, the oscillations in the boundary data and in the coefficient field give rise to quite nontrivial nonlinear interactions --under suitable assumptions, $v$ may be described in the interior in terms of a suitable $\bar a$-harmonic function (meaning that it is in the kernel of the differential operator $-\nabla \cdot \bar{a} \nabla$) satisfying certain \emph{effective boundary conditions} on $\partial\domain$. Importantly, these effective boundary conditions are not given in terms of a simple average of the boundary data \cite{AllaireAmar,ArmstrongKuusiMourratPrange,GerardVaretMasmoudi}.

In the context of regularity theory for linear elliptic operators with oscillating coefficients, a particularly important instance of an oscillating Dirichlet problem is given by the \emph{boundary layer corrector} $\theta_i^\eps$, defined as the solution to the problem
\begin{equation}
\begin{aligned}
-\nabla \cdot \big(a\big(\tfrac{\cdot}{\eps}\big)\nabla \theta_i^\eps\big)&=0 &&\text{in }\domain,
\\
\theta_i^\eps&=\phi_i(\tfrac{\cdot}{\eps}) &&\text{on }\partial\domain.
\end{aligned}
\end{equation}
By its definition, subtracting $\varepsilon \theta_i^\eps$ from the whole-space corrector $\phi_i^\eps(\cdot):= \varepsilon \phi_i (\cdot / \varepsilon)$ corrects the boundary conditions of $\phi_i^\eps$ while maintaining the defining corrector equation $-\nabla \cdot (a(e_i+\nabla (\phi_i^\eps- \varepsilon \theta_i^\eps)))=0$. In particular, $\phi_i^\eps- \varepsilon \theta_i^\eps$ yields a first-order homogenization corrector in $\domain$ with homogeneous Dirichlet boundary conditions on the boundary $\partial\domain$.

In the case of periodic homogenization, the study of the oscillating Dirichlet problem has a long history and by now a relatively complete picture has emerged in the literature.
On half-spaces with rationally aligned normals (with respect to the underlying periodic structure), Allaire and Amar \cite{AllaireAmar} have shown exponential decay of the gradient of the boundary layer corrector away from the boundary; they subsequently used this result to homogenize the oscillating Dirichlet problem on polygonal domains with rationally aligned boundaries\footnote{\,However, this homogenization result only holds along subsequences.}. For half-spaces with Diophantine normals, G\'erard-Varet and Masmoudi \cite{GerardVaretMasmoudi2,GerardVaretMasmoudi} established superpolynomial decay of the boundary layer corrector gradients; as a consequence, they were able to formulate quantitative homogenization error estimates for the oscillating Dirichlet problem on uniformly convex domains, albeit with a suboptimal rate. Recently, Armstrong, Kuusi, Mourrat, and Prange \cite{ArmstrongKuusiMourratPrange} have obtained optimal convergence estimates in $L^p$ for $p\geq 2$ for the oscillating Dirichlet problem on uniformly convex domains, in $d=2$ employing also a subsequent regularity result for the effective boundary data by Shen and Zhuge \cite{ShenZhugeRegularity}. The case of the analogous \textit{oscillating Neumann problem} has recently been treated by Shen and Zhuge \cite{ShenZhugeNeumann}. 

In principle, similar issues arise in the homogenization of random media: For stationary and suitably decorrelating random coefficient fields $a$ with $a=a^\transpose$, an analogue of the second-order error estimate \eqref{ErrorBoundFullSpace} has been derived in \cite{BellaFehrmanFischerOtto}. More precisely, in the case of such random coefficient fields, solutions to the linear elliptic PDE \eqref{LinearElliptic} on $\mathbb{R}^d$ for smooth $f$ may be approximated by a constant coefficient effective PDE $-\nabla\cdot (\bar a \nabla \bar u)=\nabla \cdot f$ up to an error of order\footnote{\,Note that for $d=3$ and $d=4$ this error estimate is of fluctuation order and therefore optimal.}
\begin{align*}
||u_\eps-\bar u||_{H^{-1}(B_1)}
\leq \mathcal{C}
\begin{cases}
\eps^{3/2}&\text{ for }d=3,
\\
\eps^{2} |\log \eps|&\text{ for }d=4,
\\
\eps^{2}&\text{ for }d\geq 5.
\end{cases}
\end{align*}
Again, in the case of the Dirichlet problem on bounded domains no analogue of this higher-order convergence result holds and the best possible homogenization error estimate for the Dirichlet problem is $||u_\eps-\bar u||_{L^p (\domain)}\leq \mathcal{C} \eps$ (for $d\geq 3$) --which may be obtained via duality methods (see, \eg, \cite{Suslina}).

While the periodic homogenization problem with fluctuating Dirichlet or Neumann boundary data is by now well-studied, in the case of stochastic homogenization only few results on the boundary layer corrector have been available in the literature. For instance, in \cite{FischerRaithel} boundary layer correctors have been constructed for half-spaces to develop a $C^{1,\alpha}$ large-scale regularity theory; however, the available estimates on the boundary layer corrector are far from optimal. Preceding this result, a Lipschitz regularity theory for random elliptic operators on general domains had been developed in \cite{ArmstrongMourrat,ArmstrongSmart}, having been achieved without an explicit construction of the boundary layer corrector.
In the recent work \cite{WangXu}, a $C^{1,\alpha}$ boundary regularity theory and a Green's function expansion for random elliptic operators on bounded sufficiently smooth domains are provided.
In \cite{JosienRaithelSchaeffner}, correctors for harmonic functions around lower-dimensional features like cusps have been studied. For correctors for interface problems, we refer to \cite{JosienRaithel} in the random setting --while interface problems in periodic homogenization are treated in \cite{BlancLeBrisLions, Josien, Zhuge}.

\smallskip\smallskip\smallskip
\noindent
\textbf{Quantitative estimates on boundary layers in stochastic homogenization.}
In the present work, considering the model case of $\domain = \Hd$ ($d \geq 3$), we establish quantitative estimates on the decay of the boundary layer corrector $\theta^{\varepsilon}_i$, given as the solution to the PDE
\begin{equation}
\label{defn_theta}
\begin{aligned}
-\nabla \cdot (a\big(\tfrac{\cdot}{\eps}\big) \nabla \theta^{\varepsilon}_i) &= 0 && \text{in }\quad \, \Hd, 
\\
\theta^{\varepsilon}_i &= \phi_i\big(\tfrac{\cdot}{\eps}\big) && \text{on } \quad \partial \Hd,
\end{aligned}
\end{equation}
where $\Hd := \left\{ x \in \mathbb{R}^d \, : \, x_1 > 0  \right\}$. More precisely, we consider an ensemble $\langle\cdot\rangle$ of coefficient fields $a$ on $\mathbb{R}^d$ -- i.e., a probability distribution on the space of coefficient fields on $\mathbb{R}^d$-- that is subject to a standard set of conditions of quantitative stochastic homogenization: We require uniform ellipticity and boundedness, stationarity (invariance of the law under spatial shifts), and quantitative decorrelation on scales larger than a microscale $\varepsilon>0$ in form of a spectral gap inequality. Under an additional small-scale regularity condition, we show that the boundary layer corrector decays as
\begin{align}
\label{GradThetaBound}
\big|\varepsilon\nabla \theta^{\varepsilon}_i(x)\big| \leq \mathcal{C}(a,x) \bigg(\frac{\varepsilon}{\dist(x,\partial\Hd)}\bigg)^{d/2-\delta} 
\quad \text{for any } \delta>0 \text{ and } x \in \mathbb{R}^d,
\end{align}
with a random field prefactor $\mathcal{C}(a,x)$ that is subject to uniform stretched exponential moment bounds. For the expected value $\mathbb{E}[\nabla \theta_i]$,
our arguments morally establish a higher-order decay
%we prove the higher-order decay
\begin{align*}
\big| \mathbb{E}[ \varepsilon \nabla  \theta^{\varepsilon}_i(x)]\big| \leq C \bigg(\frac{\varepsilon}{\dist(x,\partial\Hd)}\bigg)^{d/2+1-\delta},
\end{align*}
although we make this rigorous only in a spatially averaged sense. %(Proposition~\ref{Systematic_error})
To the best of our knowledge, the estimate \eqref{GradThetaBound} is the first optimal-order estimate on the decay of boundary layers in stochastic homogenization.

\smallskip\smallskip\smallskip
\noindent
\textbf{Application to the representative volume element approximation.}
An interesting application of these estimates for boundary layers arises in the analysis of the representative volume element (RVE) method for the approximation of the effective coefficient $\bar a$. The RVE method proceeds by taking a sample of the random coefficient field $a$ --say, on a box $[0,L]^d$ with side length $L\gg \eps$-- and approximates the effective coefficient $\bar a$ as $\bar a \approx a^\RVE$ with
\begin{align*}
a^\RVE e_i := \fint_{[0,L]^d} a^{\eps} (e_i+\nabla \phi_i^{\eps,L}) \dx,
\end{align*}
where we have introduced the notation $a^{\eps}(\cdot) = a(\cdot / \eps)$. Here, $\phi_i^{\eps, L}$ solves the corrector equation $-\nabla \cdot (a^{\eps} (e_i+\nabla \phi_i^{\eps,L}))=0$ on the representative volume element $[0,L]^d$, for instance with homogeneous Dirichlet boundary conditions.
It turns out that this RVE approximation with Dirichlet boundary data features random fluctuations of order $|a^\RVE-\mathbb{E}[a^\RVE]|\sim (\tfrac{\eps}{L})^{d/2}$ and a systematic error of order $|\mathbb{E}[a^\RVE]-\bar a|\sim \tfrac{\eps}{L}$. The fluctuations are of central limit theorem scaling and therefore of optimal order. Note, however, that for $d\geq 3$ the systematic error dominates (see, e.\,g., \cite{RavichandranNguyenSchneider} for large-scale numerical simulations) --the reason for this being a rather large contribution to the systematic error coming from a boundary layer phenomenon. Naively one might expect that by averaging only in the interior (``oversampling''), \ie  by setting
\begin{align}
\label{ModificationRVE}
{\tilde a}^\RVE e_i := \fint_{[\kappa L,(1-\kappa)L]^d} a^{\eps}(e_i+\nabla \phi_i^{\eps,L}) \dx
\end{align}
for some constant $0<\kappa\leq \tfrac{1}{4}$, one may eliminate this issue of boundary layers. Interestingly, this turns out \emph{not} to be the case: While the oversampling decreases the systematic error by a constant factor, the order of convergence remains the same. The underlying reason is that the Dirichlet corrector $\phi_i^{\eps,L}$ differs from the whole-space corrector $\phi_i^{\eps}$ (the latter of which defines the effective coefficient as $\bar ae_i := \mathbb{E}[a^{\eps}(e_i+\nabla \phi^{\eps}_i)]$) in terms of a boundary contribution: Indeed, $\phi^{\eps}_i-\phi_i^{\eps,L}$ is an $a$-harmonic function in $[0,L]^d$ with boundary data $\phi_i^{\eps}$ --an oscillating Dirichlet problem. The failure of higher-order convergence of \eqref{ModificationRVE} is due to \emph{effective boundary conditions} imposed on $\phi^{\eps}_i-\phi_i^{\eps,L}$ due to the interaction of the \emph{fluctuating boundary data} $\phi^{\eps}_i$ with the fluctuations of the coefficient field $a^{\eps}$. To correct for these effective boundary conditions, in the numerical literature (see e.\,g.\ \cite[(2.14) \& (2.22)]{yue_local_2007}, or \cite[(1.11)]{LeBrisEtAl} for a related formula) the formula
\begin{align}
\label{ImprovedRVEFormula}
a^{\RVE,new} \fint_{[\kappa L,(1-\kappa)L]^d} e_i+\nabla \phi_i^{\eps,L} \dx := \fint_{[\kappa L,(1-\kappa)L]^d} a^{\eps}(e_i+\nabla \phi_i^{\eps, L}) \dx
\end{align}
has been proposed.
In the context of elasticity, the analogue of this condition is known in the mechanics literature as the Hill-Mandel condition.
The formula \eqref{ImprovedRVEFormula} is rather natural, as the effective coefficient should provide the relation between a given macroscopic field gradient and the corresponding macroscopically averaged flux.
However, so far a rigorous justification of the use of the Hill-Mandel-type condition \eqref{ImprovedRVEFormula} in the context of oversampled RVE approaches has been lacking.
We show that \eqref{ImprovedRVEFormula}, with overwhelming probability, improves the overall error estimate of the RVE approximation of $\bar a$ to $(\tfrac{\eps}{L})^{d/2-\delta}$ for $d=3, 4$. The underlying reason is that this formula accounts for the field gradient in the interior that is caused by the effective boundary conditions of order $\eps$ of $\phi^{\eps}_i-\phi_i^{\eps,L}$, which is in general of order $\fint_{[\kappa L,(1-\kappa)L]^d} \nabla \mathbb{E}[\phi_i^{\eps,L}] \dx \sim \tfrac{\eps}{(1-2\kappa)L}$ and hence stronger than the fluctuations of order $O((\eps/L)^{d/2})$.

\smallskip\smallskip\smallskip
\noindent
\textbf{Notation.}
Throughout the paper, to simplify notation and without loss of generality we mostly consider the half-space of the form $\Hd := \{ x \in \Rd^d : x_1 >0 \}$. Thus, the natural (interior) unit normal vector is $e_1:=(1,0,\ldots,0)$. 
For points $x \in \Rd^d$ we decompose their coordinate vector $x$ as $x:=(x^\perp,x^{\parallel})$ with $x^\perp \in \Rd$ and $x^\parallel \in \Rd^{d-1}$. Occasionally, we denote a fixed point as $x_0 \in \Hd$ --this should not be confused with the notation for the coordinates $x_i: = x \cdot e_i$ for $i = 1, \ldots, d$. 

By $B_r(x)$ we denote the ball of radius $r$ centered at $x$ and $B_r^+(x) := B_r(x) \cap \Hd$. When the ball is centered at $0$ we simply write $B_r$ and $B_r^+$ respectively. By $\mathcal{O}$ we denote a (typically $C^{1,1}$, but not necessarily bounded) domain. 

By $c$ and $C$, we denote generic constants whose value may change from line to line; typically, $C$ is used to denote a large constant, while $c$ denotes a small constant.

Similarly, we use the notation $\mathcal{C}$ to denote a generic random constant (respectively, $\mathcal{C}(x)$ to denote a generic random field) whose value may change from line to line. Unless otherwise specified, the constant $\mathcal{C}$ (respectively the random field $\mathcal{C}(x)$) will be subject to a uniform stretched exponential moment bound. We remark that occasionally we emphasize the randomness and write $\mathcal{C}(a)$ or $\mathcal{C}(a,x)$.

We frequently use the convention that $``\lesssim"$ means $`` \leq C(d, \lambda, \nu, \alpha),"$ where $d$ denotes the dimension, $\lambda>0$ is the ellipticity ratio of the coefficient field $a$, and $\nu$ and $\alpha$ are from the H\"older regularity assumption (A4) below. If a universal constant has additional dependencies that we would like to emphasize, e.\,g.\ on $\delta$, then we write $``\lesssim_{\delta}"$. In general, in the statements of results we emphasize the full list of dependencies of a universal constant by writing, \eg $``\lesssim_{d, \lambda, \nu, \alpha, \delta}"$, whereas in arguments we allow ourselves not to.

Throughout the paper, we consider an ensemble $\langle \cdot \rangle$ (a probability distribution) of uniformly elliptic and bounded coefficient fields $a$; we denote the corresponding probability space as $\Omega$ and the associated expected value of a random variable $G$ by $\langle G \rangle$. Regarding terminology, a \emph{quenched} estimate is a bound that holds \as, while an \emph{annealed} estimate is a bound on suitable stochastic moments.

We use standard notation for Sobolev spaces, with the slight deviation 
\begin{align*}
H^1_{\loc}(\domain):= \left\{v : v \in H^1(B_r(x)\cap\domain) \text{ for all } r>0, x \in \domain  \right\},
\end{align*}
for a domain $\domain\subseteq \R^d$. As usual, by $\dot H^1(\domain)$ we denote the closure of the set of all $C^1$ functions with bounded support with the norm $||u||_{\dot H^1(\domain)}:=(\int_\domain |\nabla u|^2 \dx)^{1/2}$.

In some sections of this paper we are working on an arbitrary $C^{1,1}$-domain, meaning that we cannot simply set $\eps =1$ and argue by scaling. We emphasize the following non-uniform choice of notation: The boundary layer corrector $\theta^{\eps}$, defined by \eqref{defn_theta}, has boundary data of order $1$, namely $\phi \big(\tfrac{\cdot}{\eps}\big)$, while the whole-space corrector $\phi^\eps(\cdot):=\eps \phi\big(\tfrac{\cdot}{\eps}\big)$ has boundary data of order $\eps$. In other words, a corrector with Dirichlet boundary data is obtained from the whole-space corrector $\phi^\eps$ and the boundary layer corrector $\theta^\eps$ as $\phi^\eps-\eps \theta^\eps$.

\section{Main results and strategy}

To begin this section, in Section \ref{setting_discussion} we list our assumptions on the ensemble $\langle\cdot\rangle$ of coefficient fields $a$ on $\mathbb{R}^d$. We then discuss a class of ensembles on $\Rd^d$ satisfying our assumptions. In Section~\ref{Thm1_discussion}, we state our main result on the decay of the boundary layer (Theorem~\ref{thm_1}). Section~\ref{Thm2_discussion} contains our main result  on (nearly) optimal convergence rates for a modified RVE approximation in $d =3, 4$ (Theorem~\ref{thm_2}).

\subsection{Setting of our results}
\label{setting_discussion}

Throughout this article we will make the following assumptions:

\paragraph{\textit{Assumptions on the ensemble.}}

\begin{itemize}
\item[(A1)] The ensemble $\langle\cdot\rangle$ is \emph{uniformly elliptic and bounded}: There exists $\lambda>0$ such that $\langle\cdot\rangle$-\as we have
\begin{align*}
|a(x)v|&\leq |v|,
\\
a(x)v \cdot v &\geq \lambda |v|^2
\end{align*}
for all $v\in \mathbb{R}^d$ and a.\,e.\ $x\in \mathbb{R}^d$.
\item[(A2)] The ensemble $\langle\cdot\rangle$ is \emph{stationary}, that is, the law of $a(\cdot +x)$ coincides with the law of $a(\cdot)$ for any $x\in \mathbb{R}^d$. In other words, the statistics of the coefficient field do not depend on the region in space.
\item[(A3)] The ensemble is subject to decorrelation on scales larger than the microscale in the sense that it satisfies a spectral gap inequality: For any random variable $\xi(a)$ it holds that 
\begin{align}
\label{spectral_gap_L2}
\en{ (\xi - \en{\xi} )^2} \leq \en{ \int_{\Rd^d}  \bigg( \fint_{B_{\eps}(x)} \bigg| \frac{\partial \xi}{ \partial a}  \bigg| \dd y \bigg)^2  \dx},
\end{align}
where $\fint_{B_{\varepsilon}(x)} \big| \frac{\partial \xi}{ \partial a}  \big| \dd y$  is used to denote 
\begin{align}
\label{deriv}
\sup_{\delta a}\limsup_{h \rightarrow 0 } \frac{\xi(a + h \delta a ) - \xi(a)}{h} 
\end{align}
with the supremum taken over perturbations $\delta a: \Rd^d \rightarrow \mathbb{R}^{d \times d}$ that are supported in $B_{1}(x)$ and satisfy $\| \delta a \|_{L^{\infty}(\Rd^d)} \leq 1$. When $\xi$ is Fr\'echet differentiable we remark that \eqref{deriv} is just $\big|\frac{\partial \xi}{\partial a(x)}\big|$.
\item[(A4)] The coefficient field $a$ is H\"older continuous with stretched exponential moments for some H\"older exponent $\alpha \in (0,1]$. In particular, there exists a stationary random field $\mathcal{C}(a,x)$ with  $\mathbb{E}[\exp(\nu \mathcal{C}^\nu)]\leq 2$ for some $\nu >0$ such that the H\"older estimate
\begin{align}
\label{holder_norm}
\sup_{y,z\in B_{1}(x)}  \frac{|a(y)-a(z)|}{|y-z|^\alpha} \leq \mathcal{C}(a,x)
\end{align}
holds.
\end{itemize}
Note that (A1)--(A3) are a standard set of assumptions in quantitative stochastic homogenization; in particular, assumption (A3) is a standard quantification of the qualitative ergodicity assumption that is used in classical qualitative stochastic homogenization results \cite{naddafspencer98,GO1}.

Assumption (A4) is a small-scale regularity assumption that we require mostly for convenience; it enables us to state pointwise gradient bounds (which may otherwise fail due to a possible failure of regularity on scales smaller than the microscale).

\bigskip

We now give an example of a class of coefficient fields that is subject to our assumptions (A1)--(A4), see Example~1 below. In Example~2, we then then introduce another class of coefficient fields that is subject to a spectral gap estimate similar to (A2). Note that both examples are standard examples used to illustrate the applicability of quantitative stochastic homogenization results.

\medskip

\noindent\textbf{\textit{Examples of ensembles satisfying (A1) - (A4).}} 

\medskip

\noindent \textbf{Example 1:} We consider a Gaussian ensemble that satisfies a certain decorrelation estimate. In particular, let $\tilde{a}(x)$ be a scalar Gaussian field that is centered and stationary and, furthermore, satisfies a decorrelation estimate of the form 
\begin{align*}
\int_{\Rd^d} \sup_{|x| = |y|} \big| \en{\tilde{a}(x) \tilde{a}(0) }\big| \, \dd y \leq 1.
\end{align*}
Defining $a:=g(\tilde{a}(x))$, where $g$ is a bounded Lipschitz map from $\Rd$ into the subset of $\Rd^{d \times d}$ consisting of the uniformly elliptic matrices with ellipticity ratio $\lambda>0$, it is shown that (A1)--(A3) are satisfied, \eg,  in \cite{DuerinckxGloriaSpectralGap}. For Gaussian random fields $\tilde a(x)$ whose power spectrum decays sufficiently quickly, the random fields $\tilde a$ are also \as in $C^1_{loc}$ and thus also satisfy (A4) (see \cite[Lemma 3.1]{JosienOtto_2019}).

\bigskip

\noindent \textbf{Example 2:} The second example involves random spherical inclusions. More precisely, as in \cite{Torquato} we consider a stationary Poisson point process on $\mathbb{R}^{d}\times [0,\infty)$. Let $(X_i,h_i)_{i\in \mathbb{N}}$ be a realization; then consider the points $X_i\in \mathbb{R}^d$ in the order of their coordinate $h_i$ and for each point $X_i$ place a ball of radius $1$ around $X_i$ if the ball does not intersect with any of the previously placed balls. Denote the union of the balls obtained by this procedure by $U$. Then define a coefficient field as $a(x):=\lambda \chi_{\mathbb{R}^d\setminus U}+\chi_U$. While such ensembles of coefficient fields consisting of random high-conductivity inclusions do not satisfy the spectral gap inequality (A3), they have been shown to satisfy a slightly different version of a spectral gap inequality in \cite{DuerinckxGloriaSpectralGap}. We expect that a relatively straightforward adaption of our arguments would make our results applicable also to ensembles satisfying such an ``oscillation spectral gap'' estimate, though it would involve some additional work to obtain the necessary small-scale (microscopic) regularity theory (as derived from assumption (A4) in our present work).

\subsection{Theorem \ref{thm_1}: Optimal decay estimate for the boundary layer}

\label{Thm1_discussion}

In our first main result we show that if the ensemble $\langle \cdot \rangle$ of coefficient fields on $\Rd^d$ satisfies (A1)-(A4), then we are able to obtain an (almost) optimal decay estimate for the gradient of the boundary layer $|\nabla \theta^{\eps}|$ away from the boundary $\partial \Hd$. We remark that this estimate is \textit{quenched}, in the sense that it is stated for \aen coefficient field $a$ and includes a random constant. Here is the statement:

\begin{theorem} \label{thm_1} Let $d\geq 3$, $i=1, \ldots, d$, and $\langle\cdot\rangle$ be an ensemble of coefficient fields on $\Rd^d$ that satisfies the assumptions (A1)-(A4).
Denote by $\theta^{\eps}_i \in H^1_{\loc}(\Hd ; \mathbb{R}^d)$ the half-space boundary layer corrector, that is, the unique\footnote{\,Uniqueness follows from the Liouville principle for random elliptic operators on half-spaces proven in \cite{FischerRaithel}.} weak solution to \eqref{defn_theta} with the property that $\nabla \theta^{\eps}_i$ decays away from $\partial \Hd$ --in the sense that
\begin{align*}
\lim_{r\rightarrow \infty} \fint_{B^+_r} |\nabla \theta^{\eps}_i|^2 \dd x =0.
\end{align*}
Let $0< \delta \ll 1$.
Then there exists a random field $\mathcal{C}(a,x)$ that is stationary with respect to shifts tangential to the boundary and has stretched exponential moments --in the sense that there exists $C = C(d,\lambda, \nu, \alpha,\delta)$ with
\begin{align}
\label{stretched}
\en{\exp\bigg(\frac{\mathcal{C}(a,x)^{1/C}}{C}\bigg)}\leq 2
\end{align}
for any $x\in \Hd$-- such that the gradient of the boundary layer corrector satisfies the decay estimate
\begin{align}
\label{main_goal}
|\eps \nabla \theta^{\eps}_i (x)|  \lesssim_{d, \lambda, \nu, \alpha, \delta}  \mathcal{C}(a, x) \bigg(1 + \frac{x^{\perp}}{\varepsilon} \bigg)^{-\frac{d}{2} + \delta}
\end{align}
for any $x \in \Hd$.  
%\textcolor{red}{Additionally, in expectation the gradient of the boundary layer corrector satisfies higher order decay
%\begin{align*}
%\left| \langle \nabla \theta_i (x) \rangle \right| \lesssim_{d, \lambda \alpha, \nu, \delta} \left( 1 + \frac{x^{\perp}}{\varepsilon} \right)^{-\frac{d}{2} -1 + \delta}
%\end{align*}}
\end{theorem}

The decay of $|\eps \nabla \theta^{\eps}_i|$ in \eqref{main_goal} is, up to the $\delta$-loss, of the optimal scaling that may be expected. This can be seen in the following heuristic argument: Fixing $x_0 \in \Hd$ and $e\in \R^d$, let $h$ solve the divergence form equation 
\begin{equation}
\label{h_heur}
\begin{aligned}
-\nabla \cdot (a^\transpose\big(\tfrac{\cdot}{\eps}\big) \nabla h^{\eps}) &= \nabla \cdot (\delta_{x_0} e)&&\text{in }\Hd,
\\
h^{\eps}& =0&&\text{on }\partial \Hd,
\end{aligned}
\end{equation}
and rewrite $\eps \nabla \theta^{\eps}_i(x_0)$ as 
\begin{align}
\label{heuristics_1}
e \cdot \eps \nabla \theta^{\eps}_i(x_0) =  \int_{\Hd}  - \nabla \cdot (\delta_{x_0} e ) \eps \theta^{\eps}_i \, \dx   = - \int_{\partial \Hd} \phi^{\eps}_i e_1 \cdot a^\transpose\big(\tfrac{\cdot}{\eps}\big) \nabla h^{\eps}  \, \dd S.
\end{align}
Notice that to obtain the above relation we have used the equation \eqref{defn_theta}, that $h^{\eps} = 0$ on $\partial \Hd$, and the identity $\eps \theta^{\eps}_i = \phi^{\eps}_i$ on $\partial \Hd$. Now, we expect the right-hand side of \eqref{heuristics_1} to have the same asymptotic scaling behaviour as the corresponding quantity with the Green's function $h$ replaced by its homogenized counterpart $\bar h$ solving $-\nabla \cdot (\bar a^\transpose \nabla \bar h) = \nabla \cdot (\delta_{x_0} e)$ in $\Hd$. Using that $|\nabla \bar{h} (\cdot)| \leq C |\cdot - x_0|^{-d}$, we obtain
\begin{align} 
\label{heuristics_2}
 \int_{\partial \Hd} \phi^{\eps}_i e_1 \cdot \bar{a}^\transpose \nabla \bar{h}  \, \dd S  \sim (x_0^\perp)^{-1} \fint_{B_{x_0^{\perp}}(x_0^\parallel)} \phi^{\eps}_i  \, \dd S ,
\end{align}
where we have used $B_{x_0^\perp}(x_0^\parallel)$ to denote the $d-1$ dimensional ball with radius $x_0^\perp$ centered around $x_0^\parallel$ in $\partial \Hd$. Since under the assumptions (A1) - (A3) volume averages of the gradient of the whole-space corrector display central limit theorem scaling \cite{MourratOtto}, one expects that
\begin{align*}
\fint_{B_{x_0^{\perp}}(x_0^\parallel)}  \phi_i^{\eps}  \, \dd S \sim |x_0^\perp| \bigg( \frac{|x_0^\perp|}{\varepsilon} \bigg)^{-\frac{d}{2}}.
\end{align*}
Combining this observation with \eqref{heuristics_1} and \eqref{heuristics_2} suggests that \eqref{main_goal} cannot be improved --aside from the $\delta$-loss.

\subsection{Theorem \ref{thm_2}: Improved error estimate for the representative volume element method.}
\label{Thm2_discussion}

In our second main result we use Theorem \ref{thm_1} to obtain an improved error estimate for the RVE method with oversampling (see \eqref{ImprovedRVEFormula}). As already described above, the error in the RVE method can be split into two contributions: the \textit{random} error and the \textit{systematic} error. While the former is well-understood to be of central limit theorem scaling, without additional modifications such as ``screening'' terms in the corrector equation \cite{LeBrisEtAl,GloriaRVE, GO_RVE} the systematic error would dominate in $d\geq 3$ dimensions. Our understanding of boundary layers enables us to show that a minor modification (proposed already in the numerics literature \cite{yue_local_2007} and being an analogue of the Hill-Mandel condition in mechanics) improves the bound on the systematic error, thereby improving the overall rate of convergence for $d \geq 3$. Our result reads as follows:

\begin{theorem}
\label{thm_2}
Let $d\geq 3$ and let $\langle\cdot\rangle$ be an ensemble of coefficient fields on $\Rd^d$ that satisfies the assumptions (A1)-(A4). Let $\bar a\in \mathbb{R}^{d\times d}$ be the corresponding effective coefficient.
Denote by $\phi_i^{\eps,L}$ the homogenization corrector with homogeneous Dirichlet boundary conditions on the box $[0,L]^d$, \ie  let $\phi_i^{\eps,L}$ satisfy
\begin{equation}
\label{box_corrector}
\begin{aligned}
-\nabla \cdot (a^{\eps}(e_i+\nabla \phi_i^{\eps,L})) &=0 &&\text{in }(0,L)^d,
\\
\phi_i^{\eps,L} &=0 &&\text{on }\partial [0,L]^d.
\end{aligned}
\end{equation}
%Then the following assertions hold true:
%\begin{itemize}
%\item[a)]
Given any oversampling parameter $\kappa \in (0,\tfrac{1}{4}]$, define the RVE approximation $\bar a^L\in \mathbb{R}^{d\times d}$ for the effective coefficient as the solution to the system of linear equations
\begin{align}
\label{DefBarA}
\bar a^L \fint_{[\kappa L,(1-\kappa)L]^d} e_i+\nabla \phi_i^{\eps,L} \dx
=\fint_{[\kappa L,(1-\kappa)L]^d} a^{\eps}(e_i+\nabla \phi_i^{\eps,L}) \dx,
\end{align}
for all $1\leq i\leq d$. Then, for any $\delta>0$ and any $0<s\leq c(d,\lambda,\alpha,\nu) (\tfrac{L}{\eps})^{\min\{d/2,2\}-\delta}$, the error estimate
\begin{align*}
\mathbb{P}\big[\big|\bar a^L-\bar a\big|>s \big(\tfrac{\eps}{L}\big)^{\min\{d/2,2\}-\delta} \big] \leq 2\exp(-s^{1/C}/C)
\end{align*}
holds with $C=C(d,\lambda,\alpha,\nu,\delta)$. In particular, with at least this probability $\bar a^L$ is well-defined, \ie the defining linear system has a unique solution.
\end{theorem}
\noindent Overall, for $d=3,4$ we see that with overwhelming probability, the error $|\bar a^L-\bar a|$ is at most of order $(\tfrac{\eps}{L})^{d/2-\delta}$, improving the error estimate of the standard RVE approximation to (almost) the fluctuation scaling. This behavior is confirmed in large-scale numerical simulations in the upcoming work \cite{SentlerSchneiderBella}.

We remark that as a slight downside of the formula \eqref{DefBarA}, in principle we cannot rule out that there may be a small probability that the linear system for $\bar a^L$ given by \eqref{DefBarA} may fail to be invertible; thus, we cannot make a statement on the error of the expected value $\mathbb{E}[\bar a^L]$. However, the probability for this happening is exceedingly low, as shown by our stretched exponential estimate.

\subsection{Strategy}

As shown in the companion paper \cite{BFJR_1}, deterministic estimates on the boundary layer corrector $\theta_i^\eps$ --based on estimates on its boundary data $\eps^{-1} \phi_i^\eps$ in combination with large-scale regularity theory-- yield a suboptimal decay estimate of the form
\begin{align*}
|\eps \nabla \theta_i^\eps(x)| \leq \mathcal{C} \cdot \big(\tfrac{\eps}{\eps+\dist(x,\partial\Hd)}\big)^{-1/3}.
\end{align*}

Following a now-standard strategy in quantitative stochastic homogenization \cite{naddafspencer98,GO1,GNO5}, the fluctuations of the boundary layer corrector $\nabla \theta_i^\eps(x)-\mathbb{E}[\nabla \theta_i^\eps(x)]$ may (in suitable negative Sobolev norms) be estimated by means of the spectral gap inequality. This requires a bound on the sensitivity (Fr\'echet derivative) of $\nabla \theta_i^\eps$ with respect to changes in the coefficient field, which we obtain as usual by combining (large-scale) regularity results for random elliptic operators with the estimates on the corrector gradient. The required regularity results for random elliptic operators in our setting are established in the companion paper \cite{BFJR_1}. Overall, we achieve estimates on (linear functionals of) $\nabla \theta_i^\eps(x)-\mathbb{E}[\nabla \theta_i^\eps(x)]$ of nearly fluctuation order $\big(\tfrac{\eps}{\dist(x,\partial\Hd)})^{-d/2}$, see Proposition~\ref{sens_Prop1}.

The key step to proving Theorem~\ref{thm_1} is the derivation of an estimate on $\mathbb{E}[\nabla \theta_i^\eps(x)]$. Based on a two-scale expansion argument and duality, Proposition~\ref{Systematic_error} shows that the expected value $\mathbb{E}[\nabla \theta_i^\eps(x)]$ decays one order faster than the best available estimate for $\mathbb{E}[|\nabla \theta_i^\eps(x)|^2]^{1/2}$. In conjunction with the fluctuation bounds, this enables an iteration argument to obtain Theorem~\ref{thm_1}.

Our result on the RVE approximation of effective coefficients (Theorem~\ref{thm_2}) relies on a precise description of the impact of boundary layer effects on the solution in the interior of the RVE $(0,L)^d$. This description is achieved by harnessing our estimates on the half-space boundary layer corrector from Theorem~\ref{thm_1}, approximating each of the $2\cdot d$ faces of the RVE $(0,L)^d$ by a half-space and estimating the error incurred in this procedure.

We remark that many of our arguments use quantitative stochastic homogenization results that have been established in the whole-space setting. These results, in particular, include \cite[Theorem 1 \& Theorem 2]{GNO_final} and \cite[Theorem 2]{GNO5}, which we have summarized in an appendix for the convenience of the reader.

\section{Arguments for our main results}

\subsection{Proof of Theorem \ref{thm_1}} 

\label{overview_thm1}

Recall that a random field $\C(a,x)$ having stretched exponential moments is equivalent to the existence of constants $m,C>0$ such that $ \en{ |\C(a, x)|^q }^{1/q}  \leq C q^m$ holds for any $q\geq 1$ and $x \in \Hd$.  (By Jensen's inequality it is, in fact, sufficient to obtain the previous moment bound on $\C(a,x)$ only for $q\geq  q_0$ for some $q_0\geq1$.) For \eqref{main_goal} it is, therefore, sufficient to show that there exist a constant $m>0$ and $q_0 \geq 1$ such that
\begin{align}
\label{main_goal_2}
\en{ |\eps \nabla \theta^{\eps}_i (x)|^q }^{\frac{1}{q}}  \lesssim_{\delta} q^m \left(1 + \frac{x^{\perp}}{\eps} \right)^{-\frac{d}{2} +\delta},
\end{align}
for any $x \in \Hd$ and $q \geq q_0$.
%Notice that, while Theorem \ref{thm_1} is a \textit{quenched} result, the estimate \eqref{main_goal_2} is \textit{annealed} (\ie averaged over the probability space), which is more convenient for using the spectral gap inequality \eqref{spectral_gap_L2} from (A3).
Our argument for \eqref{main_goal_2} comes in three main steps:\\
%\red{(C: We say that we set $\varepsilon =1$ and then don't anywhere except in (3.1)--So, I've commented this out and added the epsilon into (3.1).)} 
%We also remark that, since in this contribution Theorem \ref{thm_1} is only stated for $\Hd$ (rather than for a general domain), we may w.l.o.g. set $\varepsilon =1$ --since the result for a general correlation length follows from rescaling $a$.\\

\medskip

\noindent \textbf{Step 1:} We recall the following suboptimal decay of $|\eps\nabla \theta^{\eps}_i(x)|$ away from $\partial \Hd$:
\begin{proposition}[{Proposition~2 of \cite{BFJR_1}}]
\label{intermediate_lemma}
Let $d\geq 3$, $i\in \{1, \ldots, d\}$, and let $\langle\cdot\rangle$ be an ensemble of coefficient fields on $\Rd^d$ that satisfies the assumptions (A1)-(A4). Let $\domain$ be either a bounded $C^{1,1}$-domain or let $\domain=\Hd$. Then there exists a random field $\C(a,x)$ with stretched exponential moments in the sense of \eqref{stretched}, such that 
\begin{align}
\label{intermediate}
|\eps \nabla \theta^{\eps}_i (x)|   \lesssim_{d, \lambda, \nu, \alpha} \mathcal{C}(a,x)\bigg(1 + \frac{\dist(x,\partial\domain)}{\varepsilon}\bigg)^{-\frac{1}{3}}
\end{align}
holds for all $x \in \domain$.
\end{proposition}

\noindent This result is applied for $\domain = \Hd$ and slightly post-processed to obtain a corresponding suboptimal decay estimate for moments of a spatial average of $|\eps \nabla \theta^{\eps}_i|^2$ (around $x$).\\

\medskip

\noindent \textbf{Step 2:} Taking the post-processed result of Step 1 as input, we show that the obtained  annealed decay estimates on moments of a spatial average of $|\eps \nabla \theta^{\eps}_i|^2$ (around $x$) are self-improving --saturating at the desired decay $-d/2 + \delta$ in Theorem~\ref{thm_1}. This iterative self-improvement is shown in Proposition \ref{iterative_lemma}.

We remark that Proposition \ref{iterative_lemma} is the workhorse of our strategy: For notational convenience we now fix $x_0 \in \Hd$. To prove Proposition \ref{iterative_lemma} we control stochastic moments of 
\begin{align*}
\fint_{B_{x_0^{\perp}/8}(x_0)} |\eps\nabla \theta^{\eps}_i|^2 \dx
\end{align*}
in terms of moments of suitable weighted averages $F^{\eps}:=\int_{\Hd} g \cdot \eps \nabla \theta^{\eps}_i \dx$ --where the weights $g$ are, in some sense, localized around $x_0$. To control the moments of the $F^{\eps}$, we control the moments of the fluctuations $F^{\eps}-\langle F^{\eps} \rangle$ (Proposition \ref{sens_Prop1}) and the expected value $\langle F^{\eps} \rangle$ (Proposition \ref{Systematic_error}). Proposition \ref{sens_Prop1} relies on the spectral gap inequality (A3) and sensitivity estimates --the arguments are contained in Section \ref{sens_analysis}. The argument for Proposition \ref{Systematic_error} is contained in Section \ref{expectation} --it relies on a careful two-scale expansion argument, which exploits the stationarity of $\theta^{\eps}_i$ and the standard whole-space correctors with respect to shifts tangential to $\partial \Hd$.

Initializing the use of Proposition \ref{iterative_lemma} with the result of Step 1, we establish a weakened version of \eqref{main_goal_2} (contained in Proposition~\ref{intermediate_2}).

\medskip

\noindent \textbf{Step 3:} To finish, we post-process the result of Proposition~\ref{intermediate_2} using \cite[Theorem 1 \& Theorem 2]{GNO_final} and the small-scale regularity assumption (A4).\\

%Our application of Proposition \ref{iterative_lemma} is initialized by the post-processed version of our result from Step 1 and, as we will see below, is limited by our estimate for the stochastic fluctuations of spatial averages of $\nabla \theta$ (see Proposition \ref{sens_Prop1})
% Finally, we use the whole-space large-scale regularity theory from \cite{GNO_final} and the assumption (A4) on the regularity of the coefficient field $a$ in order to obtain the pointwise estimate \eqref{main_goal_2}.\\

\medskip 

We now take a closer look at the iterative improvement in Step 2. As already mentioned above, the main ingredient for this improvement is given by the following proposition:

\begin{proposition}
\label{iterative_lemma}
We adopt the assumptions of Theorem \ref{thm_1} and furthermore assume that there exist constants $m,n>0$ such that
\begin{align}
\label{assumption_iterative_lemma}
\en{  \bigg(\fint_{B_{\frac{x_0^{\perp}}{2}}(x_0)}|{\eps} \nabla \theta^{\eps}_i|^2 \dx \bigg)^{\frac{q}{2}}  }^{\frac{1}{q}}\lesssim_{d,\lambda, \nu, \alpha} q^m \bigg(1+ \frac{x_0^{\perp}}{\varepsilon} \bigg)^{-n}, 
\end{align}
%\text{ and } 1 \leq i \leq d. 
holds for any $x_0 \in \Hd$ and any $q \geq 1$. Then there exists a constant $\tilde{m}>0$ such that
\begin{align}
\label{result_iterative_lemma}
\en{  \bigg(\fint_{B_{\frac{x_0^{\perp}}{8}}(x_0)}|{\eps}\nabla \theta^{\eps}_i|^2 \dx \bigg)^{\frac{q}{2}}  }^{\frac{1}{q}}\lesssim_{d, \lambda, \nu, \alpha, \delta} q^{\tilde{m}} \bigg(1+ \frac{x_0^{\perp}}{\varepsilon}\bigg)^{-\min\{n+1-\delta,\frac d2 - \delta\}}
\end{align}
holds for any $x_0 \in \Hd$, any $q \geq 1$, and any $0<\delta \ll 1$.
%, and $1 \leq i \leq d$
\end{proposition}

To initiate the iterative application of Proposition \ref{iterative_lemma}, we require access to an estimate of the form \eqref{assumption_iterative_lemma}. For this we post-process \eqref{intermediate} of Proposition \ref{intermediate_lemma} as
\begin{align}
\label{post_processed_intialize}
\begin{split}
\en{  \bigg(\fint_{B_{\frac{x_0^{\perp}}{2}}(x_0)}|\eps\nabla \theta^{\eps}_i|^2  \dx \bigg)^{\frac{q}{2}}  }^{\frac{1}{q}} & \leq  \bigg(\fint_{B_{\frac{x_0^{\perp}}{2}}(x_0)} \en{ |\eps\nabla \theta^{\eps}_i|^q } \dx  \bigg)^{\frac{1}{q}}\\
& \stackrel{\eqref{intermediate}}{\lesssim}   \bigg(1 + \frac{x_0^{\perp}}{\varepsilon}\bigg)^{-\frac{1}{3}} \bigg(\fint_{B_{\frac{x_0^{\perp}}{2}}(x_0)} \en{\C(a,x)^q} \dd x \bigg)^{\frac{1}{q}}
\\
&\lesssim  q^m \bigg(1 + \frac{x_0^{\perp}}{\varepsilon}\bigg)^{-\frac{1}{3}}
\end{split}
\end{align}
for any $x_0 \in \Hd$ and any $q\geq 2$. 

Using \eqref{post_processed_intialize} to initialize the use of Proposition \ref{iterative_lemma} with $n=\frac{1}{3}$, we obtain:

\begin{proposition} 
\label{intermediate_2}
Under the assumptions of Theorem \ref{thm_1}, there exists a constant $m>0$ such that
\begin{align}
\label{step_1}
\en{  \bigg(\fint_{B_{\frac{x_0^{\perp}}{8}}(x_0)}|\eps \nabla \theta^{\eps}_i |^2 \dx \bigg)^{\frac{q}{2}}  }^{\frac{1}{q}}\lesssim_{d, \lambda, \nu, \alpha, \delta} q^m \bigg(1+ \frac{x_0^{\perp}}{\varepsilon}\bigg)^{-\frac{d}{2}+ \delta}
\end{align}
holds for any $x_0 \in \Hd$, any $q \geq 1$, and any $0<\delta \ll 1$.
%, and $1 \leq i \leq d$.
\end{proposition}

\medskip 

With all of these ingredients, the proof of Theorem \ref{thm_1} is quite quick:

\begin{proof}[Proof of Theorem \ref{thm_1}]
By scaling, since we are on the half-space, we may assume that $\eps=1$.
As already discussed in Section \ref{overview_thm1}, it only remains to upgrade \eqref{step_1} of Proposition~\ref{intermediate_2} into the desired \eqref{main_goal}. Throughout this proof, $m >0$ is generic and may change from line to line.
We separately consider the two cases $x_0^{\perp} \leq 8$ and $x_0^{\perp}\geq 8$ --we start with the first case.

For $x_0^{\perp} \leq 8$ we begin by applying the Lipschitz regularity estimate (see, \eg, \cite[Corollary~5]{BFJR_1}) to $\theta_i-\phi_i$ and taking the $q$-th moment of both sides. After an additional application of H\"older's inequality this yields
\begin{align}
		\label{thm_1_schauder}
		&\sup_{x \in B_{1/2}(x_0) \cap \Hd} \en{|\nabla \theta_i(x)|^q}^{\frac{1}{q}}
		\\&
		\nonumber
		\lesssim q^m \bigg[ \en{\bigg( \int_{B_2(x_0) \cap \Hd} (|\nabla \theta_i|^2  + |\nabla \phi_i|^2 )\dx  \bigg)^q}^{\frac{1}{2q}}
		+
		\en{ \| a \|^{2q}_{C^{0, \alpha}(B_{2}(x_0)\cap \Hd)}}^{\frac{1}{2q}}\bigg].
	\end{align}
In this case, Proposition~\ref{intermediate_lemma} bounds the first half of the first term on the right-hand side as
\begin{align*}
\en{\bigg( \int_{B_2(x_0)\cap \Hd} |\nabla \theta_i|^2 \dx  \bigg)^{q}}^{\frac{1}{2q}} &    \lesssim q^m.
\end{align*}
The second half of the first term on the right-hand side of \eqref{thm_1_schauder} is bounded thanks to the Caccioppoli estimate combined with \eqref{GNO_bound_balls}. Last, the second term on the right-hand side of \eqref{thm_1_schauder} is bounded via assumption (A4).
This finishes our treatment of the case $x_0^{\perp} \leq 8$. 

When $x_0^{\perp} \geq 8$ we use the pointwise gradient bound \cite[Corollary~4]{BFJR_1} (which we apply to $\theta_i$ for $r:=x_0^\perp/8$, so that $\theta_i$ is not required to satisfy any boundary conditions) to upgrade \eqref{step_1} of Proposition~\ref{intermediate_2} to
\begin{align*}
\en{|\nabla \theta_i(x_0)|^q}^{1/q} &  \lesssim  \en{\bigg( \fint_{B_{\frac{x_0^{\perp}}{8}}(x_0) } |\nabla \theta_i|^2  \dx\bigg)^{q}}^{\frac{1}{2q}}  \lesssim q^m ( 1 + x_0^{\perp})^{-\frac d2+\delta}.
\end{align*} 
This finishes our argument.
\end{proof}

%\orange{[The change in (3.6) would propagate e.g. into the proof of theorem 1.]}

%\noindent Since the argument for Corollary \ref{intermediate_2} is obvious given the discussion proceeding its statement, we forgo a formal proof.

%\paragraph{\textit{Step 3: Conclusion of the argument for \eqref{main_goal_2}}} Once we have access to the result of Corollary \ref{intermediate_2}, \eqref{main_goal_2} is easily obtained under the assumption (A4) via standard Schauder theory for small scales and the large-scale regularity for $a$-harmonic functions on $\Rd^d$ (see Theorem \ref{summary_GNO} of Appendix \ref{Appendix_A}). This argument is contained in Section \ref{proof_thm_1}.\\

 \subsection{Argument for Proposition \ref{iterative_lemma}: Iterative improvement of decay estimate for the boundary layer}
 \label{proof_Prop_2}
 
To motivate the different sections of this paper, we now immediately give the argument for Proposition \ref{iterative_lemma}. To begin, we give the required auxiliary results --the first of which is an estimate that controls the gradient of an $a$-harmonic function in terms of linear functionals of the gradient, plus a small remainder term.

%a multi-scale inequality, versions of which are standard in quantitative stochastic homogenization. For use in future work, the following lemma is stated for a domain $\domain$ that is taken to be either $\Hd$ or a bounded $C^{1,1}$-domain. For the purpose of proving Proposition \ref{iterative_lemma} only the case $\domain = \Hd$ is required.

\begin{lemma}
\label{FunctionalsToNormLemma}
%\blue{[C: Should I change the universal constant notation here, does this universal constant depend on $\delta$? If yes, this propagates into the proof of Prop 2.]}
Let $a$ be a uniformly elliptic and bounded coefficient field in the sense of assumption (A1). Let $u$ be $a$-harmonic in a ball $B_R (x_0)$, $R\geq 1$ and $x_0 \in \mathbb{R}^d$. Then, for any $\delta>0$ and $A<\infty$, there exists a set of functions $g_{m,k}$, $1\leq m\leq M=M(d,\lambda,\delta,A)$, $1\leq k\leq \lfloor\log_2 \frac{R}{R^{1-\delta}} \rfloor$, subject to the bounds
\begin{align}
|g_{k,m}| &\leq (2^k R^{1-\delta})^{-d}, \label{mult_cond_1}
\\
|\nabla g_{k,m}| &\leq (2^k R^{1-\delta})^{-1-d}, \label{mult_cond_2}
\\
\supp g_{k,m} &\subseteq B_{2^k R^{1-\delta} (x_0)},  \label{mult_cond_3}
\end{align}
and such that the estimate
\begin{align}
\label{BoundForNormByFunctionals}
\fint_{B_{R^{1-\delta}}(x_0)} |\nabla u|^2 \dd x
\leq& R^{-A} \fint_{B_R(x_0)} |\nabla u|^2 \dd x
\\&\nonumber
+C(d, \lambda, \delta, A) \sum_{k=1}^{\lfloor\log_2 \frac{R}{R^{1-\delta}} \rfloor} (2^k)^{-A/\delta}
\sum_{m=1}^M \bigg|\int_{B_R(x_0)} g_{k,m} \cdot \nabla u \dd x\bigg|^2
\end{align}
holds. %\red{C: Dependency of the constants? Same in the proof.}
\end{lemma}
As the proof is rather short, we directly provide it here.

\begin{proof}
We first notice that for any $\tilde \delta>0$ there exists a finite set of functions $g_m$, $1\leq m\leq M$, supported in $B_R(x_0)$ and satisfying $|g_m|\leq R^{-d}$, $|\nabla g_m|\leq R^{-d-1}$, such that any $a$-harmonic function $u$ on $B_R(x_0)$ satisfies
\begin{align}
\label{SingleStep}
\int_{B_{R/2}(x_0)} |\nabla u|^2 \dd x
\leq \tilde \delta \int_{B_R(x_0)} |\nabla u|^2 \dd x
+C(d, \lambda)\sum_{m=1}^M \bigg|\int_{B_R(x_0)} g_{m} \cdot \nabla u \,\dd x\bigg|^2.
\end{align}
In fact, this may be obtained by a suitable variant of, \eg, \cite[Lemma~12]{BellaFehrmanFischerOtto} combined with the Caccioppoli inequality.

Iterating this bound $K:=\lfloor \log_2 \frac{R}{R^{1-\delta}}\rfloor$ many times, we arrive at
\begin{align*}
\fint_{B_{2^{-K} R}(x_0)} |\nabla u|^2 \dd x
\leq& (2^d \tilde \delta)^{K} \fint_{B_R(x_0)} |\nabla u|^2 \dd x
\\&
+C\sum_{k=1}^K \sum_{m=1}^M (2^d \tilde \delta)^{k-1} \bigg|\int_{B_R(x_0)} g_{k,m} \cdot \nabla u \dd x\bigg|^2
\end{align*}
with $g_{k,m}(x):=(2^{K-k})^{d} g_m(2^{K-k}x)$.
Choosing $\tilde \delta>0$ small enough, this implies the desired estimate.
\end{proof}

\medskip

Lemma~\ref{FunctionalsToNormLemma} reduces the proof of Proposition \ref{iterative_lemma} to estimating linear functionals of $\eps \nabla \theta^{\eps}_i$, the main challenge being to control the $2q$-th stochastic moments of 
\begin{align}
\label{def_Fkm}
F^\eps_{k,m}:=\int_{B_{\frac{x_0^{\perp}}{2}} (x_0)} g_{k,m} \cdot \eps \nabla \theta^{\eps}_i \dx.
\end{align}
For this, we begin by applying the triangle inequality as
\begin{align}
\label{triangle_expect}
\en{|F^\eps_{k,m}|^{2q} }^{\frac{1}{2q}} \leq \en{ |F^\eps_{k,m} - \en{F^\eps_{k,m} } |^{2q} }^{\frac{1}{2q}} + \big|\en{F^\eps_{k,m}} \big|.
\end{align}
Notice that if we were deriving estimates for the whole-space corrector $\phi^{\eps}_i$ instead of for the boundary layer corrector $\eps \theta^{\eps}_i$ (\ie, if we were to replace $\eps \theta^{\eps}_i \rightsquigarrow \phi^{\eps}_i$ in \eqref{def_Fkm}), then the expected value $\langle F^\eps_{k,m} \rangle$ would vanish due to the stationarity of $\nabla \phi^{\eps}_i$ and the sublinear growth of $\phi^{\eps}_i$ (see, \eg, \cite{PapanicolaouVaradhan, Kozlov}). Thus, in this case, the problem would entirely reduce to that of moment bounds for the fluctuations of $F^\eps_{k,m}$. In our situation, however, since $\nabla \theta^{\eps}_i$ is not stationary in the direction perpendicular to $\partial \Hd$ (\ie, in the $e_1$-direction), we need to control both terms on the right-hand side of \eqref{triangle_expect}. The fluctuation bounds are obtained in Proposition \ref{sens_Prop1} and the expected value $\langle F^\eps_{k,m} \rangle $ is treated in Proposition \ref{Systematic_error}, the latter being one of the novel contributions of this paper. 

%The statements of Propositions \ref{sens_Prop1} and \ref{Systematic_error} are given below, where we remark that, for use in future work, we have proven Proposition \ref{sens_Prop1} not only on $\Hd$, but also on a generic $C^{1,1}$-domain \red{--the point being that there are no major challenges in making this generalization.} We do remark that this generalization does require a large-scale regularity theory for random $a$-harmonic functions with homogeneous Dirichlet boundary data on $C^{1,1}$-domains --see Proposition \ref{large_scale_reg_domain} and Lemma \ref{GradientBoundByIntegral} of Appendix \ref{operators_on_domain}, both of which are actually also new on domains. In contrast, since our treatment of the expectation $\en{F_{k,m}}$ in Proposition \ref{Systematic_error}, depends quite heavily on our estimate being on the half-space, we do not generalize this statement to a generic domain in this contribution. 

Next, we state our estimate on the fluctuations of (linear functionals of) the boundary layer corrector.
% which, as previously mentioned (see Remark \ref{domains}), also
%For future use and includes the case of bounded $C^{1,1}$-domains.
While in the present paper we are concerned with the case of a half-space $\Hd$, in a subsequent paper we do plan to handle the case of bounded $C^{1,1}$-domains. To avoid excessive repetition in this upcoming work, we choose to prove the fluctuations estimate in Proposition \ref{sens_Prop1} below for both the case of a half-space $\Hd$ and the case of a bounded $C^{1,1}$-domain, the proof being extremely similar.
%We do this because, given the regularity results that we have previously shown in \cite{BFJR_1}, the argument does not see much change  for a $C^{1,1}$-domain.

\begin{proposition}\label{sens_Prop1}
Let $d \geq 3$, $i\in \{1, \ldots, d\}$, and let $\langle\cdot\rangle$ be an ensemble of coefficient fields on $\Rd^d$ that satisfies the assumptions (A1)-(A4). Furthermore, let $\domain$ denote either a bounded $C^{1,1}$-domain or $\Hd$; let $f \in W^{1,\infty}(\mathbb{R}^d)$ be smooth such that, in the case that $\domain = \Hd$, $f \equiv 1$ and, if $\domain$ is bounded, $f$ is supported in a bounded neighborhood of $\domain$. Let $p$ satisfy $0< p-1 \ll1$, $r \geq \varepsilon$, and $x_0 \in \domain$. Let $g  \in L^{2p}( \Rd^d ; \Rd^d)$ satisfy $\supp(g) \subset B_r(x_0)$ with $\dist(B_r(x_0), \partial\domain) \ge \frac{r}{4}$ and 
\begin{align}
\label{g_avg}
\bigg( \fint_{B_r(x_0)} |g  |^{2p} \dx \bigg)^{\frac{1}{2p}} \leq r^{-d}.
\end{align}
Consider $u_{\eps}$ solving
\begin{equation}
\begin{aligned}\label{sens_defu}
 -\nabla \cdot (a^{\eps}  \nabla u_{\eps}) &= 0 && \textrm{ in } \domain\\
  u_{\eps} &= f \phi_i\big(\tfrac{\cdot}{\eps}\big)  && \textrm{ on } \partial\domain
\end{aligned}
\end{equation}
and the random variable $F^{\eps}$ defined as 
\begin{align}
\label{defn_F}
 F^{\eps} := \int_\domain g\cdot  \eps \nabla u_{\eps} \dx.
\end{align}
Then there exists a constant $\kappa = \kappa(d, \lambda, \nu, \alpha, \beta, \domain, f) < \infty$ such that
\begin{align}
\label{sens_eqP1}
\en{ \big|F^{\eps} -\en{F^{\eps}}\big|^{2q}}^{\frac{1}{2q}} \lesssim_{d, \lambda, \nu, \alpha, \beta, \domain, f } q^{\kappa} \bigg(
\Big(\frac{\eps}{r}\Big)^{\frac{\beta d}{2}}
&\bigg( \frac{\varepsilon}{\dist(x_0,\partial\domain) +\varepsilon} \bigg)^{\frac {(1- \beta) d}{2} }
\\&\nonumber
+   \frac{\varepsilon}{\dist(x_0,\partial\domain)+ \varepsilon} \bigg( \frac{\varepsilon}{r} \bigg)^{\frac d2 }  \bigg)
\end{align}
holds for any $q \geq 1$ and any $0< \beta\ll1$.
\end{proposition}

\noindent In our argument for Proposition \ref{sens_Prop1} we rely on the spectral gap inequality \eqref{spectral_gap_L2} --in particular, for $\xi = F^{\eps}$, we express the Malliavin derivative on the right-hand side of \eqref{spectral_gap_L2} in terms of auxiliary functions $h^{\eps}$, $\hat{h}^{\eps}$, $v^{\eps}_1$, and $v^{\eps}_2$ solving \eqref{Definitionh}, \eqref{Definitionhath}, \eqref{sens_def_v1}, and \eqref{sens_def_v2} respectively. We then proceed to estimate these auxiliary functions via elliptic regularity estimates --the ground work for which was laid in \cite{BFJR_1}. The full argument is contained in Section \ref{sens_analysis} and is similar to previous treatments found, \eg,  in \cite{GNO5, FN_2020, BellaFehrmanFischerOtto, Bella_Kniely}. 

Next, we state our estimate on the expected value $\mathbb{E}[\eps \nabla\theta_i^\eps(x)]$ of the boundary layer corrector gradient; recall that we employ this result to bound the term $\langle F^\eps_{k,m} \rangle$ on the right-hand side of \eqref{triangle_expect}.
We emphasize that Proposition \ref{Systematic_error} relies critically on the stationarity properties of $\theta^\eps_i$ and does not extend to the case of a bounded $C^{1,1}$-domain.\\
%The statement of Proposition \ref{Systematic_error}, used to handle the expected values $\langle F^\eps_{k,m} \rangle$ on the right-hand side of \eqref{triangle_expect}, is as follows:
\begin{proposition}
\label{Systematic_error}
Let $d \geq 3$, let $i\in \{1, \ldots, d\}$, and let $\langle\cdot\rangle$ be an ensemble of coefficient fields that satisfies the assumptions (A1)-(A4). Suppose that there exists $n>0$ such that 
\begin{equation}\label{M:Hypo:1}
\en{ \left|\eps\nabla \theta^\eps_i (x) \right|^2 }^{\frac{1}{2}} \lesssim_{d, \lambda, \nu, \alpha} \bigg(1+\frac{x^\perp}{\varepsilon}\bigg)^{-n}
\end{equation}
holds for any $x\in \Hd$. Let $r \in [\varepsilon,\tfrac{1}{4}x_0^\perp]$, let $x_0 \in \Hd$, and let $g  \in H^1(\Hd ; \Rd^d)$ satisfy $\supp g\subset B_r(x_0)$ and
%$\supp(g) \subset B_r(x_0)$ with $\dist(B_r(x_0), \partial\Hd) \ge \frac{r}{4}$ and 
%\begin{align}
%\label{g_avg_1}
%\bigg( \fint_{B_r(x_0)} |g|^{2p} \dx \bigg)^{\frac{1}{2p}} \leq r^{-d}
%\end{align}
%as well as
\begin{align}
\label{g_avg_2}
\bigg( \fint_{B_r(x_0)} |\nabla g|^{2} \dx \bigg)^{\frac{1}{2}} \leq r^{-(d +1)}.
\end{align}
Then the expected value $\langle F^{\eps} \rangle$ of the linear functional $F^\eps:=\int_{\Hd} g \cdot \eps \nabla \theta^\eps_i \dx$ satisfies the improved decay estimate
\begin{align}
\label{M:Estim:01}
 \left|\left\langle F^\eps \right\rangle\right| \lesssim_{d,\lambda, \nu, \alpha} \frac{\eps}{r} \bigg(1+\frac{x_0^\perp}{\varepsilon}\bigg)^{-n}.
\end{align} 
\end{proposition}
\noindent Our argument for Proposition \ref{Systematic_error} is given in Section \ref{expectation}. The underlying idea is to use a two-scale expansion approximation for the ``Green's function-like'' function $h^{\eps}$ solving $-\nabla \cdot (a^\transpose\big(\tfrac{\cdot}{\eps}\big) \nabla h^{\eps})=\nabla \cdot g$ to reduce --up to the higher-order error term resulting in the right-hand side of \eqref{M:Estim:01}-- the expression for $F^{\eps}$ to a corresponding expression involving only the correctors (whole-space and boundary layer) and the homogenized analogue $\bar{h}$ solving $-\nabla \cdot (\bar a^\transpose \nabla \bar h) = \nabla \cdot g$. One then exploits stationarity tangentially to $\partial \Hd$ to conclude that the remaining expressions in fact vanish.
% as well as the vanishing expected normal flux $\langle e_1 \cdot \nabla \theta_i\rangle=0$

\bigskip
With all of the ingredients for the proof of Proposition \ref{iterative_lemma} ready, we now proceed to its proof.
\begin{proof}[Proof of Proposition~\ref{iterative_lemma}] \Wlog set $\varepsilon =1$. Furthermore, we may assume that $x_0^\perp \geq 8$ as otherwise \eqref{result_iterative_lemma} follows trivially from \eqref{assumption_iterative_lemma}. Now, set $R:= \frac{x_0^\perp}{8}$. Inserting our assumption \eqref{assumption_iterative_lemma} into Lemma~\ref{FunctionalsToNormLemma} (applied for $A=1$), we obtain
\begin{align*}
&\bigg\langle
\bigg(
\fint_{B_{R^{1-\delta}}(x_0)} |\nabla \theta_i|^2 \, \dd x\bigg)^\frac{q}{2}
\bigg\rangle^{\frac1q}
\\&
\lesssim
q^m (1+ x_0^{\perp} )^{-n-1}
+\sum_{k=1}^{\lfloor\log_2 \frac{R}{R^{1-\delta}} \rfloor}
(2^k)^{-\frac{1}{\delta}}
\bigg\langle \bigg(
\sum_{m=1}^M \bigg|\int_{B_R(x_0)} g_{k,m} \cdot \nabla \theta_i \, \dd x\bigg|^2
\bigg)^\frac{q}{2} \bigg\rangle^{\frac1q}.
\end{align*}
%Thus, our task is essentially reduced to estimating linear functionals of the form $\int g \cdot \nabla \theta_i \dx$.
For the second term on the right-hand side we now use \eqref{triangle_expect}. To bound the corresponding fluctuations we apply Proposition \ref{sens_Prop1} with $r = 2^{k} R^{1-\delta}$ and $g =g_{k,m}$. The conditions on $r$ and $g$ in Proposition \ref{sens_Prop1} follow from \eqref{mult_cond_1} and \eqref{mult_cond_3}. To estimate the corresponding expected values we apply Proposition \ref{Systematic_error} --again with $g = g_{k,m}$ and $r = 2^{k} R^{1-\delta}$, where the assumption \eqref{g_avg_2} follows from \eqref{mult_cond_2}. Furthermore, we obtain the pointwise bound \eqref{M:Hypo:1} from our hypothesis \eqref{assumption_iterative_lemma} and a down-propagation to a pointwise bound via \cite[Corollary~4]{BFJR_1}.  

Combining all of these ingredients, we obtain
\begin{align*}
& \bigg\langle
\bigg( \fint_{B_{R^{1-\delta}}(x_0)} |\nabla \theta_i|^2 \dd x\bigg)^\frac{q}{2}
\bigg\rangle^{\frac1q} \\
& \lesssim
q^m (1+ x_0^{\perp} )^{-n-1}
\\
&
 \,\, +q^{\kappa} \hspace{-.1cm}
\sum_{k=1}^{\lfloor\log_2 \frac{R}{R^{1-\delta}} \rfloor}
(2^k)^{-\frac{1}{\delta}}
\bigg(
\Big(\frac{1}{2^k R^{1-\delta}}\Big)^{\beta\frac{d}{2}}
(1+x_0^\perp)^{(\beta-1)\frac d2 } +  (1 +x_0^\perp)^{-1} \bigg( \frac{1}{2^kR^{1-\delta}} \bigg)^{\frac d2 }  \bigg)
\\&
\,\,+\sum_{k=1}^{\lfloor\log_2 \frac{R}{R^{1-\delta}} \rfloor}
(2^k)^{-\frac{1}{\delta}}
\frac{1}{2^kR^{1-\delta}}
(1+x_0^\perp)^{-n}. 
\end{align*}
Estimating the sums, inserting the definition of $R$, and choosing $\beta$ sufficiently small (depending on $\delta$), we obtain
\begin{align*}
&\bigg\langle
\bigg(
\fint_{B_{\big(\frac{x_0^\perp}{8} \big)^{1-\delta}}(x_0)} |\nabla \theta_i|^2 \dd x\bigg)^q
\bigg\rangle^{\frac1q}
\lesssim_{\delta}    
q^m (1+ x_0^{\perp} )^{-n-1+\delta}
+q^{\kappa} (1+ x_0^{\perp} )^{-(1-\delta)\frac d2}.
\end{align*}
Letting $x_0$ vary over a ball of radius $x_0^\perp/8$ around $x_0$ and integrating proves Proposition~\ref{iterative_lemma} (up to renaming $\delta$).
\end{proof}

\subsection{Proof of Theorem \ref{thm_2}}

As the proof of our second main result, Theorem~\ref{thm_2}, is rather short, we directly give it here.

\begin{proof}[Proof of Theorem~\ref{thm_2}] By a rescaling argument, we may assume that $\eps=1$.
%\red{Throughout this proof the random fields $\mathcal{C}(a,x)$ may also depend on the full cohort $(d, \lambda, \nu, \alpha, \delta)$.}

\smallskip

Our argument comes in four steps: In Step 1 we decompose $\phi_i^{L}$ solving \eqref{box_corrector} in terms of the whole-space corrector and the boundary layer correctors corresponding to the faces of $[0,L]^d$ --in our decomposition, whenever we subtract a boundary layer corrector from the whole-space corrector, we must compensate for this new contribution along the other faces of the cube (this is done via $\gamma_{i,m}^{L}$ and $\tau_{i,m}^{L}$ below). The contributions introduced via $\gamma_{i,m}^{L}$ and $\tau_{i,m}^{L}$ are estimated in Steps 2 and 3. In Step 4, we conclude our argument. \\

\smallskip

\noindent  {\bf Step 1: Representation formula for $\phi_i^{L}$.} \qquad For $1\leq m\leq 2d$ we use the convention 
\begin{align}
\mathbb{H}^m :=
\begin{cases}
 \{x\in \Rd^d :x_m>0\} & \text{ for } 1\leq m\leq d,\\
 \{x\in \Rd^d:x_{m-d}<L\} &  \text{ for } d+1\leq m\leq 2d.
\end{cases}
\end{align}
Let $\theta_i^{\mathbb{H}^m}(\infty)$ by slight abuse of notation denote the constants that the boundary layer correctors $\theta_i^{ \mathbb{H}^m}$ converge to \as as the distance to the half-space boundary becomes infinite (the existence of such constants follows from Theorem \ref{thm_1}). Furthermore, denote by $E:=(\{0,L\} \times \{0,L\} \times [0,L]^{d-2}) \cup \ldots \cup ([0,L]^{d-2}\times \{0,L\} \times \{0,L\})$ the set of all edges in case $d=3$ respectively for $d\geq 4$ the set of all $d-2$\,-dimensional hyperedges of the cube.

Note that by the defining equations, we may write
\begin{align}
\label{SumRep}
\phi_i^{L} = \phi_i - \sum_{m=1}^{2d}(\theta_i^{\mathbb{H}^m} - \gamma_{i,m}^{ L} -\tau_{i,m}^{ L}),
\end{align}
where $\gamma_{i,m}^{ L}$ is defined as the (weak) solution to the boundary value problem
\begin{equation}
\begin{aligned}
-\nabla \cdot (a\nabla \gamma_{i,m}^{ L}) &=0 &&\text{in }(0,L)^d,
\\
\gamma_{i,m}^{ L} &=  \theta_i^{\mathbb{H}^m}(\infty) &&\text{on }\partial [0,L]^d \setminus \partial \mathbb{H}^m,
\\
\gamma_{i,m}^{L} &= 0 &&\text{on }\partial [0,L]^d \cap \partial \mathbb{H}^m,
\end{aligned}
\end{equation}
and $\tau_{i,m}^{L}$ respectively of
\begin{equation}
\begin{aligned}
- \nabla \cdot (a \nabla \tau_{i,m}^{L})& =0 &&\text{in } (0,L)^d,\\
\tau_{i,m}^{L} &= \theta_i^{\mathbb{H}^m} - \theta_i^{\mathbb{H}^m}(\infty) &&\text{on }\partial [0,L]^d \setminus \partial \mathbb{H}^m,
\\
\tau_{i,m}^{L} &= 0 &&\text{on }\partial [0,L]^d \cap \partial \mathbb{H}^m.
\end{aligned}
\end{equation}
Observe that $\theta_i^{\mathbb{H}^m} - \gamma_{i,m}^{L} - \tau_{i,m}^{L} \equiv 0$ on $\partial [0,L]^d \setminus \partial \mathbb{H}^m$ and $ \theta_i^{ \mathbb{H}^m} - \gamma_{i,m}^{L} - \tau_{i,m}^{L} = \phi_i$ on $\partial [0,L]^d \cap \partial \mathbb{H}^m$.
Furthermore, in view of \eqref{main_goal} of Theorem~\ref{thm_1}, notice that the boundary conditions for $\tau_{i,m}^{L}$ imply that
\begin{equation}
\label{BoundaryEstimateTau}
\begin{aligned}
|\tau_{i,m}^{L} (x)| \leq \mathcal{C}(a,x) (1+\dist(x,\partial \mathbb{H}^m))^{1-d/2+\delta} 
&&\text{on }\partial [0,L]^d.
\end{aligned}
\end{equation}
Note that the boundary data for $\gamma_{i,m}^{L}$ and $\tau_{i,m}^{L}$ may be discontinuous at the set of edges $E$; however, this is not an issue when working in the space $H^1((0,L)^d)$, as the edges have capacity zero.
%\blue{Note also that some mathematical issues may arise due to the discontinuity of the boundary data for $\gamma_{i,m}^L$ and $\tau_{i,m}^L$ at the corners and edges; to maintain strict mathematical rigor, a smoothing argument on a small scale would be required. However, none of the below arguments will be affected by the smoothing procedure, hence we omit details of this technicality.}

The representation \eqref{SumRep} enables us to split
\begin{align*}
&\fint_{[\kappa L,(1-\kappa)L]^d} a(e_i+\nabla \phi_i^{L}) \dd x
\\&
=\fint_{[\kappa L,(1-\kappa)L]^d} a (e_i+\nabla \phi_i) \, \dd x
-\sum_{m=1}^{2d}
\fint_{[\kappa L,(1-\kappa)L]^d} a \nabla \theta_i^{\mathbb{H}^m} \, \dd x
\\&~~~
+\sum_{m=1}^{2d}\fint_{[\kappa L,(1-\kappa)L]^d} a \nabla \gamma_{i,m}^{L} \, \dd x
+\sum_{m=1}^{2d}\fint_{[\kappa L,(1-\kappa)L]^d} a \nabla \tau_{i,m}^{L} \, \dd x.
\end{align*}
While the first term on the right-hand side is the desired one and the second term may be estimated by \eqref{main_goal} from Theorem~\ref{thm_1}, it remains to account for the contributions of the last two terms.\\

\smallskip

\noindent {\bf Step 2: Estimate on $\fint_{[\kappa L,(1-\kappa)L]^d}  \nabla \tau_{i,m}^{ L} \dd x$.} \qquad
Our first goal is to derive the bound
\begin{align}
\label{EstimateTau}
\bigg|
\fint_{[\kappa L,(1-\kappa)L]^d} \nabla \tau_{i,m}^{L} \dd x
\bigg|
+
\bigg|
\fint_{[\kappa L,(1-\kappa)L]^d}  a \nabla \tau_{i,m}^{L} \dd x
\bigg|
\leq \mathcal{C}(a)
(L^{-d/2+\delta}+L^{-2}).
\end{align}
%Heuristically, by the using the Poisson kernel, we expect (except at the corners and edges) 
%\begin{align*}
%|\nabla \tau_{i,m}^L(x)|
%&\leq \mathcal{C}(a,x) \int_{\partial [0,L]^d} \frac{\mathcal{C}(a,\tilde x)}{|x-\tilde x|^d} |\tau_{i,m}^L(\tilde x)| \, \dd S(\tilde x)
%\\&
%\leq
%\mathcal{C}(a,x) \dist(x,E)^{1-d/2+\delta} \dist(x,\partial [0,L]^d)^{-1} + \mathcal{C}(a,x) \dist(x,E)^{-2},
%\end{align*}
%where the two terms on the the right-hand side represent the splitting of the surface integral into the cases $\dist(x, E) \leq \dist(\tilde{x}, E)$ and $\dist(\tilde{x}, E) \leq \dist(x, E)$. However, this is in fact a more precise estimate than actually required for our conclusion. To make the required estimates rigorous without too much effort,
To this aim, we define $h$ as the weak solution to $-\nabla \cdot (a^\transpose \nabla h)=\nabla \cdot ( \chi_{[\kappa L,(1-\kappa)L]^d} e_k)$ in $(0,L)^d$ with $h=0$ on $\partial [0,L]^d$. Using the equations for $h$ and (in the second step) for $\tau_{i,m}^{L}$, we obtain
\begin{align*}
& \fint_{[\kappa L,(1-\kappa)L]^d} e_k \cdot \nabla \tau_{i,m}^{L}  \, \dd x\\
&
= (1-2\kappa)^{-d} L^{-d} \int_{ \partial [0,L]^d} \tau_{i,m}^{L} \vec{n}\cdot a^\transpose \nabla h \, \dd S
-(1-2\kappa)^{-d} L^{-d} \int_{[0,L]^d} a \nabla \tau_{i,m}^{L} \cdot \nabla h \, \dd x
\\&
=(1-2\kappa)^{-d} L^{-d} \int_{\partial[0,L]^d} \tau_{i,m}^{L} \vec{n}\cdot a^\transpose \nabla h \, \dd S,
\end{align*}
where $\vec{n}$ denotes the outward normal vector. A regularity theory estimate for $\nabla h$ ($h$ being $a$-harmonic in a neighborhood of the boundary) in conjunction with the energy estimate for $h$ yields 
\begin{align}
\label{Lipschitz_corner}
|\nabla h (x)|
\leq \mathcal{C}(a,x) \bigg(\fint_{[0,L]^d} |\nabla h|^2 \dd x\bigg)^{1/2}
\leq \mathcal{C}(a,x)
\quad\text{for }x\in \partial [0,L]^d.
\end{align}
We remark that this follows from Theorem 2.2 of \cite{JosienRaithelSchaeffner} in the case that $d=2$, along with the regularity assumption (A4) on $a$. In the higher-dimensional setting, the regularity theory for harmonic functions on polygonal domains includes not only certain singular ``corner corrections'' (treated in \cite{JosienRaithelSchaeffner}), but also ``edge corrections'' --see, \eg, \cite{Dauge_book}. However, as we have shown in \cite[Theorem~7]{BFJR_1}, the convexity of the box $[0,L]^d$ allows us to get away with not including the edge contributions.

Plugging \eqref{Lipschitz_corner} into the previous  bound and using the estimate \eqref{BoundaryEstimateTau} for the boundary data of $\tau_{i,m}^{L}$ as well as $\tau_{i,m}^{L} \equiv 0$ on $\partial [0,L]^d \cap \partial \mathbb{H}_m$, we obtain
\begin{align*}
\bigg|\fint_{[\kappa L,(1-\kappa)L]^d} e_k \cdot \nabla \tau_{i,m}^{L} \dd x\bigg|
\leq \mathcal{C}(a) (L^{-d/2+\delta} + L^{-2}
).
\end{align*}
Proceeding analogously for the flux (simply inserting a factor of $ a^\transpose$ in the right-hand side of the equation for $h$), we obtain \eqref{EstimateTau}.\\

\smallskip
\noindent
{\bf Step 3: Estimate on $\fint_{[\kappa L,(1-\kappa)L]^d} \nabla \gamma_{i,m}^{L} \dd x$.} \qquad
Concerning $\gamma_{i,m}^{L}$, we employ a two-scale expansion ansatz
\begin{align}
\label{2scale_RVE}
\gamma_{i,m}^{L} = \bar \gamma_{i,m}^{L} + \phi_j^{L} \partial_j \bar \gamma_{i,m}^L + w_{i,m}^{L}
\end{align}
with $\bar \gamma_{i,m}^L$ solving
\begin{equation}
\begin{aligned}
-\nabla \cdot (\bar a \nabla \bar \gamma_{i,m}^L) &=0 &&\text{in }(0,L)^d,
\\
\bar \gamma_{i,m}^L &=\theta_i^{\mathbb{H}^m}(\infty)  &&\text{on }\partial [0,L]^d\setminus \partial \mathbb{H}^m,
\\
\bar \gamma_{i,m}^L &=0 &&\text{on }\partial [0,L]^d\cap \partial \mathbb{H}^m,
\end{aligned}
\end{equation}
and with $w_{i,m}^L$ solving
\begin{align*}
-\nabla \cdot (a\nabla w_{i,m}^L) = \nabla \cdot ((a\phi_j^L-\sigma_j) \partial_j \nabla \bar \gamma_{i,m}^L ) + a \nabla(\phi_j^L - \phi_j ) \cdot \nabla \partial_j  \bar \gamma_{i,m}^L 
\end{align*}
with homogeneous Dirichlet boundary data. Note that the equation satisfied by $w_{i,m}^L$ is obtained by the usual property of the two-scale expansion
\begin{align*}
-\nabla \cdot \bigg(a\nabla \bigg(\bar \gamma_{i,m}^L+ \phi_j\partial_j \bar \gamma_{i,m}^L\bigg)\bigg) = - \nabla \cdot ((a\phi_j-\sigma_j)\partial_j \nabla \bar \gamma_{i,m}^L)
\end{align*}
as well as $\phi_j^L-\phi_j$ being $a$-harmonic in $(0,L)^d$. Furthermore, due to the piecewise constant boundary data of $\bar \gamma_j^L$, we have that
\begin{subequations}
\begin{align}
\label{bargamma_bound_1}
|\nabla \bar \gamma_j^L(x)|&\leq C \dist(x,E)^{-1},
\\
\label{bargamma_bound}
|D^2 \bar \gamma_j^L(x)|&\leq C \dist(x,E)^{-2},
\end{align}
\end{subequations}
which can, \eg, be seen using the Poisson kernel representation.   %\blue{[C: To use the Poisson kernel formula don't we have to have like a $C^{1,1}$ boundary or something? Note to self, it is okay because we are convex.]}

Defining the function $h$ as the solution to the equation $-\nabla \cdot (a^\transpose\nabla h)=\nabla \cdot ( \chi_{[\kappa L,(1-\kappa)L]^d} e_k)$ in $(0,L)^d$ with homogeneous Dirichlet boundary data, we deduce using first this equation and then the equation for $w_{i,m}^L$
\begin{align*}
&\fint_{[\kappa L,(1-\kappa)L]^d} e_k \cdot \nabla w_{i,m}^L \, \dd x\\
& =  -(1-2\kappa)^{-d} L^{-d}
\int_{(0,L)^d} a^\transpose \nabla h \cdot \nabla w_{i,m}^L \, \dd x
\\&
=(1-2\kappa)^{-d} L^{-d}
\int_{(0,L)^d} (a\phi_j^L-\sigma_j)\partial_j \nabla \bar \gamma_{i,m}^L \cdot \nabla h  \,\dd x 
\\&~~~
- (1 -2 \kappa)^{-d} L^{-d}\int_{(0,L)^d} h a \nabla ( \phi_j^L - \phi_j) \cdot \nabla \partial_j \bar \gamma_{i,m}^L \dx.
\end{align*}
Notice that by \eqref{bargamma_bound} the integrals on the right-hand side are well-defined since $|\nabla h(x)| \leq \mathcal{C}(a,x) \big(\dist(x, E)/L\big)^{\delta}$, which follows from the convexity of the box $[0,L]^d$ (see, \eg, \cite[Theorem~7]{BFJR_1}). To estimate the right-hand side we again use that $|\nabla h(x)|\leq \mathcal{C}(a,x) (\dist(x, E)/L)^{\delta}$ -- additionally using the fundamental theorem of calculus and the homogeneous Dirichlet boundary data one may also obtain the estimate $|h(x)| \leq \C(a,x) \dist(x, \partial[0,L]^d) (\dist(x, E)/L)^{\delta}$. Combining these bounds with \eqref{BoundDifferencePhiPhiL} below and using again an analogous argument for the flux contribution, we deduce
\begin{align}
\label{BoundW}
\bigg|
\fint_{[\kappa L,(1-\kappa)L]^d} \nabla w_{i,m}^L \, \dd x
\bigg|
+\bigg|
\fint_{[\kappa L,(1-\kappa)L]^d} a \nabla w_{i,m}^L \, \dd x
\bigg|
\leq \mathcal{C}(a) L^{-2} \log L.
\end{align}

For the previous bound, we have used an estimate on $\phi_i-\phi_i^L$.
%To derive this bound, let $h$ satisfy $-\nabla \cdot (a^\ast \nabla h)=\nabla \cdot \delta_{x_0}$ with $h=0$ on $\partial [0,L]^d$. Then
%\begin{align*}
%\nabla \phi_i(x_0)-\nabla \phi_i^L(x_0)=\int_{\partial [0,L]^d} \phi_i \vec{n} \cdot a^\ast\nabla h \,dS
%\end{align*}
To derive this bound, we deduce from \eqref{PoissonKernelEstimateAHarmonic}
\begin{align*}
|\nabla \phi_i(x_0)-\nabla \phi_i^L(x_0)|
\leq \mathcal{C}(a,x_0) \inf_b \int_{\partial [0,L]^d} \frac{\mathcal{C}(a,x) |\phi_i-b|}{|x-x_0|^d}  \,dS
\end{align*}
and thus (due to
%$|\nabla h|(x)\leq \mathcal{C}(a,x,x_0) |x-x_0|^{-d}$,
the corrector bounds from \cite[Theorem 2]{GNO5} and the small-scale regularity from (A4) applied to $x_i+\phi_i$) we have that
\begin{align}
\label{BoundDifferencePhiPhiL}
|\nabla \phi_i(x_0)-\nabla \phi_i^L(x_0)| \leq \frac{\mathcal{C}(a,x_0)}{\dist(x_0,\partial [0,L]^d)+1} 
\end{align}
and as a consequence
\begin{align}
\label{BoundDifferencePhiPhiL2}
|\phi_i(x_0)-\phi_i^L(x_0)| \leq \mathcal{C}(a,x_0) \log L.
\end{align}
Plugging in the definition of $w_{i,m}^L$ from \eqref{2scale_RVE} into \eqref{BoundW} and using \eqref{bargamma_bound}, \eqref{BoundDifferencePhiPhiL}, \eqref{BoundDifferencePhiPhiL2}, and the bound on $\phi_i$ to estimate the terms $\fint_{[\kappa L,(1-\kappa)L]^d} \phi_j^L \nabla \partial_j {\bar \gamma}_{i,m}^L \,dx$ and  $\fint_{[\kappa L,(1-\kappa)L]^d} a \phi_j^L \nabla \partial_j {\bar \gamma}_{i,m}^L \,dx$, we thus obtain
\begin{align}
\nonumber
&
\bigg|
\fint_{[\kappa L,(1-\kappa)L]^d} \nabla \gamma_{i,m}^L \dd x
-\fint_{[\kappa L,(1-\kappa)L]^d} (e_j+\nabla \phi_j) \partial_j {\bar \gamma}_{i,m}^L \dd x
\bigg|
\\&
\label{EquationGamma}
+\bigg|
\fint_{[\kappa L,(1-\kappa)L]^d} a \nabla \gamma_{i,m}^L \dd x
-\fint_{[\kappa L,(1-\kappa)L]^d} a(e_j+\nabla \phi_j) \partial_j {\bar \gamma}_{i,m}^L \dd x
\bigg|
\\&
\nonumber
\leq \mathcal{C}(a) L^{-2+\delta}.
\end{align}

\smallskip 

\noindent {\bf Step 4: Conclusion.} \qquad 
Using the splitting \eqref{SumRep}, the decay estimate for the boundary layer corrector gradient \eqref{main_goal}, the bounds \eqref{EstimateTau} and \eqref{EquationGamma}, and the classical fluctuation estimate $|\fint_{[\kappa L,(1-\kappa)L]^d} g \nabla \phi_i \dx|\leq \mathcal{C}(a) ||g||_{L^\infty} L^{-d/2}$ for any deterministic $g$ (see, \eg, \cite{GNO5}) in conjunction with \eqref{bargamma_bound_1}, we deduce
\begin{subequations}
\label{EquationsForALwithR}
\begin{align}
\fint_{[\kappa L,(1-\kappa)L]^d} e_i+\nabla \phi_i^L \dx
= \fint_{[\kappa L,(1-\kappa)L]^d} \left( e_i+\sum_{m=1}^{2d}\nabla \bar \gamma_{i,m}^L \right) \dx
+\mathcal{R}_1
\end{align}
and
\begin{align}
\notag
&\fint_{[\kappa L,(1-\kappa)L]^d} a(e_i+\nabla \phi_i^L) \dx
\\&
= \fint_{[\kappa L,(1-\kappa)L]^d} a(e_i+\nabla \phi_i) \dx
+ \fint_{[\kappa L,(1-\kappa)L]^d} a(e_j+\nabla \phi_j)\sum_{m=1}^{2d}\partial_j \bar\gamma_{i,m}^L \dx
+ \mathcal{R}_2
\notag
\\&
= \bar a \fint_{[\kappa L,(1-\kappa)L]^d} e_i+\sum_{m=1}^{2d}\nabla \bar \gamma_{i,m}^L \dx
+ \mathcal{R}_3
\end{align}
\end{subequations}
with random variables $\mathcal{R}_k$ estimated as $|\mathcal{R}_1|+|\mathcal{R}_2|+|\mathcal{R}_3|\leq \mathcal{C} (L^{-d/2+\delta}+L^{-2+\delta})$ for some random variable $\mathcal{C}$ with uniformly bounded stretched exponential moments. Here in the last step we have used \eqref{bargamma_bound_1} in conjunction with the (classical) result that for any deterministic $g$ with $||g||_{L^\infty}\leq 1$ we have
\begin{align*}
\fint_{[\kappa L,(1-\kappa)L]^d} a(e_j+\nabla \phi_j) g_j \dx = \bar a \fint_{[\kappa L,(1-\kappa)L]^d} g \dx + \mathcal{R}
\end{align*}
with $|\mathcal{R}|\leq \mathcal{C} L^{-d/2}$; note that this is an immediate consequence of $\mathbb{E}[a(e_j+\nabla \phi_j) g_j]=\bar a g$ and the fluctuation bounds for the functional $\fint_{[\kappa L,(1-\kappa)L]^d} a(e_j+\nabla \phi_j) g_j \dx$ established, \eg, in \cite{GNO5}.

Note that if all the $\mathcal{R}_j$ were zero, the equations \eqref{EquationsForALwithR} and the definition \eqref{DefBarA} together would imply $\bar a^L = \bar a$.
In general, we see that the quantity $\bar a^L$ is obtained by inverting a $d^2$-dimensional system of linear equations, and in fact this system is (with overwhelming probability) a perturbation of order $L^{-\min\{d/2,2\}+\delta}$ of well-conditioned deterministic system with solution $\bar a$ (note that the deterministic system is well-conditioned since $\fint_{[\kappa L,(1-\kappa)L]^d} |\nabla \bar \gamma_{i,m}^L| \dx \lesssim L^{-1}$). The perturbation series for linear systems of equations now yields the desired assertion.
\end{proof}

\section{Argument for Proposition \ref{sens_Prop1}: Fluctuation estimates for the boundary layer corrector}

\label{sens_analysis}

In our proof of Proposition \ref{sens_Prop1} we use that the spectral gap inequality in assumption (A3) entails a corresponding estimate on arbitrary polynomial moments.
\begin{lemma} 
\label{pmoments}Let $\langle \cdot \rangle$ be an ensemble satisfying assumption (A3). Then for any random variable $F$ and $q\geq 1$ the estimate 
\begin{align}
\label{pLSI}
\en{ (F - \en{F})^{2q}}^{\frac{1}{2q}} \leq C q^\frac{1}{2} \varepsilon^{\frac{d}{2}}  \en{ \bigg( \int_{\Rd^d}  \bigg(\fint_{B_{\varepsilon}(x)} \Big| \frac{\partial F}{\partial a}\Big| \, \dd y \bigg)^{2}  \dx \bigg)^q  }^{\frac{1}{2q}}
\end{align}
holds. 
\end{lemma}
\noindent This standard lemma has been shown, \eg,  in \cite[Prop. 3.1]{DG_2019}. Applying \eqref{pLSI} to the random variable $F^{\eps}$ defined in \eqref{defn_F} and deriving an expression for the ``vertical derivatives'' on the right-hand side, we find that \eqref{sens_eqP1} reduces to estimating certain auxiliary functions that are used to describe $\frac{\partial F^{\eps}}{ \partial a^{\eps}}$.

We organize the argument into three parts: First, in Section \ref{identify_deriv}, we derive an expression for $\frac{\partial F^{\eps}}{ \partial a^{\eps}}$ in terms of auxiliary functions $h^{\eps}$, $v^{\eps}_1$, and $v^{\eps}_2$ --each solving an elliptic equation with coefficient field $a^\transpose\big(\tfrac{\cdot}{\eps}\big)$ on either a bounded $C^{1,1}$-domain or $\Hd$ with homogeneous Dirichlet boundary data (and, in the case of $\Hd$, a decay condition) or on $\Rd^d$ with a decay condition. Second, using a weighted Meyers estimate which we have previously proven in \cite[Proposition~6]{BFJR_1}, we estimate the auxiliary functions. Lastly, in Section \ref{proof_prop_5}, we finish our argument by combining these estimates with Lemma \ref{pmoments} and the expression for $\frac{\partial F^{\eps}}{ \partial a^{\eps}}$ that has been derived in Section \ref{identify_deriv}. 

% (see, \textit{e.\,g.}, \cite[Lemma 7]{BellaFehrmanFischerOtto},\cite[Step 3, Proof of Proposition 3]{GNO_reg_3}, \cite[Lemma 35]{FN_2020})

\subsection{Identification of the Malliavin / Fr\'echet derivative}
\label{identify_deriv}

In this section we calculate $\frac{\partial F^{\eps}}{ \partial a^{\eps}}$ in terms of auxiliary functions $h^{\eps}$, $v^{\eps}_1$, and $v^{\eps}_2$.

\begin{lemma} We adopt the assumptions of Proposition \ref{sens_Prop1}. Let $h^{\eps} \in \dot H^1_{0}(\domain)$ solve 
\begin{equation}
\label{Definitionh}
\begin{aligned}
-\nabla \cdot (a^\transpose\big(\tfrac{\cdot}{\eps}\big)\nabla h^{\eps}) &= \nabla \cdot \eps g && \text{in }\domain,
\\
h^{\eps}& =  0&&\text{on }\partial\domain,
\end{aligned}
\end{equation} 
and $\hat h^{\eps} \in \dot H^1(\Rd^d)$ be the decaying solution of 
\begin{align}
\label{Definitionhath}
-\nabla \cdot (a^\transpose\big(\tfrac{\cdot}{\eps}\big)\nabla \hat h^{\eps}) = \nabla \cdot \eps g\quad\text{on }\mathbb{R}^d.
\end{align}
We, furthermore, define $v^{\eps}_1$ and $v^{\eps}_2 \in \dot H^1(\Rd^d)$ as the decaying solutions of
\begin{align}\label{sens_def_v1}
 - \nabla \cdot (a^\transpose\big(\tfrac{\cdot}{\eps}\big) \nabla v^{\eps}_1) &= \frac{1}{\eps}
\nabla f \cdot a^\transpose\big(\tfrac{\cdot}{\eps}\big) (\chi_\domain \nabla h^{\eps}- \nabla \hat h^{\eps}) \quad  \textrm{in } \mathbb{R}^d,
\end{align}
and 
\begin{align}
\label{sens_def_v2}
 - \nabla \cdot ( a^\transpose\big(\tfrac{\cdot}{\eps}\big) \nabla v^{\eps}_2 ) &= - \frac{1}{\eps} \nabla \cdot (f a^\transpose\big(\tfrac{\cdot}{\eps}\big) (\chi_\domain \nabla h^{\eps} - \nabla \hat h^{\eps}) ) \quad \textrm{in } \mathbb{R}^d,
\end{align}
respectively. Then, for $F^{\eps}$ defined in \eqref{defn_F}, the vertical derivative $\frac{\partial F^{\eps}}{\partial a^{\eps}}$ is given by 
\begin{align}
\label{MalliavinDerivative}
\frac{\partial F^{\eps}}{\partial a^{\eps}} = \chi_\domain \nabla u_{\eps} \otimes \nabla h^{\eps} - (e_i+\nabla\phi^{\eps}_i) \otimes (\nabla v^{\eps}_1 + \nabla v^{\eps}_2). 
\end{align}
\end{lemma}

\begin{proof} We differentiate $F^{\eps}$ with respect to $a^{\eps}$ via a formal method of C\'ea involving an augmented Lagrangian approach \cite{Cea_1986}.
	The latter expresses the equations for $u_{\eps}$ and $\phi^{\eps}_i$ as Lagrangian constraints, which naturally give rise to the adjoint states $\hat{h}^{\eps}$ and $v^{\eps}_1 + v^{\eps}_2$.
	
	We set $\tilde{u}^{\eps}:=u_{\eps}- f\phi_i\big(\tfrac{\cdot}{\eps}\big) \in \dot H^1_{0}(\domain)$ and define an augmented Lagrangian
	\begin{align}\label{Def:Lag}
		\mathcal{L}(a^{\eps}, \phi^{\eps}_i, \tilde{u}^{\eps}, v^{\eps}, h^{\eps}) :=& \int_{\mathcal{O}} g \cdot \eps\nabla( \tilde{u}^{\eps} + f \phi_i\big(\tfrac{\cdot}{\eps}\big))  \dx + \int_{\mathcal{O}} \nabla h^{\eps} \cdot a^{\eps} \nabla (\tilde{u}^{\eps} + f \phi^{\eps}_i)  \dx
		\\&\nonumber
		+ \int_{\mathbb{R}^d} \nabla v^{\eps} \cdot a^{\eps} (\nabla \phi^{\eps}_i + e_i) \dx,
	\end{align}
	where the two last integrals are nothing but the variational form associated with \eqref{sens_defu} and the corrector equation on $\phi^{\eps}_i$, where $v^{\eps} \in \dot H^1 (\mathbb{R}^d)$ is sufficiently decaying at infinity and where $h^{\eps} \in \dot H^1_0(\mathcal{O})$.
	For such $h^{\eps}$ and $v^{\eps}$, it is obvious from the definition of $\phi^{\eps}_i$ and $\tilde{u}^{\eps}$ that $\mathcal{L}(a^{\eps},\phi^{\eps}_i,\tilde{u}^{\eps}, v^{\eps}, h^{\eps}) = F^{\eps}$.
	
	We now choose specific adjoint functions $h^{\eps}$ and $v^{\eps}$ such that $\partial_{\tilde{u}^{\eps}} \mathcal{L} = 0$ and $\partial_{\phi^{\eps}_i} \mathcal{L} = 0$.
	The first equation reads
	\begin{align*}
		\langle \partial_{\tilde{u}^{\eps}} \mathcal{L}(a^{\eps}, \phi^{\eps}_i, \tilde{u}^{\eps}, v^{\eps}, h^{\eps}), \delta \tilde{u}^{\eps}\rangle = 
		\int_{\mathcal{O}} g \cdot \eps \nabla \delta \tilde{u}^{\eps} \dx
		+
		\int_{\mathcal{O}} \nabla h^{\eps} \cdot a^{\eps} \nabla \delta \tilde{u}^{\eps} \dx = 0,
	\end{align*}
	in which we recognize the variational form associated with \eqref{Definitionh}. Notice that here $\delta \tilde{u}^{\eps}$ denotes the direction of the derivative with respect to $\tilde{u}^{\eps}$.
	The second equation reads
	\begin{align*}
		\langle \partial_{\phi^{\eps}_i} \mathcal{L}(a^{\eps}, \phi^{\eps}_i, \tilde{u}^{\eps}, v^{\eps}, h^{\eps}), \delta \phi^{\eps} \rangle
		=& \int_{\R^d} g \cdot \nabla( f \delta \phi^{\eps}) \dx + 
		\int_{\R^d} \chi_{\OO} \nabla h^{\eps} \cdot a^{\eps} \nabla (f \delta \phi^{\eps}) \dx
		\\&
		+ \int_{\mathbb{R}^d} \nabla v^{\eps} \cdot a^{\eps} \nabla \delta \phi^{\eps} \dx \\&=0,
	\end{align*}
	where we recall that $g$ is compactly supported in $\OO$ and again $\delta \phi^{\eps}$ is the direction of the derivative.
		Defining $\hat{h}^{\eps}$ by \eqref{Definitionhath}, the above equation yields
	\begin{align*}
		-\frac{1}{\eps}\int_{\R^d} \nabla \hat{h}^{\eps}  \cdot a^{\eps} \nabla (f \delta \phi^{\eps}) \dx + 
		\int_{\R^d} \chi_{\OO} \nabla h^{\eps} \cdot a^{\eps} \nabla (f \delta \phi^{\eps}) \dx
		+ \int_{\mathbb{R}^d} \nabla v^{\eps} \cdot a^{\eps} \nabla \delta \phi^{\eps} \dx=0,
	\end{align*}
	which is a defining equation for $v^{\eps}$ while $\delta \phi^{\eps}$ varies.
	Expanding the term $\nabla (f \delta \phi^{\eps})$ and performing an integration by parts, we recognize that $v^{\eps}=-v^{\eps}_1-v^{\eps}_2$ for $v^{\eps}_1$ and $v^{\eps}_2$ defined by \eqref{sens_def_v1} and \eqref{sens_def_v2} respectively.
	
	For such a choice of $h^{\eps}$ and $v^{\eps}$, since all partial derivatives of $\mathcal{L}$ vanish but the very first one, we immediately obtain \eqref{MalliavinDerivative} from \eqref{Def:Lag}.
\end{proof}

\subsection{Estimates on the auxiliary functions}
\label{aux_estimates}

Having calculated the derivative of $F^{\eps}$ with respect to changes in the coefficient field $a^{\eps}$ in the previous step, we now seek to control the terms on the right-hand side of \eqref{MalliavinDerivative} by obtaining suitable estimates for $v^{\eps}_1, v^{\eps}_2$, and $h^{\eps}$. 

For this, the main tool is a weighted Meyers estimate, the proof of which relies on a large-scale regularity result for $a$-harmonic functions on $\domain$ with homogeneous Dirichlet boundary data. In particular, it relies on:

\begin{proposition}[{Theorem~2 of \cite{FischerRaithel} \& Theorem~3 of \cite{BFJR_1}}]
\label{large_scale_reg_domain} We adopt the assumptions of Proposition~\ref{sens_Prop1}. For $R>0$ and $x_0 \in \overline{\domain}$, let $u \in H^1_{\loc}(\domain)$ solve 
\begin{equation}
\label{u_with_massive_term}
\begin{aligned}
-\nabla \cdot (a^{\eps}\nabla u) &= 0 && \text{in } \quad  B_R(x_0) \cap \domain,\\
u & = 0 && \text{on } \quad  B_R(x_0) \cap \partial \domain.
\end{aligned}
\end{equation}
If $\domain$ is a bounded $C^{1,1}$ domain, assume additionally that $R\leq c(\domain)$.
Then there exists a random field $r_{\domain}^*(a,x) \geq \varepsilon$ such that $ \frac{r_{\domain}^*}{\varepsilon}$ has stretched exponential moments (in the sense of \eqref{stretched}) and the relation
\begin{align}
\label{large_scale_mvp_domain}
\fint_{\domain \cap B_r(x_0)} |\nabla u|^2 \dx  \lesssim_{d, \lambda, \domain} \fint_{\domain \cap B_R(x_0)} |\nabla u|^2 \dx 
\end{align}
holds for any $C(\domain) \geq R \geq r \geq r_{\domain}^*(a^{\eps},x_0)$.

In the case that $\domain = \Hd$, the minimal radius $r_{\Hd}^*(a^{\eps}, x)$ is stationary with respect to shifts tangential to $\partial \Hd$.
\end{proposition}

\noindent As shown in \cite[Proposition~6]{BFJR_1}, this implies the following large-scale Meyers estimate: 

\begin{lemma}[{\cite[Proposition~6]{BFJR_1}}]
We adopt the assumptions of Proposition \ref{sens_Prop1}. Furthermore, let $x_0 \in \domain$ and $R \geq \max\{r_{\domain}^*(a^{\eps},x_0),r_{\domain}^*(a^\transpose\big(\tfrac{\cdot}{\eps}\big),x_0)\} \geq \eps$. For $f \in L^2(\domain)$ and $g \in L^2(\domain,\mathbb{R}^d)$, let $v$ denote the unique solution of 
\begin{equation}
\label{P002}
\begin{aligned}
-\nabla \cdot (a^\transpose\big(\tfrac{\cdot}{\eps}\big) \nabla v) &= \nabla \cdot g + f \quad &&\text{in }\domain,\\
v &= 0 \quad &&\text{on }\partial\domain
\end{aligned}
\end{equation}
and introduce the weight
\begin{equation}
\label{omega}
 \omega_{\alpha,R}(x) := \bigg( \frac{|x-x_0|}{R} + 1\bigg)^\alpha.
 \end{equation}
Then, in the case that $f=0$, there exists $\bar p(\lambda, d, \domain) > 1$ such that, for $1 \le p < \bar p$ and $0 \le \alpha_0 < \alpha_1 < d(2p-1)$, the estimate
\begin{align}
\label{Lemma_4_est_1}
\begin{split}
&\bigg( \int_{\domain}  |\nabla v|^{2p} \omega_{\alpha_0,R}  \dx \bigg)^{\frac{1}{2p}}  \lesssim_{d,\lambda,p, \domain,\alpha_0, \alpha_1} 
\bigg( \int_{\domain} |g|^{2p}  \omega_{\alpha_1,R}  \dx \bigg)^{\frac{1}{2p}} 
\end{split}
\end{align}
holds. Likewise, in the case that $g = 0$, there exists $\bar p(\lambda, d, \domain) > 1$ such that, for $1 \le p < \bar p$ and $0 \le \alpha_0 < \alpha_1 -2p <\alpha_1 < d(2p-1)$, the estimate
\begin{align}
\label{Lemma_4_est_2}
\begin{split}
&\bigg( \int_{\domain} |\nabla v|^{2p}  \omega_{\alpha_0,R}  \dx \bigg)^{\frac{1}{2p}} \lesssim_{d,\lambda,p, \domain,\alpha_0, \alpha_1}
R \bigg( \int_{\domain} | f|^{2p}  \omega_{\alpha_1,R}  \dx \bigg)^{\frac{1}{2p}}
\end{split}
\end{align}
holds. We emphasize that when $\domain = \Hd$, the boundary condition in \eqref{P002} is augmented by a decay condition in the far-field of the infinite part of the domain.
\end{lemma}

\noindent We remark that the whole-space version (\ie $\domain = \Rd^d$) of this large-scale weighted Meyers estimate has already been shown, \eg, in \cite[Lemma 7]{BellaFehrmanFischerOtto}.

An immediate consequence of the above estimate is the following corollary:
\begin{corollary}\label{sens_Cor1}
Adopt the assumptions of Proposition \ref{sens_Prop1}. Then, for $h^{\eps}$ and $\hat{h}^{\eps}$ solving \eqref{Definitionh} and \eqref{Definitionhath} respectively, there exists $\bar p (d,\lambda)> 1$ such that, for $1 \le p < \bar p$ and $0 \le \alpha < d(2p-1)$, the estimates
\begin{align}
\label{Meyers1} 
\bigg( \int_{\domain}  |\nabla h^{\eps}|^{2p} \omega_{\alpha,r+r^{\domain}_*(x_0)} \dx   \bigg)^{\frac{1}{2p}} &\lesssim_{d,\lambda,p,\domain, \alpha} \bigg( \int_{B_r(x_0)} |\eps g|^{2p} \dx  \bigg)^{\frac{1}{2p}},
\\
\label{Meyers2} 
\bigg( \int_{\mathbb{R}^d} |\nabla \hat h^{\eps}|^{2p} \omega_{\alpha,r+r^{\domain}_*(x_0)} \dx  \bigg)^{\frac{1}{2p}} &\lesssim_{d,\lambda,p,\domain, \alpha}  \bigg( \int_{B_r(x_0)} |\eps g|^{2p} \dx \bigg)^{\frac{1}{2p}}
\end{align}
hold.
\end{corollary}

\noindent Since this corollary follows easily from the large-scale weighted Meyers estimate, we forego a formal proof --notice that to obtain the result one simply applies \eqref{Lemma_4_est_1} to $h^{\eps}$ and $\hat{h}^{\eps}$ and uses the definition of the weight \eqref{omega}, along with the property that $g$ is supported in $B_r(x_0)$.

To obtain satisfactory control of $\nabla v^{\eps}_1$ and $\nabla v^{\eps}_2$, we require a better estimate for $\nabla(h^{\eps} - \hat h^{\eps})$ than that following from the triangle inequality and Corollary~\ref{sens_Cor1}. Towards this end, we prove the following:

\begin{lemma}\label{sens_Lm2}
%Let $\langle\cdot\rangle$ be an ensemble of coefficient fields on $\Rd^d$ that satisfies the assumptions (A1)-(A4). Furthermore, let $x_0 \in \domain$ and $r>\varepsilon$ be such that $\dist(B_r(x_0), \partial(\domain)) \geq \frac{r}{4}$, $g\in L^2(\domain)$ with $\supp(g) \subset B_r(x_0)$, 
Adopt the assumptions of Proposition \ref{sens_Prop1}, and let $R := \dist(x,\partial\domain) + r + r^{\domain}_*(x_0)$. Let $h^{\eps}$ and $\hat{h}^{\eps}$ solve \eqref{Definitionh} and \eqref{Definitionhath} respectively. Then there exists $\bar{p}>1$ such that, for all $1 \leq p < \bar{p}$ and $0 \leq \alpha_0 < \alpha_1 < d(2p-1)$ additionally satisfying $\alpha_1 >2p$, the estimate
\begin{align}
\label{Meyers3}
\begin{split}
& \bigg( \int_{\domain}  |\nabla (h^{\eps}-\hat h^{\eps})|^{2p}  \omega_{\alpha_0,R} 
\dx  \bigg)^{\frac{1}{2p}}
\\&
\lesssim_{d,\lambda,p,\domain, \alpha_0, \alpha_1}
\bigg( \frac{\dist(x_0,\partial\domain)}{r + r^{\domain}_*(x_0)} + 1 \bigg)^{-\frac{\alpha_1}{2p}} \bigg( \int_{B_r(x_0)} |\eps g|^{2p} \dx \bigg)^{\frac{1}{2p}}
\end{split}
\end{align} 
holds.
\end{lemma}

 We make use of Lemma \ref{sens_Lm2}, which is proven at the end of this section, in the form of the next corollary. In particular, combining the estimate~\eqref{Meyers3} on $ \nabla (h^{\eps} - \hat h^{\eps})$ inside $\domain$ with \eqref{Meyers2} for $\nabla \hat h^{\eps}$ outside of $\domain$, we obtain:

\begin{corollary}\label{sens_Cor2}
Adopt the assumptions of Proposition \ref{sens_Prop1}, let $R := \dist(x,\partial\domain) + r + r^{\domain}_*(x_0)$, and $v_1^{\eps}$ and $v_2^{\eps}$ solve \eqref{sens_def_v1} and \eqref{sens_def_v2} respectively. Then, for  $1 \leq p < \bar{p}$ and $0 \le \alpha_0 < \alpha_1 -2p <\alpha_1 < d(2p-1)$, the estimate 
 \begin{align}
 \label{sens_v1v2est}
 &\bigg( \int_{\Rd^d}|\nabla v^{\eps}_i|^{2p} \omega_{\alpha_0,R}
 \dx \bigg)^{\frac{1}{2p}}
 \\ &\nonumber
 \lesssim_ {d,\lambda,p, \domain,\alpha_0, \alpha_1, f} \bigg( \frac{\dist(x_0,\partial\domain)}{r + r^{\domain}_*(x_0)} + 1 \bigg)^{\frac{- \alpha_1}{2p}} \bigg( \int_{B_r(x_{0})} |g|^{2p} \dx \bigg)^{\frac{1}{2p}}
\end{align}
holds for $i=1,2$. In fact, in the case that $i=2$, \eqref{sens_v1v2est} holds for $0 \le \alpha_0 <\alpha_1 < d(2p-1)$.
\end{corollary}

\begin{proof}
We may assume w.\,l.\,o.\,g.\ that $R \leq \diam(\domain)$.

We then first consider the case that $i =1$. To obtain \eqref{sens_v1v2est}, we apply the whole-space version of the weighted Meyers estimate \eqref{Lemma_4_est_2} to $v^{\eps}_1$ solving \eqref{sens_def_v1}. This yields that
\begin{align*}
&\bigg( \int_{\Rd^d} |\nabla v^{\eps}_1|^{2p}  \omega_{\alpha_0,R} \dx \bigg)^{\frac{1}{2p}} \lesssim \frac{R}{\eps}  \bigg( \int_{\Rd^d} | \nabla f \cdot  a^\transpose\big(\tfrac{\cdot}{\eps}\big) (\chi_\domain \nabla h^{\eps} - \nabla \hat h^{\eps}))|^{2p}  \omega_{\alpha_1,R} \dx \bigg)^{\frac{1}{2p}}
\end{align*}
for $1< p \leq \bar{p}$ and $0<  \alpha_0 < \alpha_1 -2p < \alpha_1 < d(2p-1)$. In the case that $\domain = \Hd$, the right-hand side vanishes since $f \equiv 1$. When alternatively $\domain$ is a bounded $C^{1,1}$ -domain, \eqref{sens_v1v2est} follows from \eqref{Meyers2}, \eqref{Meyers3}, and that $|\nabla f | \lesssim 1$. 

The estimate for $i =2$ is obtained in essentially the same way, with the application of \eqref{Lemma_4_est_1} to $v^{\eps}_2$ solving \eqref{sens_def_v2}. Since we use \eqref{Lemma_4_est_1} instead of \eqref{Lemma_4_est_2}, the result \eqref{sens_v1v2est} holds for $0 \leq \alpha_0 < \alpha_1 < d(2p-1)$.
\end{proof}

To complete this section, we now give the proof of Lemma \ref{sens_Lm2}. For this we apply the weighted Meyers estimate to $h^{\eps} - \eta \hat{h}^{\eps}$, where the cut-off $\eta$ ensures homogeneous Dirichlet boundary data on $\domain$. Here is the argument:

\begin{proof}[Proof of Lemma \ref{sens_Lm2}]
Let $L = \dist(x_0, \partial \domain)$. Since \eqref{Meyers3} follows from \eqref{Meyers1} and \eqref{Meyers2} by the triangle inequality when $L \le 2(r^{\domain}_*(x_0) + r)$, we may assume that $ L \ge 2(r^{\domain}_*(x_0) + r)$. 

The main idea of our argument is to split $h^{\eps}-\hat h^{\eps}$ into a function which vanishes on $\partial\domain$ and a remainder, and show that~\eqref{Meyers3} holds for both of these functions. To this end, let $\hat \eta : [0,\infty) \to [0,1]$ be a smooth function satisfying $\hat \eta(t) = 1$ for $t \in [0,\frac{1}{2}]$, $\hat \eta(t) = 0$ for $t \ge \frac{5}{8}$, and $|\hat \eta'| \le 9$, and use it to define a cutoff function $\eta(x) := \hat \eta ( \frac{ |x-x_0| }{L} )$. Since $L \ge 2(r^{\domain}_*(x_0) + r)$ and $g$ is supported in $B_r(x_0)$, we see that $\eta \equiv 1$ on the support of $g$; furthermore, $\eta = 0$ in a neighborhood of the boundary $\partial\domain$ and $\supp(\nabla \eta) \subseteq B_{\frac{3L}{4}}(x_0) \setminus B_{\frac{L}{2}}(x_0)$.

%Moreover we see that $\nabla \eta \neq 0$ only in the set $\{ x \in \domain : 2 \dist (x,\partial\domain) \le |x-x_0| \le 3 \dist (x,\partial\domain) \}$, and also $|\nabla \eta| \lesssim |x-x_0|^{-1}$ there. 
 
Defining $w^{\eps} := h^{\eps} - \eta \hat h^{\eps}$ and using the equations \eqref{Definitionh} and \eqref{Definitionhath}, we find that
\begin{align*}
  -\nabla \cdot ((a^\eps)^\transpose \nabla w^{\eps})&= -\nabla \cdot  (a^\transpose\big(\tfrac{\cdot}{\eps}\big) \nabla h^{\eps})  + \nabla \cdot  (a^\transpose\big(\tfrac{\cdot}{\eps}\big) \hat h^{\eps} \nabla \eta) +  \nabla \cdot  (a^\transpose\big(\tfrac{\cdot}{\eps}\big)  \eta \nabla \hat h^{\eps}) 
 \\
 &= \nabla \cdot (\eps g) + \nabla \cdot ( \hat h^{\eps} a^\transpose\big(\tfrac{\cdot}{\eps}\big) \nabla \eta) + \eta ( - \nabla \cdot (\eps g) ) + \nabla \eta \cdot a^\transpose\big(\tfrac{\cdot}{\eps}\big) \nabla \hat h^{\eps} 
 \\
 &= \nabla \cdot ( \hat h^{\eps} a^\transpose\big(\tfrac{\cdot}{\eps}\big) \nabla \eta) + \nabla \eta \cdot a^\transpose\big(\tfrac{\cdot}{\eps}\big) \nabla \hat h^{\eps} \\
 & =  \nabla \cdot ( \hat h^{\eps} a^\transpose\big(\tfrac{\cdot}{\eps}\big) \nabla \eta + K^{\eps} )  \qquad \text{on } \domain,
\end{align*}
where $K^{\eps} = \chi_{B_{2L}(x_0)} \nabla k^{\eps} $, with $k^{\eps}$ being the solution of $\Delta k^{\eps} = \nabla \eta \cdot a^\transpose\big(\tfrac{\cdot}{\eps}\big) \nabla \hat h^{\eps}$ on $\domain \cap B_{2L}(x_0)$ with homogeneous Dirichlet boundary data on $\partial \domain \cap B_{2L}(x_0)$ and homogeneous Neumann boundary data on $\domain \cap \partial B_{2L}(x_0)$. Since both $h^{\eps} = 0$ and $\eta = 0$ on $\partial\domain$, we see that $w^{\eps} = 0$ on $\partial\domain$ as well.

Applying \eqref{Lemma_4_est_1} to $w^{\eps}$ and using that $L \geq r^{\domain}_*(x_0)$, 
we have that 
\begin{align}
\label{Lemma_4_Lemma_5}
\begin{split}
&\bigg( \int_{\domain}  |\nabla w^{\eps}|^{2p} \omega_{\alpha_0, L} \dx \bigg)^{\frac{1}{2p}} \lesssim \bigg( \int_{\domain} |\hat h^{\eps} \nabla \eta|^{2p}  \omega_{\alpha_1,L}  \dx \bigg)^{\frac{1}{2p}} +\bigg( \int_{\domain} |K^{\eps}|^{2p}  \omega_{\alpha_1,L} \dx \bigg)^{\frac{1}{2p}} , 
\end{split}
\end{align}
for $1< p < \bar{p}$ and $0 \leq \alpha_0 < \alpha_1 < d(2p-1)$.  

We then treat the two terms on the right-hand side of \eqref{Lemma_4_Lemma_5} separately, beginning with the second term. For this, we first notice that the standard Meyers estimate applied to $k^{\eps}$ gives
\begin{align*}
\int_{ \domain \cap B_{2L}(x_0)} |\nabla k^{\eps}|^{2p} \dx  \lesssim \int_{\domain \cap B_{\frac{3L}{4}}(x_0) \setminus B_{\frac{L}{2}}(x_0)} |\nabla \hat{h}^{\eps}|^{2p}  \dx ,
\end{align*}
where we have used that $|\nabla \eta | \leq \frac{9}{L}$. Since $\omega_{\alpha_1, L} \sim 1$ in the support of $K^{\eps}$, this implies that  
\begin{align}
\label{lemma_5_new}
\begin{split}
&\bigg( \int_{\domain} |K^{\eps}|^{2p}  \omega_{\alpha_1, L }  \dx \bigg)^{\frac{1}{2p}} \lesssim \bigg(\int_{ B_{\frac{3L}{4}}(x_0) \setminus B_{\frac{L}{2}}(x_0)} |\nabla \hat{h}^{\eps}|^{2p}  \omega_{\alpha_1, L }   \dx  \bigg)^{\frac{1}{2p}}.
\end{split}
\end{align}
In order to estimate the first integral on the right-hand side of \eqref{Lemma_4_Lemma_5}, we again use $|\nabla \eta | \leq \frac{9}{L}$ and the definition of the weight \eqref{omega} to write
\begin{align}
\label{Lemma_5_intermediate_1}
\begin{split}
& \bigg( \int_{\domain} |\hat h^{\eps} \nabla \eta|^{2p}  \omega_{\alpha_1, L } \dx \bigg)^{\frac{1}{2p}}  \lesssim \frac{1}{L} \bigg( \int_{B_{\frac{3L}{4}}(x_0) \setminus B_{\frac{L}{2}}(x_0)} |\hat h^{\eps}|^{2p}  \omega_{\alpha_1-2p,  L } \dx \bigg)^{\frac{1}{2p}}.
\end{split}
\end{align}
To further process the right-hand side we show that 
\begin{align}
\label{sens_eq05}
\begin{split}
& \bigg( \int_{ B_{\frac{3L}{4}}(x_0) \setminus B_{\frac{L}{2}}(x_0) } |\hat h^{\eps}|^{2p} \omega_{\alpha_1 - 2p, L }  \dx \bigg)^{\frac{1}{2p}} \lesssim L    \bigg( \int_{\Rd^d \setminus B_{\frac{L}{2}}(x_0) } |\nabla \hat h^{\eps}|^{2p} \omega_{\alpha_1, L }  \dx   \bigg)^{\frac{1}{2p}},
 \end{split}
\end{align}
which, when combined with \eqref{Lemma_5_intermediate_1}, yields
\begin{align}
\label{Lemma_5_intermediate_2}
\begin{split}
& \bigg( \int_{\domain} |\hat h^{\eps} \nabla \eta|^{2p}  \omega_{\alpha_1, L } \dx \bigg)^{\frac{1}{2p}} \lesssim  \bigg( \int_{\Rd^d \setminus B_{\frac{L}{2}}(x_0) } |\nabla \hat h^{\eps}|^{2p} \omega_{\alpha_1, L }   \dx \bigg)^{\frac{1}{2p}}.
\end{split}
\end{align}

We now prove \eqref{sens_eq05}. Using that $ \omega_{\alpha_1 - 2p, L }  \sim 1$ on $B_{\frac{3L}{4}}(x_0) \setminus B_{\frac{L}{2}}(x_0)$ and the critical Sobolev embedding along with H\"{o}lder's inequality, we obtain 
\begin{align*}
 &\bigg( \int_{ B_{\frac{3L}{4}}(x_0) \setminus B_{\frac{L}{2}}(x_0) } |\hat h^{\eps}|^{2p} \omega_{\alpha_1 - 2p, L }  \dx  \bigg)^{\frac{1}{2p}}
  \lesssim  \bigg( \int_{ \Rd^d \setminus B_{\frac{L}{2}}(x_0) } |\nabla \hat h^{\eps}|^{\frac{2pd}{2p+d}}  \dx
 \bigg)^{\frac{2p+d}{2pd}}\\
 &~~~~~~~~~~~~~~~~~\lesssim  \bigg( \int_{ \Rd^d \setminus B_{\frac{L}{2}}(x_0) } |\nabla \hat h^{\eps}|^{2p} \omega_{\alpha_1, L} \dx
 \bigg)^{\frac{1}{2p}}  \bigg( \int_{ \Rd^d \setminus B_{\frac{L}{2}}(x_0) } |\omega_{\alpha_1, L}|^{-\frac{d}{2p}}  \dx
 \bigg)^{\frac{1}{d}}. 
\end{align*}
By the definition \eqref{omega} we see that 
\begin{align*}
 \bigg( \int_{ \Rd^d \setminus B_{\frac{L}{2}}(x_0) } |\omega_{\alpha_1, L}|^{-\frac{d}{2p}}   \dx
 \bigg)^{\frac{1}{d}} \sim L
\end{align*}
for any $\alpha_1>2p$.

Together \eqref{Lemma_4_Lemma_5}, \eqref{lemma_5_new}, and \eqref{Lemma_5_intermediate_2} give that 
\begin{align}
\label{Lemma_4_Lemma_5_2}
\begin{split}
&\bigg( \int_{\domain}   |\nabla w|^{2p}  \omega_{\alpha_0, L }  \dx \bigg)^{\frac{1}{2p}} \lesssim \bigg( \int_{\Rd^d \setminus B_{\frac{L}{2}}(x_0) } |\nabla \hat h^{\eps}|^{2p} \omega_{\alpha_1, L }   \dx   \bigg)^{\frac{1}{2p}},
\end{split}
\end{align}
for $1< p < \bar{p}$ and $0 \leq \alpha_0 < \alpha_1 < d(2p-1)$ with $\alpha_1>2p$ (observe that the latter two conditions require $d\geq 3$). Now, since $L \geq 2(r^{\domain}_*(x_0) + r)$, from \eqref{Meyers2} we obtain
\begin{align*}
& \bigg( \int_{\Rd^d \setminus B_{\frac{L}{2}}(x_0)} |\nabla \hat h^{\eps}|^{2p} \omega_{\alpha_1, L}   \dx   \bigg)^{\frac{1}{2p}} \lesssim \bigg(\frac{L}{r+r^{\domain}_*(x_0)} +1 \bigg)^{-\frac{\alpha_1}{2p}} \bigg( \int_{B_r(x_0)} |\eps g|^{2p} \dx  \bigg)^{\frac{1}{2p}}.
\end{align*}
Having made this observation, \eqref{Lemma_4_Lemma_5_2} becomes 
\begin{align}
\label{Lemma_4_Lemma_5_3}
\begin{split}
&\bigg( \int_{\domain} |\nabla w^{\eps}|^{2p} \omega_{\alpha_0, L } \dx \bigg)^{\frac{1}{2p}}   \lesssim \bigg(\frac{L}{r + r^{\domain}_*(x_0)} +1 \bigg)^{-\frac{\alpha_1}{2p}} \bigg( \int_{B_r(x_0)} |\eps g|^{2p} \dx \bigg)^{\frac{1}{2p}}.
\end{split}
\end{align}
To go from \eqref{Lemma_4_Lemma_5_3} to \eqref{Meyers3}, by the triangle inequality and by $w^{\eps} = h^{\eps} - \eta \hat h^{\eps}$ it is enough to show 
\begin{align*}
\begin{split}
& \bigg( \int_{\domain} |\nabla((1-\eta) \hat h^{\eps})|^{2p} \omega_{\alpha_0, L } \dx   \bigg)^{\frac{1}{2p}}\lesssim
\bigg( \frac{L}{r + r^{\domain}_*(x_0)} + 1 \bigg)^{-\frac{\alpha_1}{2p}} 
\bigg( \int_{B_r(x_0)} |\eps g|^{2p} \dx  \bigg)^{\frac{1}{2p}}.
\end{split}
\end{align*}
For this, we use that $\supp(1 - \eta)\subseteq \domain  \setminus B_{\frac{L}{2}}(x_0)$ and \eqref{Meyers2} to write 
\begin{align*}
\bigg( \int_{\domain} |(1-\eta) \nabla \hat h^{\eps}|^{2p} \omega_{\alpha_0,L} \dx \bigg)^{\frac{1}{2p}} & \leq C
\bigg( \int_{\domain  \setminus B_{\frac{L}{2}}(x_0)} |\nabla \hat h^{\eps}|^{2p} \omega_{\alpha_1,L} \dx \bigg)^{\frac{1}{2p}}\\
& \leq C
\bigg( \frac{L}{r + r^{\domain}_*(x_0)} + 1 \bigg)^{-\frac{\alpha_1}{2p}} 
\bigg( \int_{B_r(x_0)} |\eps g|^{2p} \dx  \bigg)^{\frac{1}{2p}}
\end{align*}
and we bound the term $\big( \int_{\domain} |\hat h^{\eps} \nabla(1-\eta)|^{2p} \omega_{\alpha_0, L } \dx \big)^{\frac{1}{2p}}$ via \eqref{Lemma_5_intermediate_2} and \eqref{Meyers2}.
\end{proof}

\subsection{Proof of Proposition \ref{sens_Prop1}}
\label{proof_prop_5}

Applying the $L^{2q}$-version of the spectral gap inequality, \eqref{pLSI}, to $F^{\eps}$ as defined in \eqref{defn_F} and using the expression for the vertical derivative $\frac{\partial F^{\eps}}{\partial a^{\eps}}$ derived in \eqref{MalliavinDerivative}, we find that 
\begin{align}
\label{pLSI_new}
\begin{split}
&\en{ (F^{\eps} - \en{F^{\eps}})^{2q}}^{\frac{1}{2q}}
\\
& \lesssim q^{\frac{1}{2}} \varepsilon^{\frac{d}{2}}  \en{ \bigg(  \int_{\Rd^d}  \bigg( \fint_{B_{\varepsilon}(x)} \bigg| \frac{\partial F^{\eps}}{ \partial a^{\eps}}  \bigg| \dd y \bigg)^{2} \dx \bigg)^q  }^{\frac{1}{2q}}\\
& \lesssim q^{\frac{1}{2}} \varepsilon^{\frac{d}{2}}  \Bigg[\en{  \bigg(\int_{\Rd^d}  \bigg( \fint_{B_{\varepsilon}(x)} \big| \ \chi_\domain \nabla u_{\eps} \otimes \nabla h^{\eps} \big| \dd y \bigg)^{2}  \dx \bigg)^q }^{\frac{1}{2q}} \\
& \qquad  \qquad \quad  +  \en{  \bigg(\int_{\Rd^d}  \bigg( \fint_{B_{\varepsilon}(x)} \big|(e_i+\nabla\phi^{\eps}_i) \otimes (\nabla v^{\eps}_1 + \nabla v^{\eps}_2)\big| \dd y \bigg)^{2}   \dx \bigg)^q}^{\frac{1}{2q}} \Bigg].
\end{split}
\end{align}
With the help of the estimates on the auxiliary functions $h^{\eps}$, $v^{\eps}_1$, and $v^{\eps}_2$  that we have obtained in Section \ref{aux_estimates}, we now estimate the two terms on right-hand side of \eqref{pLSI_new}.

For the first term on the right-hand side of \eqref{pLSI_new}, we show the following bound:

\begin{lemma}\label{sens_Lm4} Adopt the assumptions of Proposition \ref{sens_Prop1}, let $h^{\eps}$ solve \eqref{Definitionh}, and define $L := \dist(x_0,\partial\domain)+\varepsilon$. Then, for $q \ge \frac{p}{4(p-1)}$, we have that 
 \begin{align}
  \label{sens_M1}
 \begin{split}
 &    \en{  \bigg(\int_{\Rd^d}  \bigg( \fint_{B_{\varepsilon}(x)} \big| \ \chi_\domain \nabla u_{\eps} \otimes \nabla h^{\eps} \big| \dd y \bigg)^{2} \dx  \bigg)^q }^{\frac{1}{2q}}   \\
 &  \lesssim_{d, \lambda, \domain, f, p, \beta} q^{C(d, \lambda, \domain)}
\bigg(
r^{-d/2} 
 \bigg( \frac{r}{L}\bigg)^{(1 - \beta) \frac d2}
+
 \frac{\varepsilon}{L}  r^{-\frac{d}{2}}
\bigg) 
\en{ \bigg( \frac{r^{\domain}_*(x_0)}{\varepsilon} + 1 \bigg)^{\frac{4qd(p( 2- \beta) - 1)}{p}}}^{\frac{1}{4q}},
     \end{split}
 \end{align}
 for $0< \beta \ll 1$.
\end{lemma}

To prove Lemma \ref{sens_Lm4}, we require the bounds on $h^{\eps}$ provided by Corollary \ref{sens_Cor1} as well as the following two technical results. The first of these is a simple consequence of H\"older's inequality:
\begin{lemma}\label{sens_Lmaux}
For any exponent $p>1$, any pair of non-negative functions $\varphi$ and $\psi$, and any positive weight function $\omega(x)$ with $\operatorname{osc}_{B_\eps(x)} \omega \leq \frac{1}{2}\omega(x)$ we have that
\begin{equation}\label{sens_eqaux}
   \int_{\Rd^d}\bigg(\fint_{B_{\varepsilon}(x)} \varphi\psi \, \dd y  \bigg)^2  \dx \lesssim 
   \bigg( \int_{\Rd^d} |\psi|^{2p}\omega \dx \bigg)^{\frac1p} \bigg( \int_{\Rd^d}  \omega^{-\frac{1}{p-1}} \bigg(\fint_{B_{\varepsilon}(x)}|\varphi|^2  \dd y \bigg)^\frac{p}{p-1} \dx \bigg)^{\frac{p-1}{p}}.
  \end{equation}
\end{lemma}

\noindent The second technical result is a decay estimate on the $q$-th moment of $|\nabla u_{\eps}(x_0)|$ in terms of $\dist(x_0, \partial \domain)$, which we obtain by applying \eqref{PoissonKernelEstimateAHarmonic} to $u_{\eps}$ and taking the expectation. 

\begin{lemma}\label{sens_LmEstu}
Under the assumptions of Proposition \ref{sens_Prop1}, we have that 
 \begin{align} 
 \label{sens_u01}
\en{ |\nabla u_{\eps} (x_0)|^q}^{\frac 1q} \lesssim_{d, \lambda, \domain,f,\delta} \frac{q^{C(d, \lambda, \domain, f)} }{\dist(x_0,\partial\domain) + \varepsilon} \bigg(\frac{\dist(x_0,\partial\domain)+\eps}{\dist(x_0,\partial\domain)}\bigg)^\delta
\end{align}
for any $\delta>0$ and any $q \geq 1$.
 \end{lemma}
\noindent We prove both Lemmas \ref{sens_Lmaux} and \ref{sens_LmEstu} in Section \ref{aux_lems_fluc}.\\

For now taking for granted Lemmas \ref{sens_Lmaux} and \ref{sens_LmEstu}, here is the argument for Lemma~\ref{sens_Lm4}:

\begin{proof}[Proof of Lemma \ref{sens_Lm4}]

To obtain \eqref{sens_M1} we separately consider the ``inner" and ``outer" contributions of the integrand, corresponding to the domains $B_{L/2}(x_0)$ respectively $\domain \setminus B_{L/2}(x_0)$.

Starting with the inner contribution and, using~\eqref{sens_eqaux} of Lemma \ref{sens_Lmaux} with $\psi= |\nabla h^{\eps}|$, $\varphi = |\chi_{\domain \cap B_{\frac{L}{2}} (x_0)} \nabla u_{\eps}|$, and $\omega = \omega_{\gamma,r + r^{\domain}_*(x_0)}$ for $\gamma > 0$ to be chosen later, we obtain that
\begin{align*}%\label{sens_eq07}
& \int_{\Rd^d}\bigg( \fint_{B_{\varepsilon}(x)} \big|  \chi_{\domain \cap B_{\frac{L}{2}} (x_0)} \nabla u_{\eps} \otimes \nabla h^{\eps} \big| \dd y \bigg)^{2} \dx
\\&
\leq \bigg( \int_{\Rd^d} |\nabla h^{\eps}|^{2p} \omega_{\gamma,r+r^{\domain}_*(x_0)} \dx   \bigg)^{\frac1p}
    \\&~~~~~~~~~~\times
\bigg( \int_{\Rd^d} \omega_{\gamma, r + r^{\domain}_*(x_0)}^{-\frac{1}{p-1}} \bigg(\fint_{B_{\varepsilon}(x)} |\chi_{\domain \cap B_{\frac{L}{2}} (x_0)} \nabla u_{\eps}|^2 \ \dd y \bigg)^\frac{p}{p-1} \dx \bigg)^{\frac{p-1}{p}}.
\end{align*}
Taking the $q$-th moment of both sides, we again use H\"older's inequality to the extent of
\begin{align*}
 &\en{\bigg( \int_{\Rd^d}\bigg( \fint_{B_{\varepsilon}(x)} \big|  \chi_{\domain \cap B_{\frac{L}{2}} (x_0)} \nabla u_{\eps} \otimes \nabla h^{\eps} \big| \dd y \bigg)^{2} \dx \bigg)^q }^{\frac{1}{2q}}\\
  %  &\le 
   %\en{
   % \bigg( \int_{\Rd^d} |\nabla h|^{2p} \omega_{\gamma,r+r^*}  \dx  \bigg)^{\frac{q}{p}} \bigg( \int_{\Rd^d} \omega_{\gamma,\r}^{-\frac{1}{p-1}} \bigg(\fint_{B_{\varepsilon}(x)} |\chi_{\domain \cap B_{\frac{L}{2}} (x_0)} \nabla u|^2 \ \dd y \bigg)^\frac{p}{p-1}  \dx \bigg)^{\frac{q(p-1)}{p}} }^{\frac{1}{2q}}
%\\
&   \leq  \en{
    \bigg( \int_{\Rd^d} |\nabla h^{\eps}|^{2p} \omega_{\gamma,r+r^{\domain}_*(x_0)}  \dx \bigg)^{\frac{2q}{p}} }^{\frac{1}{4q}}
    \\&~~~~~~~~~~\times
    \en{ \bigg( \int_{\Rd^d} \omega_{\gamma,r+r^{\domain}_*(x_0)}^{-\frac{1}{p-1}} \bigg(\fint_{B_{\varepsilon}(x)} |\chi_{\domain \cap B_{\frac{L}{2}} (x_0)} \nabla u_{\eps}|^2 \ \dd y \bigg)^\frac{p}{p-1}  \dx\bigg)^{\frac{2q(p-1)}{p}} }^{\frac{1}{4q}}.
\end{align*}
Using the definition of the weights \eqref{omega}, that $r\geq \varepsilon$, and applications of H\"{o}lder's and Jensen's inequalities (with the conditions $q \geq p/ 4(p-1)$ and $\gamma > d(p-1)$), we treat the second term on the right-hand side as
\begin{align*}
& \en{ \bigg( \int_{\Rd^d} \omega_{\gamma,r+r^{\domain}_*(x_0)}^{-\frac{1}{p-1}} \bigg(\fint_{B_{\varepsilon}(x)} |\chi_{\domain \cap B_{\frac{L}{2}} (x_0)} \nabla u_{\eps}|^2 \ \dd y \bigg)^\frac{p}{p-1} \dx \bigg)^{\frac{2q(p-1)}{p}} }^{\frac{1}{4q}}
\\
& \lesssim  \en{ \bigg( \int_{\Rd^d} \omega_{-\frac{\gamma}{p-1},r  } \bigg(1 + \frac{r^{\domain}_*(x_0)}{\varepsilon}\bigg)^{\frac{\gamma}{p-1}}\bigg(\fint_{B_{\varepsilon}(x)} |\chi_{\domain \cap B_{\frac{L}{2}} (x_0)} \nabla u_{\eps}|^2 \ \dd y \bigg)^\frac{p}{p-1} \dx \bigg)^{\frac{2q(p-1)}{p}} }^{\frac{1}{4q}}
\\
&\leq r^{\frac{d(p-1)}{2p}} \en{\bigg(1 + \frac{r^{\domain}_*(x_0)}{\varepsilon}\bigg)^{\frac{4\gamma q}{p}}}^{\frac{1}{8 q}} 
    \\&~~~~~~~~~~\times \en{ \bigg( r^{-d} \int_{\Rd^d} \omega_{-\frac{\gamma}{p-1},r  } \bigg(\fint_{B_{\varepsilon}(x)} |\chi_{\domain \cap B_{\frac{L}{2}} (x_0)} \nabla u_{\eps}|^2 \ \dd y \bigg)^\frac{p}{p-1} \dx \bigg)^{\frac{4q(p-1)}{p}} }^{\frac{1}{8q}}
\\
& \lesssim r^{\frac{d(p-1)}{2p}} \en{\bigg(1 + \frac{r^{\domain}_*(x_0)}{\varepsilon}\bigg)^{\frac{4\gamma q}{p}}}^{\frac{1}{8 q}}
    \\&~~~~~~~~~~\times
\en{  r^{-d} \int_{\Rd^d} \omega_{-\frac{\gamma}{p-1},r   } \bigg(\fint_{B_{\varepsilon}(x)} |\chi_{\domain \cap B_{\frac{L}{2}} (x_0)} \nabla u_{\eps}|^2 \ \dd y \bigg)^{4q} \dx }^{\frac{1}{8q}}.
\end{align*}
Combining the previous estimates with \eqref{sens_u01} from Lemma \ref{sens_LmEstu} gives that
\begin{align}
\label{inner_contribution_intermediate}
\begin{split}
 &\en{\bigg( \int_{\Rd^d}\bigg( \fint_{B_{\varepsilon}(x)} \big|  \chi_{\domain \cap B_{\frac{L}{2}} (x_0)} \nabla u_{\eps} \otimes \nabla h^{\eps}  \big| \dd y \bigg)^{2} \dx \bigg)^q }^{\frac{1}{2q}}\\
   &   \lesssim q^{C} r^{\frac{d}{2}} L^{-1} \en{
    \bigg( r^{-d}\int_{\Rd^d} |\nabla h^{\eps}|^{2p} \omega_{\gamma,r + r^{\domain}_*(x_0)} \dx \bigg)^{\frac{2q}{p}} }^{\frac{1}{4q}}  \en{\bigg(1 + \frac{r^{\domain}_*(x_0)}{\varepsilon}\bigg)^{\frac{4\gamma q}{p}}}^{\frac{1}{8 q}},
    \end{split}
\end{align}
notice that here we have used that we are on the ``inner contribution". To finish this estimate, we use \eqref{Meyers1} to write
\begin{align*}
\bigg( r^{-d} \int_{\domain} |\nabla h^{\eps}|^{2p} \omega_{\gamma,r + r^{\domain}_*(x_0)} \dx \bigg)^{\frac{1}{2p}} \lesssim \bigg( \fint_{B_r(x_0)} |\eps g|^{2p} \, \dx\bigg)^{\frac{1}{2p}},
\end{align*}
where we require that $\gamma < d(2p-1)$. Plugging this into \eqref{inner_contribution_intermediate} yields
\begin{align}
\label{icf}
\begin{split}
&  \en{\bigg( \int_{\Rd^d}\bigg( \fint_{B_{\varepsilon}(x)} \big|  \chi_{\domain \cap B_{\frac{L}{2}} (x_0)} \nabla u_{\eps} \otimes \nabla h^{\eps} \big| \dd y \bigg)^{2} \dx \bigg)^q }^{\frac{1}{2q}}\\
 &  \lesssim q^{C} r^{\frac{d}{2}}  \frac{\varepsilon}{L} \en{\bigg(1 + \frac{r^{\domain}_*(x_0)}{\varepsilon}\bigg)^{\frac{4\gamma q}{p}}}^{\frac{1}{8 q}} \bigg( \fint_{B_r(x_0)} |g|^{2p} \, \dx\bigg)^{\frac{1}{2p}},
 \end{split}
\end{align}
where the condition on $\gamma$ is that  $ d(p - 1) < \gamma < d(2p-1)$.

The estimate for the outer contribution proceeds in a similar fashion.
Again applying \eqref{sens_eqaux} of Lemma \ref{sens_Lmaux} gives, using also H\"{o}lder's inequality,
\begin{align*}
& \en{ \bigg( \int_{\Rd^d} \bigg( \fint_{B_{\varepsilon}(x)} | \chi_{\domain \setminus B_{\frac L2}(x_0)} \nabla u_{\eps} \otimes \nabla h^{\eps}| \, \dd y  \bigg)^2 \dx \bigg)^{q} }^\frac{1}{2q}
\\
&\le 
\en{ \bigg( \int_{\domain \setminus B_{\frac L2}(x_0)+B_\eps} |\nabla h^{\eps}|^{2p} \omega_{\gamma,L + r^{\domain}_*(x_0)} \dx \bigg)^{\frac{2q}{p}} }^{\frac{1}{4q}} \\
& \hspace{2cm} \times \en{ \bigg( \int_{\domain \setminus B_{\frac L2}(x_0)+B_\eps} \omega_{\gamma,L+r^{\domain}_*(x_0)}^{-\frac{1}{p-1}}  \bigg(\fint_{B_{\varepsilon}(x)} \chi_\domain |\nabla u_{\eps}|^2 \, \dd y \bigg)^\frac{p}{p-1} \dx \bigg)^\frac{ 2q(p-1)}{p} }^{\frac{1}{4q}}.
\end{align*}
Writing $L + r^{\domain}_*(x_0) =  L(\frac{r^{\domain}_*(x_0)}{L}+1)$ and using that $L \geq \varepsilon$, we get by H\"older's inequality 
\begin{align*}
 &   \en{ \bigg( \int_{\domain \setminus B_{\frac L2}(x_0)+B_\eps} \omega_{-\frac{\gamma }{p-1},L + r^{\domain}_*(x_0)} \bigg(\fint_{B_{\varepsilon}(x)} \chi_\domain | \nabla u_{\eps}|^2 \, \dd y \bigg)^\frac{ p}{p-1} \dx \bigg)^\frac{ 2q(p-1)}{p} }^{\frac{1}{4q}}
\\
 &  \lesssim  \en{ \bigg( \int_{\domain \setminus B_{\frac L2}(x_0)+B_\eps} \omega_{-\frac{\gamma }{p-1},L } \bigg(1 + \frac{r^{\domain}_*(x_0)}{\varepsilon}\bigg)^{\frac{\gamma}{p-1}} \bigg(\fint_{B_{\varepsilon}(x)} \chi_\domain | \nabla u_{\eps}|^2 \, \dd y \bigg)^\frac{ p}{p-1} \dx \bigg)^\frac{ 2q(p-1)}{p} }^{\frac{1}{4q}}
\\
  &  \leq L^{\frac{d(p-1)}{2p}}   \en{ \bigg(1 + \frac{r^{\domain}_*(x_0)}{\varepsilon}\bigg)^{\frac{4 q \gamma}{p}}}^{\frac{1}{8q}}
 \\
  & \qquad \times\en{  \bigg( L^{-d}  \int_{\domain \setminus B_{\frac L2}(x_0)+B_\eps}    \omega_{-\frac{\gamma }{p-1},L }   \bigg( \fint_{B_{\varepsilon}(x)} \chi_\domain |\nabla u_{\eps}|^2 \, \dd y \bigg)^\frac{ p}{p-1} \dx \bigg)^\frac{ 4q(p-1)}{p} }^{\frac{1}{8q}}.
\end{align*}
We then apply Jensen's inequality and inject the estimate 
\begin{align*}
\en{ \bigg(\fint_{B_{\varepsilon} (x)} \chi_\domain |\nabla u_{\eps}|^2 \, \dd y \bigg)^{4q} }^{\frac{1}{8q}} \lesssim q^{C} \eps^{-1}
\end{align*}
(which follows from \eqref{sens_u01} of Lemma~\ref{sens_LmEstu}), which yields 
\begin{align*}
 &   \en{ \bigg( \int_{\domain \setminus B_{\frac L2}(x_0)+B_\eps} \omega_{-\frac{\gamma }{p-1},L + r^{\domain}_*(x_0)} \bigg(\fint_{B_{\varepsilon}(x)} \chi_\domain | \nabla u_{\eps}|^2 \, \dd y \bigg)^\frac{p}{p-1} \dx \bigg)^\frac{ 2q(p-1)}{p} }^{\frac{1}{4q}}
\\
%   &  \leq L^{\frac{d(p-1)}{2p}}   \en{ (1 + \frac{\rx}{\varepsilon})^{\frac{4 q \gamma}{p}}}^{\frac{1}{8q}} \en{   L^{-d}  \int_{\domain \setminus B_{\frac L2}(x_0)}    \omega_{\frac{-\gamma }{p-1},L }   \bigg(  \fint_{B_{\varepsilon}(x)} |\nabla u^T|^2 \, \dd y \bigg)^{4q}  }^{\frac{1}{8q}}
\\&
\lesssim  q^{C} \eps^{-1}  L^{\frac{d(p-1)}{2p}}   \en{ \bigg(1 + \frac{r^{\domain}_*(x_0)}{\varepsilon}\bigg)^{\frac{4 q \gamma}{p}}}^{\frac{1}{8q}}.
\end{align*}
Combining the above estimates, we obtain that
\begin{align}
\label{april_1_1}
\begin{split}
& \en{ \bigg( \int_{\Rd^d} \bigg( \fint_{B_{\varepsilon}(x)} | \chi_{\domain \setminus B_{\frac L2}(x_0)} \nabla u_{\eps} \otimes \nabla h^{\eps}| \, \dd y  \bigg)^2 \dx \bigg)^{q} }^\frac{1}{2q}\\
  &\lesssim 
q^{C} \eps^{-1}  L^{\frac{d}{2}}  \en{ \bigg( L^{-d} \int_{\domain \setminus B_{\frac L2}(x_0)+B_\eps} |\nabla h^{\eps}|^{2p} \omega_{\gamma,  L + r^{\domain}_*(x_0)} \dx \bigg)^{\frac{2q}{p}} }^{\frac{1}{4q}} \\&~~~~~\times \en{ \bigg(1 + \frac{r^{\domain}_*(x_0)}{\varepsilon}\bigg)^{\frac{4 q \gamma}{p}}}^{\frac{1}{8q}}.
 \end{split}
\end{align}
From \eqref{Meyers1} we infer, using that $r\geq \varepsilon$,
\begin{align*}
 &\bigg( L^{-d} \int_{\domain \setminus B_{\frac{L}{2}}(x_0) + B_\eps} |\nabla h^{\eps}|^{2p} \omega_{\gamma, L + r^{\domain}_*(x_0)} \dx \bigg)^{\frac{1}{2p}}
\\
 &\lesssim_{\gamma} \bigg(\frac{r}{L}\bigg)^{\frac{d}{2p}} \bigg(\frac{r+r^{\domain}_*(x_0)}{L}\bigg)^{\frac{\gamma}{2p}} \bigg( \fint_{B_r(x_0)} |\eps g|^{2p} \dx \bigg)^{\frac{1}{2p}}\\
 & \lesssim \bigg(\frac{r}{L}\bigg)^{\frac{d + \gamma}{2p}} \bigg(1 + \frac{r^{\domain}_*(x_0)}{\varepsilon}\bigg)^{\frac{\gamma}{2p}} \bigg( \fint_{B_r(x_0)} |\eps g|^{2p} \dx \bigg)^{\frac{1}{2p}}.
\end{align*}
Plugging this into \eqref{april_1_1}, we obtain
\begin{align*}
& \en{ \bigg( \int_{\Rd^d} \bigg( \fint_{B_{\varepsilon}(x)} | \chi_{\domain \setminus B_{\frac L2}(x_0)} \nabla u_{\eps} \otimes \nabla h^{\eps}| \, \dd y  \bigg)^2 \dx \bigg)^{q} }^\frac{1}{2q}
\\&
\lesssim_{\gamma} q^{C} 
L^{\frac{d}{2}}  \bigg(\frac{r}{L}\bigg)^{\frac{d + \gamma}{2p}}
\en{ \bigg(1+ \frac{r^{\domain}_*(x_0)}{\varepsilon}\bigg)^{\frac{4q\gamma}{p}} }^{\frac{1}{4q}} \bigg( \fint_{B_r(x_0)} | g|^{2p} \dx \bigg)^{\frac{1}{2p}}.
\end{align*}
Choosing $\gamma := d(p(2  - \beta) -1 )$ and combining this inequality with \eqref{icf} and \eqref{g_avg}, we obtain \eqref{sens_M1} via the triangle inequality.
\end{proof}

To prove Proposition \ref{sens_Prop1} it remains to bound the second term on the right-hand side of \eqref{pLSI_new}. For this term we obtain:

\begin{lemma}\label{sens_Lm3}
Adopt the assumptions of Proposition \ref{sens_Prop1} and let $v^{\eps}_1$ and $v^{\eps}_2$ solve \eqref{sens_def_v1} and \eqref{sens_def_v2} respectively. Then, for $q \ge \frac{p}{2(p-1)}$, we have that
 \begin{align}
 \label{prop_5_second_piece}
 \begin{split}
& \en{  \bigg(\int_{\Rd^d}  \bigg( \fint_{B_{\varepsilon}(x)} \big|(e_i+\nabla \phi^{\eps}_i) \otimes (\nabla v^{\eps}_1 + \nabla v^{\eps}_2)\big| \dd y \bigg)^{2} \dx \bigg)^q}^{\frac{1}{2q}} \\
& \lesssim_{d,\lambda, \domain, f, \beta, p} r^{-\beta d/2} \bigg( \frac{1}{\dist(x_0,\partial\domain) + r} \bigg)^{(1-\beta)\frac{d}{2}} 
\en{\bigg(\frac{r^{\domain}_*(x_0)}{\varepsilon}+1\bigg)^{\frac{2qd(3p-2)}{p}   }  }^{\frac{1}{4q}}
\\&~~~~~~~~~~~~~~~~~~\times
 \en{ \bigg(\fint_{ B_{\varepsilon}(0)} |\nabla\phi^{\eps}_i+e_i|^2 \dx \bigg)^{2q}}^{\frac{1}{4q}},
 \end{split}
 \end{align} 
 for any $1 \gg \beta >0$. 
\end{lemma}
\noindent For the proof of this lemma we again use the H\"older-like estimate \eqref{sens_eqaux} from Lemma \ref{sens_Lmaux}, now in conjunction with Corollary \ref{sens_Cor2}:
%  R := \dist(x_0,\partial\domain) + r + \rx \quad \textrm{and} \quad \r := \dist(x_0,\partial\domain) + r 
\begin{proof}
Throughout this argument we assume that $p$, $\alpha_0$, and $\alpha_1$ are subject to the assumptions in Corollary \ref{sens_Cor2}, \ie  that $d(p-1)< \alpha_0 < \alpha_1 - 2p < \alpha_1 < d(2p-1)$. Furthermore, we use the convention $R:= \dist(x_0,\partial\domain) + r + r^{\domain}_*(x_0)$ and $L:=\dist(x_0,\partial\domain)$.
 
Using~\eqref{sens_eqaux} with $\psi=|\nabla(v^{\eps}_1 + v^{\eps}_2)|$, $\varphi = |e_i+\nabla\phi^{\eps}_i|$, and $\omega = \omega_{\alpha_0, R}$ we get 
\begin{align*}
 & \int_{\Rd^d} \bigg( \fint_{B_{\varepsilon}(x)} |(e_i+\nabla\phi^{\eps}_i) \otimes (\nabla v^{\eps}_1 + \nabla v^{\eps}_2)| \, \dd y \bigg)^2 \dx \\
  & \le 
  \bigg( \int_{\Rd^d} |\nabla (v^{\eps}_1 + v^{\eps}_2)|^{2p}\omega_{\alpha_0, R} \dx  \bigg)^{\frac{1}{p}} \bigg( \int_{\Rd^d} \omega_{\alpha_0, R}^{-\frac{1}{p-1}} \bigg(\fint_{B_{\varepsilon}(x)} |e_i+\nabla\phi^{\eps}_i|^2 \, \dd y  \bigg)^\frac{p}{p-1} \dx \bigg)^{\frac{p-1}{p}}.
 \end{align*}
 %
 %%%%%%%%%%%
 %%%%%%%%%%%
Taking the $q$-th moment of both sides, we infer by~\eqref{sens_v1v2est} and  \eqref{g_avg} that
\begin{align*}
& \en{\bigg( \int_{\Rd^d} \bigg( \fint_{B_{\varepsilon}(x)} |(e_i+\nabla\phi^{\eps}_i) \otimes (\nabla v^{\eps}_1 + \nabla v^{\eps}_2)| \, \dd y \bigg)^2 \dx  \bigg)^q }^{1/2q}
\\&
\lesssim_{\alpha_0, \alpha_1} 
\en{
\bigg( \frac{ L }{ r+r^{\domain}_*(x_0)} + 1 \bigg)^{-\frac{q \alpha_1}{p}}
\bigg(\int_{\Rd^d} \omega_{\alpha_0, R}^{-\frac{1}{p-1}} \bigg( \fint_{B_{\varepsilon}(x)} |e_i + \nabla\phi_i^\eps|^2 \, \dd y \bigg)^{\frac{p}{p-1}} \dx \bigg)^{\frac{q(p-1)}{p}}
}^{1/2q}
\\&~~~~~~~~~\times r^{\frac{d(1-2 p)}{2p}}.
\end{align*}
Taking into consideration the following two bounds
\begin{align*}
 \bigg( \frac{L}{r+r^{\domain}_*(x_0)} + 1 \bigg)^{-1} &\le  \bigg( \frac{L}{r}+1 \bigg)^{-1}  \bigg( \frac{r^{\domain}_*(x_0)}{\varepsilon} + 1 \bigg),
\\
 \bigg( \frac{|x-x_0|}{R} + 1 \bigg)^{-1} &\le \bigg(\frac{|x-x_0|}{L + r }+1 \bigg)^{-1} \bigg( \frac{r^{\domain}_*(x_0)}{\varepsilon} + 1 \bigg),
\end{align*}
both being a consequence of $r \geq \varepsilon$, and then applying H\"older's inequality we obtain 
\begin{align*}
\begin{split}
&\en{\bigg( \int_{\Rd^d} \bigg( \fint_{B_{\varepsilon}(x)} |(e_i+\nabla\phi_i^\eps) \otimes (\nabla v_1^\eps + \nabla v_2^\eps)| \, \dd y \bigg)^2 \dx  \bigg)^q }^{\frac{1}{2q}}\\
&
\lesssim_{\alpha_0, \alpha_1} r^{\frac{d (1- 2p)}{2p}} (L + r )^{\frac{d(p-1)}{2p}}
\bigg( \frac{L}{r}+1\bigg)^{-\frac{\alpha_1}{2p}}
\en{\bigg(\frac{r^{\domain}_*(x_0)}{\varepsilon}+1\bigg)^{\frac{2q(\alpha_0 + \alpha_1)}{p} }  }^{\frac{1}{4q}}
\\&\ \ \ \times
\en{\bigg( (L + r)^{-d}\int_{\Rd^d}\bigg( \frac{|x-x_0|}{L + r}+1 \bigg)^{-\frac{\alpha_0}{p-1}} \bigg( \fint_{B_{\varepsilon(x)} }|\nabla\phi_i^\eps+e_i|^2 \, \dd y \bigg)^{\frac{p}{p-1}} \dx \bigg)^{\frac{2q (p-1)}{p}}}^{\frac{1}{4q}}.
\end{split}
\end{align*}
Next we apply Jensen's inequality (using that $\alpha_0> d(p-1)$ and $\frac{2q(p-1)}{p}>1$), to the extent of  
\begin{align*}
\begin{split}
&\en{\bigg( \int_{\Rd^d} \bigg( \fint_{B_{\varepsilon}(x)} |(e_i+\nabla\phi_i^\eps) \otimes (\nabla v_1^\eps + \nabla v_2^\eps)| \, \dd y \bigg)^2 \dx  \bigg)^q }^{\frac{1}{2q}}\\
&
\lesssim_{\alpha_0, \alpha_1} r^{\frac{d (1- 2p)}{2p}} (L + r)^{\frac{d(p-1)}{2p}}
\bigg( \frac{L}{r}+1\bigg)^{-\frac{\alpha_1}{2p}}
\en{\bigg(\frac{r^{\domain}_*(x_0)}{\varepsilon}+1\bigg)^{\frac{2q(\alpha_0 + \alpha_1)}{p}   }  }^{\frac{1}{4q}}
\\&\ \ \ \times
\en{(L+ r)^{-d}\int_{\Rd^d}\bigg( \frac{|x-x_0|}{L+r}+1 \bigg)^{-\frac{\alpha_0}{p-1}} \bigg( \fint_{B_{\varepsilon(x)} }|\nabla\phi_i^\eps+e_i|^2 \, \dd y \bigg)^{2q} \dx \, }^{\frac{1}{4q}}.
\end{split}
\end{align*}
Using that, by the stationarity of $\nabla \phi_i^\eps$, 
\begin{align*}
\en{ \bigg(\fint_{B_{\varepsilon} (x)} |\nabla\phi_i^\eps+e_i|^2  \dd y \bigg)^{2q} }
= \en{ \bigg(\fint_{ B_{\varepsilon}(0)} |\nabla\phi_i^\eps+e_i|^2 \dx \bigg)^{2q}} 
\end{align*}
for any $x \in \Rd^d$, we find by setting  $\alpha_0 := d(p-1) +  \beta d p$ and  $\alpha_1 := d(2p-1) -  \beta d p$ that
\begin{align*}
&\en{\bigg( \int_{\Rd^d} \bigg( \fint_{B_{\varepsilon}(x)} |(e_i+\nabla\phi_i^\eps) \otimes (\nabla v_1^\eps + \nabla v_2^\eps)| \, \dd y \bigg)^2 \dx  \bigg)^q }^{\frac{1}{2q}}
\\&
\lesssim_{\alpha_0, \alpha_1}
r^{-\beta d/2}
\bigg( \frac{1}{L + r} \bigg)^{(1-\beta)\frac{d}{2}} 
\en{\bigg(\frac{r^{\domain}_*(x_0)}{\varepsilon}+1\bigg)^{\frac{2qd(3p-2)}{p}   }  }^{\frac{1}{4q}}
\\&~~~~~~~~~~~\times
 \en{ \bigg(\fint_{ B_{\varepsilon}(0)} |\nabla\phi_i^\eps+e_i|^2 \dx \bigg)^{2q}}^{\frac{1}{4q}}
\end{align*}
for $ 0 < \beta < \frac{1}{2}$. Note that to satisfy the condition $\alpha_0 < \alpha_1 - 2p$ one needs $d > 2$.
\end{proof}

With Lemmas~\ref{sens_Lm4} and \ref{sens_Lm3} in-hand, we are now ready to prove Proposition \ref{sens_Prop1}:

\begin{proof}[Proof of Proposition \ref{sens_Prop1}] 
Combining \eqref{pLSI_new} with \eqref{sens_M1} of Lemma \ref{sens_Lm4} and \eqref{prop_5_second_piece} of Lemma \ref{sens_Lm3}, we obtain
\begin{align*}
& \en{ (F - \en{F})^{2q}}^{\frac{1}{2q}} \\
& \lesssim q^{\frac{1}{2}} \varepsilon^{\frac{d}{2}}  \en{ \bigg(  \int_{\Rd^d}  \bigg( \fint_{B_{\varepsilon}(x)} \bigg| \frac{\partial F}{ \partial a}  \bigg| \dd y \bigg)^{2}  \dx \bigg)^q  }^{\frac{1}{2q}}\\
& \lesssim_{d, \lambda, \domain, f, \beta, p, f}  q^{C(d, \lambda, \domain)}
\bigg( 
\Big(\frac{\eps}{r}\Big)^{\beta d/2}
 \bigg( \frac{\varepsilon}{\dist(x_0,\partial\domain)}\bigg)^{(1 - \beta) \frac d2}
+
 \frac{\varepsilon}{\dist(x_0,\partial\domain)}  \Big(\frac{\eps}{r}\Big)^{\frac{d}{2}}
\bigg)\\
& \hspace{6cm} \times
\en{ \bigg( \frac{r^*_{\domain}(x_0)}{\varepsilon} + 1 \bigg)^{\frac{4qd(p( 2- \beta) - 1)}{p}}}^{\frac{1}{4q}}
\\&~~~~~~
+q^{C(d,\lambda,\domain)} \Big(\frac{\eps}{r}\Big)^{\beta d/2} \bigg( \frac{\varepsilon}{\dist(x_0,\partial\domain) } \bigg)^{(1-\beta)\frac{d}{2}}\\
& \hspace{2cm} \times
\en{\bigg(\frac{r^*_{\domain}(x_0)}{\varepsilon}+1\bigg)^{\frac{2qd(3p-2)}{p}   }  }^{\frac{1}{4q}}
 \en{ \bigg(\fint_{ B_{\varepsilon}(0)} |\nabla\phi_i^\eps+e_i|^2 \dx \bigg)^{2q}}^{\frac{1}{4q}}.
\end{align*}
Using that $r^*_\domain$ has stretched exponential moments (see \cite[Theorem 1 \& Theorem 2]{GNO_final} for $r^*_{\R^d}$, which by \cite{FischerRaithel} entails the corresponding result for $r_{\Hd}^*$; the case of bounded $C^{1,1}$-domains is treated in \cite[Theorem~3]{BFJR_1}) concludes the proof.
\end{proof}

%\red{In view of \eqref{}, making use of Lemmas~\ref{sens_Lm4} and \ref{sens_Lm3} yields:
%\begin{corollary}\label{sens_Cor3}
%Let $r$, $x_0$, $p$, and $g$ be as in Proposition~\ref{sens_Prop1}, where in addition we assume the normalization $\bigg( \fint_{B_r(x_0)} |g|^{2p} \bigg)^{\frac{1}{2p}} \le r^{-d}$. Then for any $q \ge 1$ 
%\begin{equation}
%\en{ \big|F-\mathbb{E}[F]\big|^{2q}}^{\frac{1}{2q}}
%\leq 
%Cq^\kappa \bigg( \bigg( \frac{\varepsilon}{\dist(x_0,\partial\domain)} \bigg)^{(1-\beta)\frac d2} + 
%	   \frac{\varepsilon}{\dist(x_0,\partial\domain)} \bigg( \frac{\varepsilon}{r} \bigg)^{(1-\beta)\frac d2}  \bigg),
%\end{equation}
%where the constant $C < \infty$ and exponent $\kappa < \infty$ depend on $d,\lambda,\domain,\beta,p$. 
%\end{corollary}}

%\red{Missing: statement giving estimate of $\tilde F$}

\subsection{Proofs of auxiliary lemmas for Proposition \ref{sens_Prop1}}
\label{aux_lems_fluc}

 We begin with the proof of Lemma \ref{sens_Lmaux}, which is an easy consequence of the H\"older inequality.
 
\begin{proof}[Proof of Lemma \ref{sens_Lmaux}]
We apply H\"{o}lder's inequality twice:
\begin{align*}
&\int_{\Rd^d}\bigg(\fint_{B_{\varepsilon}(x)} \varphi\psi \, \dd y  \bigg)^2 \dx \, 
\leq  \int_{\Rd^d} \bigg(\fint_{B_{\varepsilon}(x)} |\varphi|^2 \omega^{-\frac{1}{p}} \, \dd y  \bigg) \bigg(\fint_{B_{\varepsilon}(x)} |\psi|^2 \omega^{\frac{1}{p}} \, \dd y \bigg) \, \dx
\\&
\leq \bigg( \int_{\Rd^d} \bigg(\fint_{B_{\varepsilon}(x)} |\varphi|^2 \omega^{-\frac{1}{p}} \, \dd y \bigg)^{\frac{p}{p-1}} \dx \bigg)^{\frac{p-1}{p}}    \bigg( \int_{\Rd^d}  \bigg(\fint_{B_{\varepsilon}(x)} |\psi|^2 \omega^{\frac{1}{p}} \, \dd y \bigg)^{p}  \dx \bigg)^{\frac{1}{p}}.
\end{align*}
The claim then follows from Jensen's inequality.
% since 
%\begin{align}
% \bigg( \int_{\Rd^d}  \bigg(\fint_{B_{\varepsilon}(x)} |\psi|^2 \omega^{\frac{1}{p}} \, \dd y \bigg)^{p} \,  \bigg)^{\frac{1}{p}} \leq  \bigg( \int_{\Rd^d}\fint_{B_{\varepsilon}(x)} |\psi|^{2p} \omega\, \dd y \,  \bigg)^{\frac{1}{p}}.
%\end{align}
\end{proof}

We now prove Lemma \ref{sens_LmEstu}, for which we use a gradient estimate from \cite[Corollary~5]{BFJR_1}.
\begin{lemma}[{\cite[Corollary~5]{BFJR_1}}]
Let the assumptions (A1) -- (A4) be satisfied and $\domain$ be a bounded $C^{1,1}$ domain or $\Hd$, and $u \in H^1_{\loc}(\domain)$ a weak solution of
\begin{equation*}
\begin{aligned}
-\nabla \cdot (a\nabla u) &= \nabla \cdot g + f
&&\text{in }\domain,
\\
u&=0,
&&\text{on }\partial\domain
\end{aligned}
\end{equation*}
where for $\domain = \Hd$ the homogeneous Dirichlet data is replaced by a decay condition in the far-field. Then there exists a random field $\mathcal{C}(a,x)$ with stretched exponential moments in the sense of \eqref{stretched} such that
\begin{align}
\label{BoundSolutionRandomOperator}
|\nabla u(x_0)| 
\leq&
\mathcal{C}(a,x_0)
\int_{\domain} \mathcal{C}(a,x) \frac{|g| + |f| \dist(x,\partial\domain)}{| x_0- x|^d} \dx
\end{align}
holds for all $x_0 \in \domain$.
\end{lemma}

\begin{proof}[Proof of Lemma \ref{sens_LmEstu}]
We start with the case when $x_0$ is not close to the boundary, by assuming $\dist(x_0,\partial\domain) \geq 3\varepsilon$. Let $\eta$ be a smooth cut-off function for $\partial\domain$ in its $\varepsilon$-neighborhood, \ie a smooth function such that $\eta = 1$ on $\partial\domain$, $\eta(x) = 0$ if $\dist(x,\partial\domain) \ge \varepsilon$, and $|\nabla \eta| \le C(\domain)\varepsilon^{-1}$. Then, letting $w^{\eps} := u_{\eps} - \eta f\phi_i\big(\tfrac{\cdot}{\eps}\big) $, we find that  
 $-\nabla \cdot (a\big(\tfrac{\cdot}{\eps}\big) \nabla w^{\eps}) = \nabla \cdot (a\big(\tfrac{\cdot}{\eps}\big) \nabla (\eta f \phi_i\big(\tfrac{\cdot}{\eps}\big)))$ in $\domain$ and $w^{\eps} = 0$ on $\partial\domain$. 
 %We split the right-hand side into $\rho_0 := \nabla( \eta f \phi) \chi_{U_\varepsilon}$ and $\rho_1 := \nabla( \eta f \phi) \chi_{\domain \setminus U_\varepsilon}$, where $U_\varepsilon := \{ x \in \domain : \dist(x,U) < \varepsilon \}$, so that in particular $\dist(U,\supp \rho_1) \ge \varepsilon$. Correspondingly, we denote by $w_i$, $i=0,1$ solutions of $-\nabla \cdot a \nabla w_i = \nabla \cdot \rho_i$ in $\domain$ with $w = 0$ on $\partial\domain$.  We first estimate $\nabla w_1$.
By \eqref{BoundSolutionRandomOperator}, using the boundedness of $a\big(\tfrac{\cdot}{\eps}\big)$ and the definition of $\eta$ (which implies $u=w$ in a neighborhood of $x_0$) we get
 \begin{equation}
  |\nabla u_{\eps} (x_0)| \le \mathcal{C}(a,x_0) \int_{\domain} \mathcal{C}(a,x) \frac{|\nabla(\eta f \phi_i\big(\tfrac{\cdot}{\eps}\big))|  }{|x-x_0|^d} \dx.
 \end{equation}
 Using the notation $\domain_\varepsilon := \{ x \in \domain : \dist(x,\partial\domain) < \varepsilon\}$, we obtain using H\"older's inequality 
\begin{equation}\nonumber
\en{ |\nabla u_{\eps} (x_0)|^q }^{\frac 1q} \lesssim 
q^{C}
\en{ \bigg( \int_{\domain_\varepsilon} \mathcal{C}(a,x) \frac{|\nabla(\eta f \phi_i\big(\tfrac{\cdot}{\eps}\big))| }{|x-x_0|^d} \dx \bigg)^{2q} }^{\frac{1}{2q}}.
\end{equation}
By the corrector bounds from \cite[Theorem 2]{GNO5} and the Caccioppoli inequality we know that
\begin{align}
\label{BoundGradientCorrector}
\fint_{B_\eps(y_0)} |\nabla \phi_i\big(\tfrac{\cdot}{\eps}\big)|^2 \dx
\leq \mathcal{C}(a,y_0) .
\end{align}
Covering $\domain_\varepsilon$ by balls of radius $\eps$ (of which at most $C(d)$ cover any given point) and using $|\nabla(\eta f \phi_i\big(\tfrac{\cdot}{\eps}\big))|\lesssim \eps^{-1} |\nabla \phi_i\big(\tfrac{\cdot}{\eps}\big)| + \eps^{-1} |\phi_i\big(\tfrac{\cdot}{\eps}\big)|$, we deduce
\begin{equation}
\nonumber
\en{ |\nabla u_{\eps} (x_0)|^q }^{\frac 1q} \lesssim 
 \frac{q^{C} }{\eps+\dist(x_0,\partial\domain)}.
\end{equation}
It remains to prove the result in case $\dist(x_0,\partial\domain) \le 3\varepsilon$. Consider again $w^{\eps}:=u_{\eps} - \eta f\phi_i\big(\tfrac{\cdot}{\eps}\big)$. Since in this case $x_0$ is close to (or even contained in) the support of $\nabla(\eta f \phi_i\big(\tfrac{\cdot}{\eps}\big))$, we now define $w^{\eps}_0$ as the solution of $-\nabla \cdot (a^{\eps} \nabla w^{\eps}_0) = \nabla \cdot (a^{\eps} \nabla(\eta f \phi_i\big(\tfrac{\cdot}{\eps}\big))\chi_{B_{3\varepsilon}(x_0)}))$ with $w^{\eps}_0 = 0$ on $\partial\domain$, and $w^{\eps}_1 := w^{\eps} - w^{\eps}_0$. Combining the energy estimate for $w_0$ with~\eqref{BoundGradientCorrector} we get that 
\begin{align*}
\en{\bigg( \fint_{B_{\varepsilon}(x_0) \cap \domain} |\nabla w^{\eps}_0|^2 \dx \bigg)^q }^{\frac{1}{2q}} \lesssim q^C \eps^{-1}, 
\end{align*}
which by the Schauder estimate (see, \eg, \cite[Theorem~3]{BFJR_1}) and the regularity of the corrector $\phi$ contained in Lemma~\ref{CorrectorRegularity} implies that $\left< |\nabla w^{\eps}_0(x_0)|^q \right>^{\frac{1}{q}} \lesssim q^C \eps^{-1}$. Since $x_0$ is $3\varepsilon$ away from the support of $\chi_{B_{3\varepsilon}^c(x_0)}$, an argument similar to the first case gives that $\left< |\nabla w_1(x_0)|^q \right>^{\frac{1}{q}} \lesssim q^C \eps^{-1}$. Combining the two previous estimates then yields $\left< |\nabla w^{\eps}(x_0)|^q \right>^{\frac{1}{q}} \lesssim q^C \eps^{-1}$. To complete the argument we then make use of \cite[Theorem 1 \& Theorem 2]{GNO_final} and Lemma~\ref{CorrectorRegularity}.

\end{proof}

\section{Argument for Proposition \ref{Systematic_error}: Bound for expectation of weighted average of the boundary layer}

\label{expectation}

We next prove the estimate on $\mathbb{E}[\nabla \theta_i]$ in case $\domain=\Hd$.
To this aim we rely critically on the stationarity of the boundary layer corrector (and also the whole-space corrector) with respect to shifts tangential to the boundary. 

\newcommand{\Sym}{\mathfrak{S}}

\subsection{Proof of Proposition \ref{Systematic_error}}

We remark that throughout this section we are once again only on the half-space $\Hd$ and may, therefore, \wolog set $\eps=1$.

For our proof of Proposition \ref{Systematic_error} we first express the integral on the left-hand side of \eqref{M:Estim:01} in terms of $\bar h\in \dot H^1(\Hd)$, which is the decaying solution of 
\begin{equation}
\label{Definitionh_hom}
\begin{aligned}
-\nabla \cdot (\bar{a}^*\nabla \bar{h}) &= \nabla \cdot g && \text{in }\Hd,
\\
\bar{h}& =  0&&\text{on }\partial\Hd.
\end{aligned}
\end{equation} 
Notice $\bar{h}$ is the homogenized solution corresponding to the heterogeneous solution $h$ solving $-\nabla \cdot (a^\transpose \nabla h) = \nabla \cdot g$. We, in particular, show the following identity:

\begin{lemma}\label{M:Lem:1} Adopting the assumptions of Proposition \ref{Systematic_error} and letting $\bar{h}$ solve \eqref{Definitionh_hom}, we have that 
	\begin{equation}
	\label{M:Id:01}
	\left\langle\int_{\Hd} \nabla \theta_i \cdot g \dx \right\rangle
	=
	\int_{\big[x_0^{\perp} - 2r, x_0^{\perp} + 2r\big]  \times \Rd^{d-1}} \en{\nabla\theta_{i}\cdot \bigg(\phi_j^* a^\transpose-\sigma_j^*\bigg)}\nabla\partial_j \bar{h} \dx,
	\end{equation}
for $j = 1, \ldots, d$. 
\end{lemma}
\noindent We prove Lemma \ref{M:Lem:1} in Section \ref{exp_aux} below and first show how it entails Proposition~\ref{Systematic_error}.

\begin{proof}[Proof of Proposition \ref{Systematic_error}] Recall that w.\,l.\,o.\,g.\ we have set $\eps=1$. Using Lemma \ref{M:Lem:1} and the H\"older estimate we obtain that
	\begin{align}
	\label{M:Num:014}
	\left|\left\langle\int_{\Hd} \nabla \theta_{i} \cdot g \dx \right\rangle\right| \leq
	\int_{\big[x_0^{\perp} - 2r, x_0^{\perp} + 2r\big]  \times \mathbb{R}^{d-1}} \left\langle|\nabla\theta_{i}|^2 \right\rangle^{\frac{1}{2}} \left\langle \left|\phi_j^* a^\transpose-\sigma_j^*\right|^2\right\rangle^{\frac{1}{2}} |\nabla\partial_j \bar{h}| \dx.
	\end{align}
	Using our assumption \eqref{M:Hypo:1} on the decay of $|\nabla \theta_i|$ as well as the corrector bounds from \eqref{GNO_bound_balls}, we then obtain 
	\begin{equation}\label{M:Num:020}
	\begin{aligned}
	\left|\left\langle\int_{\Hd} \nabla \theta_{i} \cdot g \dx \right\rangle\right|
	&\lesssim \big( 1+ x_0^{\perp}\big)^{-n} \bigg(\int_{\big[x_0^{\perp} - 2r, x_0^{\perp} + 2r\big]  \times Q_{4r}(x_0^{\parallel}) }  |\nabla\partial_j \bar{h}| \dx
	\\&~~~~~~~~~~~~~~~~~~~~~~~~~
	+\int_{\big[x_0^{\perp} - 2r, x_0^{\perp} + 2r\big]  \times \big(\mathbb{R}^{d-1}\backslash Q_{4r}(x_0^{\parallel})\big)}  |\nabla\partial_j \bar{h}| \dx \bigg)
	\end{aligned}
	\end{equation}
	where $Q_{4r}(x_0^{\parallel})$ is the $d-1$-dimensional box in $\partial \Hd$ of side length $4r$ centered at the point $x_0^{\parallel}$. 
	
	Treating the first term on the right-hand side of \eqref{M:Num:020} is a simple matter of using the H\"older inequality and the energy estimate applied to a differentiated form of~\eqref{Definitionh_hom}:
\begin{align}\label{M:Num:022}
\int_{\big[x_0^{\perp} - 2r, x_0^{\perp} + 2r\big]  \times Q_{4r}(x_0^{\parallel})  } |\nabla\partial_j \bar{h}| \dx
&\lesssim
r^{\frac{d}{2}} \bigg(\int_{\big[x_0^{\perp} - 2r, x_0^{\perp} + 2r\big]  \times Q_{4r}(x_0^{\parallel})  } |\nabla\partial_j \bar{h}|^2 \dx\bigg)^{\frac{1}{2}} 
\\& \nonumber
\lesssim 
r^{\frac{d}{2}} \bigg(\int_{\mathbb{R}^d} |\nabla g|^2 \dx\bigg)^{\frac{1}{2}}
\stackrel{\eqref{g_avg_2}}{\lesssim} r^{-1}.
\end{align}
In order to handle the second term the right-hand side of \eqref{M:Num:020} we first consider the case that $\bar a^\transpose$ is a multiple of the identity matrix. We notice that, denoting by $\bar{G}^*$ the Green's function associated with the operator $-\nabla \cdot \bar{a}^* \nabla$ and letting $\Sym$ be the matrix associated with the reflection with respect to $\partial \Hd$, \ie 
	\begin{equation*}
	\Sym=-e_1 \otimes e_1 + \sum_{j=2}^{d} e_j \otimes e_j,
	\end{equation*}
we can write 
	\begin{equation}
	\begin{split}
	\label{M:Def_hbar3}
	\bar{h}(x):=&
	\int_{\Hd} \bigg(\bar{G}^*(x-y) - \bar{G}^*(x-\Sym y) \bigg) \nabla \cdot g(y) {\rm{ d}} y
	\\
	=&
	\int_{\Hd} \bigg(\nabla \bar{G}^*(x-y) - \Sym \nabla \bar{G}^*(x-\Sym y)\bigg)\cdot g(y) {\rm{ d}} y.
	\end{split}
	\end{equation}
Using this in combination with the classical estimates on the homogenized Green function $\bar{G}^*$ and its derivatives yields the following bound for $x \notin \supp(g)$ and $k \in \mathbb{N}_0$:
	\begin{equation}
	\label{M:Bound_1}
	\left|\nabla^k \bar{h}(x)\right| \lesssim  \frac{\|g\|_{{\rm L}^1(\Hd)}}{\dist(x,\supp(g)))^{d-1+k}} .
	\end{equation}
	If $\bar a^\transpose$ is not a multiple of the identity matrix, the estimate \eqref{M:Bound_1} may be derived by transforming $\bar a^\transpose$ to a multiple of the identity matrix by a linear change of coordinates (which maps the half-space $\Hd$ to another half-space), obtaining the bound \eqref{M:Bound_1} in the transformed coordinates, and changing back to the original coordinates.
%	In particular, for any $x \in [R,4R]\times \bigg(\mathbb{R}^{d-1}\backslash [-2R,2R]^{d-1}\bigg)$, there holds
%	\begin{equation*}
%	\left|\nabla^2 \bar{h}(x)\right| \lesssim |x^\parallel|^{-d-1} \left\|g\right\|_{{\rm{L}}^1(\Hd)}.
%%	\lesssim R^d |x^\parallel|^{-d-1}.
%	\end{equation*}
	Having derived \eqref{M:Bound_1}, using the Poincar\'e and the H\"older inequalities, we get
	\begin{equation}\label{M:Num:023}
	\begin{aligned}
	&\int_{\big[x_0^{\perp} - 2r, x_0^{\perp} + 2r\big]  \times \big(\mathbb{R}^{d-1}\backslash Q_{4r}(x_0^{\parallel})\big)}   |\nabla\partial_j \bar{h}| \dx
	\\&
	\lesssim 
	\left\|g\right\|_{{\rm{L}}^1(\Hd)} r \int_{\mathbb{R}^{d-1}\backslash Q_{4r}(x_0^{\parallel})} |x^\parallel - x_0^{\parallel}|^{-d-1} {\rm{d}} x^\parallel
	\\
	&
	\lesssim   r^{-1} \left\|g\right\|_{{\rm{L}}^1(\Hd)}
	\lesssim r^{d-1} \bigg(\fint_{\mathbb{R}^d} |g|^{2} \dx\bigg)^{\frac{1}{2}}
	\stackrel{\eqref{g_avg_2}}{\lesssim} r^{-1}.
	\end{aligned}
	\end{equation}
	Inserting \eqref{M:Num:022} and \eqref{M:Num:023} into \eqref{M:Num:020} yields \eqref{M:Estim:01}.
\end{proof}

\subsection{Proof of auxiliary lemma for Proposition \ref{Systematic_error}}
\label{exp_aux}

%For the sake of simplicity, we first establish Lemma \ref{M:Lem:1} in dimension $d > 3$.
%Then, we extend this proof to the case $d = 2$, only stressing the changing arguments.

\begin{proof}[Proof of Lemma \ref{M:Lem:1}]
The proof is divided in four Steps.
In Step 1, we establish the identity
\begin{equation}
	\label{M:Num:002}
	\begin{aligned}
	\int_{\Hd} \nabla \theta_{i} \cdot g  \dx
	&=\int_{\Hd}\nabla\theta_{i}\cdot \big(\phi_j^* a^\transpose-\sigma_j^*\big)\nabla\partial_i \bar{h}  \dx
	\\
	&\quad + \int_{\partial \Hd}\Big( \phi^*_j e_1 \cdot a\nabla \theta_{i} - \phi_{i}e_1 \cdot a^\transpose \big(e_j+\nabla \phi^*_j\big)\Big)  \partial_j\bar{h} \,{\rm{d}} S
	\\
	&\quad -\int_{\partial \Hd}\phi_{i}e_1 \cdot \sigma_j^*\nabla\partial_j \bar{h} \,{\rm{d}} S
	\\
	&=:I_1+I_2+I_3.
	\end{aligned}
\end{equation}
	Then, we show that $\bar{h}$ satisfies the following identity:
	\begin{equation}\label{M:Vanish_hbar}
		\int_{Re_1 + \partial \Hd} \nabla^k \bar{h} \,{\rm{d}} S=0 \qquad \quad\text{for all } k \in \mathbb{N} \setminus \{0\} \text{ and } R \in [0,x_0^\perp-r) \cup (x_0^\perp+r,+\infty).
	\end{equation}
	For the proof of \eqref{M:Vanish_hbar}, we proceed in two steps: in Step 2, we only consider derivatives that are tangential to $\partial \mathbb{H}^d_+$, and, in Step 3, we consider perpendicular derivatives. In Step 4 we take the expected value and use stationarity in the tangential direction to conclude.

	\paragraph{Step 1 : Argument for \eqref{M:Num:002}}
	By classical computations, denoting the two-scale expansion as
	\begin{equation}
	\label{M:Def_htilde}
	\widetilde{h}:=\bar{h}+\phi_j^*\partial_j\bar{h},
	\end{equation} 
	we have that
	\begin{equation}\label{M:Classical}
	-\nabla \cdot (a^\transpose \nabla \widetilde{h})=
	\nabla \cdot g + \nabla \cdot \big( \big(\sigma_j^* -\phi_j^* a^\transpose  \big)\nabla\partial_j \bar{h} \big) \qquad \quad\text{in } \Hd.
	\end{equation}
	Therefore,
	\begin{equation}\label{M:Num:001}
	\begin{aligned}
	\int_{\Hd} \nabla \theta_{i} \cdot g \dx
	\overset{\eqref{M:Classical}}{=}&
	-\int_{\Hd}\nabla\theta_{i}\cdot \big(a^\transpose \nabla \widetilde{h} +\big(\sigma_j^* -\phi_j^* a^\transpose  \big)\nabla\partial_j \bar{h}  \big) \dx
	\\
	& - \int_{\partial \Hd}\theta_{i} e_1  \cdot \big(a^\transpose \nabla \widetilde{h} +\big(\sigma_j^* -\phi_j^* a^\transpose  \big)\nabla\partial_j \bar{h}  \big)\,{\rm{d}} S.
	\end{aligned}
	\end{equation}
	Also, we remark that, by \eqref{defn_theta}, \eqref{Definitionh_hom}, and \eqref{M:Def_htilde} we have
	\begin{align*}
	\int_{\Hd}\nabla\theta_{i}\cdot a^\transpose \nabla \widetilde{h} \dx
	&=-\int_{\partial \Hd} \phi^*_j e_1 \cdot a\nabla \theta_{i} \, \partial_j\bar{h} \,{\rm{d}} S,
	\end{align*}
	and by definition \eqref{M:Def_htilde}
	\begin{align*}
	&\int_{\partial \Hd}\theta_{i} e_1 \cdot \big(a^\transpose \nabla \widetilde{h} +\big(\sigma_j^* -\phi_j^* a^\transpose  \big)\nabla\partial_j \bar{h}  \big) {\rm{d}} S
	\\&
	=
	\int_{\partial \Hd}\theta_{i} e_1 \cdot a^\transpose \big(e_j+\nabla \phi^*_j\big) \partial_j \bar{h} \,{\rm{d}} S +\int_{\partial \Hd}\theta_{i}e_1 \cdot \sigma_j^*\nabla\partial_j \bar{h}\,  {\rm{d}} S.
	\end{align*}
	As a consequence, inserting the two above identities into \eqref{M:Num:001} and recalling that $\theta_{i} = \phi_{i}$ on $\partial \Hd$ yields \eqref{M:Num:002}.
	
%	\paragraph{Step 2: Argument for \eqref{M:Def_hbar3}}
%	Thus, we show that $\bar{h}$ is the decaying solution to \eqref{M:Def_hbar}.
%	By integration by parts, \eqref{M:Def_hbar3} is equivalent to
%	\begin{equation}\label{M:Def_hbar2}
%	\bar{h}(x)=\int_{\Hd} \bigg(\bar{G}^*(x-y) - \bar{G}^*(x-\Sym y) \bigg) \nabla \cdot g(y) {\rm{d}} y.
%	\end{equation}
%	If $x \in \partial \Hd$, there holds $\bar{G}^*(x-y) - \bar{G}^*(x-\Sym y)=0$, so that $\bar{h}(x)=0$.
%	Moreover, by definition, if $x \in \Hd$, there holds
%	\begin{align*}
%	-\nabla \cdot \bar{a} \nabla \bar{h}(x)
%	&=-\int_{\Hd} \nabla \cdot \bar{a} \nabla\bigg(\bar{G}^*(x-y) - \bar{G}^*(x-\Sym y) \bigg) \nabla \cdot g(y) {\rm{d}} y
%	\\
%	&=\int_{\Hd}\bigg(\delta_{x}(y) - \delta_{x}(\Sym y) \bigg) \nabla \cdot g(y) {\rm{d}} y
%	=\nabla \cdot g(x).
%	\end{align*}
%	This shows that	$\bar{h}$ defined by \eqref{M:Def_hbar} satisfies \eqref{}.
	
	\paragraph{Step 2 : Argument for \eqref{M:Vanish_hbar} for tangential derivatives.}
	Let $R \in \mathbb{R}$, $j = 2, \ldots, d$, and $k \in \mathbb{N}_0$.
	We show that
	\begin{equation}
	\label{M:Step:2}
	\int_{Re_1+\partial \Hd} \partial_j \nabla^k \bar{h} \, {\rm{d}} S=0.
	\end{equation}
	To show \eqref{M:Step:2}, we make use of the divergence theorem and the decay \eqref{M:Bound_1} to the extent of
	\begin{align*}
	\left|\int_{Re_1+\partial \Hd} \partial_j \nabla^k \bar{h} \, {\rm{d}} S\right|
	&\leq  \limsup_{l\uparrow\infty} \left|\int_{\partial [-l,l]^{d-1}} |\nabla^k \bar{h}(R,x^\parallel) | \, {\rm{d}} S(x^\parallel)\right|
	\\
	&\lesssim C(R) \limsup_{l\uparrow\infty} l^{d-2} l^{-d+1-k}=0.
	\end{align*}
	
	\paragraph{Step 3 : Argument for \eqref{M:Vanish_hbar} for perpendicular direction.}
	Let  $R \in [0,x_0^\perp-r) \cup (x_0^\perp+r,+\infty)$.
	Notice that, since the support of $g$ is disjoint from $Re_1+\partial \Hd$, we have that $\nabla \cdot \bar{a} \nabla \bar{h}=0$ in a neighborhood of $Re_1+\partial \Hd$.
	Differentiating this equation $k$ times (for $k \in \mathbb{N}_0$) yields
	\begin{equation*}
	\bar{a}_{11}\partial_1\partial_1 \nabla^k \bar{h}(x)=-\sum_{(i,j) \neq (1,1)} \bar{a}_{ij} \partial_i\partial_j  \nabla^k\bar{h}(x) \qquad \quad\text{for all } x \in Re_1+\partial \Hd.
	\end{equation*}
	Using the previous result and that $\bar{a}^{*}$ is uniformly elliptic, \eqref{M:Step:2} entails
	\begin{equation}
	\label{M:Step:3_1}
	\int_{Re_1+\partial \Hd} \partial_1\partial_1 \nabla^k \bar{h} \, {\rm{d}} S=0.
	\end{equation}
	Therefore, for establishing \eqref{M:Vanish_hbar}, it only remains to show that
	\begin{equation}\label{M:Step:3_2}
	\int_{Re_1+\partial \Hd} \partial_1 \bar{h}  \,{\rm{d}} S=0.
	\end{equation}
	
	We retrieve this via an indirect argument, remarking that, by \eqref{M:Step:2} and since $\bar{a}^*$ is elliptic, it suffices to establish
	\begin{equation}\label{M:Step:3_3}
	\int_{Re_1+\partial \Hd} (\bar{a} \nabla \bar{h}) \cdot e_1 \, {\rm{d}} S=0.
	\end{equation} 
	% which is conveniently interpreted as the flux through the surface $Re_1+\partial \Hd$.
	
	Here comes the argument for \eqref{M:Step:3_3}.
	Since $g$ is compactly supported in the interior of $\Hd$, by the divergence theorem, we obtain
	\begin{equation*}\label{new_label_exp}
	\begin{aligned}
	-\int_{[R,R+l] \times [-l,l]^{d-1}} \nabla \cdot \bar{a} \nabla \bar{h} \dx
	&=\int_{[R,R+l] \times [-l,l]^{d-1}} \nabla \cdot g \dx
	\\&
	=\int_{\partial\big([R,R+l] \times [-l,l]^{d-1} \big)} g \cdot \vec{n}\, {\rm{d}} S =0,
	\end{aligned}
	\end{equation*}
	for $l \gg x_0^\perp+r$. Using once more the divergence theorem yields
	\begin{equation}
	\label{M:Num:010}
	\begin{aligned}
	\int_{[R,R+l] \times [-l,l]^{d-1}} \nabla \cdot \bar{a} \nabla \bar{h} \dx
	=&-\int_{\{R\} \times [-l,l]^{d-1}}e_1 \cdot \bar{a} \nabla \bar{h} \, {\rm{d}} S
	\\
	&+\int_{[R,R+l] \times \partial [-l,l]^{d-1}}\vec{n} \cdot \bar{a} \nabla \bar{h} \, {\rm{d}}S
	\\
	&+\int_{\{R+l\} \times [-l,l]^{d-1}} e_1 \cdot \bar{a} \nabla \bar{h} \, {\rm{d}} S.
	\end{aligned}
	\end{equation}
	By \eqref{M:Bound_1}, the second and third terms on the right-hand side are controlled as follows
	\begin{align}\label{M:Num:011}
	&\left|\int_{[R,R+l] \times \partial [-l,l]^{d-1}}\vec{n} \cdot \bar{a} \nabla \bar{h} \, {\rm{d}}S
	+\int_{\{R+l\} \times [-l,l]^{d-1}} e_1 \cdot \bar{a} \nabla \bar{h} \, {\rm{d}} S\right|
	\\&\nonumber
	\lesssim
	C(r) l^{d-1}l^{-d} \lesssim C(r) l^{-1},
	\end{align}
	so that, letting $l\uparrow \infty$ and using \eqref{new_label_exp}, we retrieve \eqref{M:Step:3_3} as
	\begin{equation*}
	\int_{Re_1+\partial \Hd} e_1 \cdot \bar{a} \nabla \bar{h} \, {\rm{d}} S
	=
	\lim_{l \uparrow \infty}\int_{\{R\} \times [-l,l]^{d-1}}e_1 \cdot \bar{a} \nabla \bar{h} \, {\rm{d}} S = 0.
	\end{equation*}
	
	\paragraph{Step 4 : Argument for \eqref{M:Id:01}}
	Taking the expectation of \eqref{M:Num:002}, we get
	\begin{equation}\label{M:Num:100}
	\left\langle \int_{\Hd} \nabla \theta_{i} \cdot g \dx \right\rangle
	=:\left\langle I_1 \right\rangle+\left\langle I_2 \right\rangle +\left\langle I_3\right\rangle.
	\end{equation}
	We first show that
	\begin{equation}
	\label{M:Num:015}
	\left\langle I_2 \right\rangle=\left\langle I_3 \right\rangle=0,
	\end{equation}
	and then that
	\begin{equation}
	\label{M:Num:016}
	\left\langle I_1 \right\rangle= \int_{[x_0^\perp-r,x_0^\perp+r]\times \mathbb{R}^{d-1}} \left\langle\nabla\theta_{i}\cdot \big(\phi_j^* a^\transpose-\sigma_j^*\big)\right\rangle\nabla\partial_j \bar{h} \dx.
	\end{equation}
	Indeed, the identity \eqref{M:Id:01} is obtained by inserting \eqref{M:Num:015} and \eqref{M:Num:016} into \eqref{M:Num:100}.
		
	Here comes the argument for \eqref{M:Num:015}.
	Using the stationarity of $a^\transpose$, $\nabla \theta_{i}$, $\phi_{i}$, $\phi_j^*$ and $\sigma_j^*$ with respect to tangential shifts, we may define the (constant) tensors
	\begin{equation*}
	\begin{cases}
	M_{j}:= \left\langle \big(\phi^*_j e_1 \cdot a\nabla \theta_{i} - \phi_{i}e_1 \cdot a^\transpose \big(e_j+\nabla \phi^*_j\big)\big)(0)\right\rangle,
	\\
	N_{jk}:=-\left\langle \big( \phi_{i} e_1 \cdot \sigma_j^* e_k\big)(0) \right\rangle,
	\end{cases}
	\end{equation*}
	and rewrite $\left\langle I_2 \right\rangle$ and $\left\langle I_3 \right\rangle$ as
	\begin{equation*}
	\left\langle I_2 \right\rangle=M_j \int_{\partial \Hd} \partial_j \bar{h} \, {\rm{d}} S \qquad\quad\text{and } \qquad \left\langle I_3 \right\rangle=N_{jk} \int_{\partial \Hd} \partial_j\partial_k \bar{h} \,{\rm{d}} S.
	\end{equation*}
	Hence, \eqref{M:Vanish_hbar} immediately entails \eqref{M:Num:015}.
	
	Last, we show \eqref{M:Num:016}.
	Taking the expectation of $I_1$, we obtain
	\begin{equation*}
	\left\langle I_1 \right\rangle =\int_{\Hd}\left\langle \nabla\theta_{i}\cdot \big(\phi_j^* a^\transpose-\sigma_j^*\big)\right\rangle\nabla\partial_j \bar{h} \dx.
	\end{equation*}
	By \eqref{M:Vanish_hbar}, since $\left\langle \nabla\theta_{i}\cdot \big(\phi_j^* a^\transpose-\sigma_j^*\big)\right\rangle(x)$ only depends on $x^\perp$ (by the stationarity of $\nabla\theta_{i}, \phi_j^*, a^\transpose,$ and $\sigma_j^*$ in all directions but $e_1$), we may rewrite the above integral as
	\begin{equation*}
	\left\langle I_1 \right\rangle =\int_0^\infty \left\langle \big(\nabla\theta_{i}\cdot \big(\phi_j^* a^\transpose-\sigma_j^*\big)\big)(x^\perp,0)\right\rangle  \bigg( \int_{\mathbb{R}^{d-1}} \nabla\partial_i \bar{h}(x^\perp,x^\parallel) \mathrm{d}x^\parallel \bigg) \mathrm{d}x^\perp.
	\end{equation*}
	Using \eqref{M:Vanish_hbar} then yields \eqref{M:Num:016}.
\end{proof}

\appendix
\section{Regularity of random elliptic operators on $\Rd^d$ and corrector estimates}
\label{Appendix_A}

We first summarize the contents of \cite[Theorem 1 \& Theorem 2]{GNO_final}, which hold under the assumptions (A1)-(A3). In particular, we have that:

\begin{theorem*}[Theorem~1 \& Theorem~2 in \cite{GNO_final}] \label{summary_GNO} Let assumptions (A1) -- (A3) be satisfied. Then there exists a random field $r^*=r^*(a,x)$ such that $\frac{r^*}{\varepsilon}$ has stretched exponential moments in the sense of \eqref{stretched} with the following property: Let $u \in H^1_{\loc}(\Rd^d)$ be $a$-harmonic in $B_R(x_0)$ for $R>0$ and $x_0 \in \Rd^d$, \ie  suppose that $u$ satisfies 
\begin{align*}
-\nabla \cdot (a\nabla u)  &= 0 \quad \quad\text{in } \quad \, B_R(x_0).
\end{align*}
Then, for any $r,R$ with $R\geq r \geq r^*(x_0,a)>0$, it holds that 
\begin{align}
\label{large_scale_mvp}
\fint_{B_r(x_0)} |\nabla u|^2 \, \dx \lesssim_{d,\lambda} \fint_{B_R(x_0)} |\nabla u|^2 \, \dx.
\end{align}
\end{theorem*}

In the current contribution we have also used the following estimate for the first-order whole-space homogenization corrector that has been proven in \cite[Theorem~2]{GNO5} under the assumption that the ensemble satisfies a log-Sobolev inequality. The result can be found in the more general setting of homogenization for nonlinear uniformly elliptic systems under the assumption that the ensemble satisfies a spectral gap in \cite[Corollary~15]{FN_2020}, or for the degenerate linear elliptic setting in \cite{Bella_Kniely}. 
\begin{theorem*}
\label{CorrectorBoundGNO4}
Under assumptions (A1)--(A3) and for $d\geq 3$, there exists a random field $\C(a, x)$ with stretched exponential moments in the sense \eqref{stretched} such that 
\begin{align}
\label{GNO_bound_balls}
\sup_{r\geq \eps} \bigg(\fint_{B_r(x_0)} |\phi^\eps|^2+|\sigma^\eps|^2 \, \dd x \bigg)^{\frac{1}{2}} \leq 
%\begin{cases}
% \mathcal{C}(a,x)  \varepsilon \sqrt{ \log{\frac{1}{\varepsilon}}} & \textrm{if } d =2 \\
\mathcal{C}(a,x_0)  \varepsilon
%&\textrm{if } d \geq 3
%\end{cases}
\end{align}
holds for any $x_0 \in \Rd^d$.
\end{theorem*}

Relying in addition on assumption (A4), we also obtain the following small-scale regularity properties of the correctors. For their (short) proof, we refer to the companion paper \cite{BFJR_1}.
\begin{lemma}
\label{CorrectorRegularity}
Let assumptions (A1)--(A4) be satisfied. Then for any $0<\gamma<\nu$ and $x_0\in \mathbb{R}^d$ there exists a random constant $\mathcal{C}(a,x_0)$ with a uniform bound on suitable stretched exponential moments
\begin{align}
\label{Bidule:0001}
\left \langle \exp\bigg(\mathcal{C}(a,x_0)^{1/C}/C\bigg) \right \rangle \leq 2
\end{align}
with $C$ depending possibly on $\gamma$ but not on $x_0$ such that the following is true: The estimates
\begin{align*}
|\nabla \phi_i^\eps(x)-\nabla \phi_i^\eps(y)|\leq \mathcal{C}(a,x_0) \frac{|x-y|^\gamma}{\varepsilon^\gamma}
\end{align*}
and
\begin{align*}
|\nabla \sigma_{ijk}^\eps(x)-\nabla \sigma_{ijk}^\eps(y)|\leq \mathcal{C}(a,x_0) \frac{|x-y|^\gamma}{\varepsilon^\gamma}
\end{align*}
hold for any $x,y\in B_\eps(x_0)$.

Furthermore, if additionally $\domain$ is a bounded $C^{1,1}$ domain or if $\domain=\Hd$ the following holds true: For any $0<\gamma<\nu$ and any $x_0 \in \overline{\domain}$ there exists a random constant $\mathcal{C}(a,x_0)$ with a uniform bound on suitable stretched exponential moments, \eg, \eqref{Bidule:0001}, such that the estimate
\begin{align*}
|\eps \nabla \theta^\eps_i(x)-\eps \nabla \theta^\eps_i(y)|\leq \mathcal{C}(a,y) \frac{|x-y|^\gamma}{\varepsilon^\gamma}
\end{align*}
holds for any $x,y\in \overline{\domain}$ with $|x-x_0|\leq \varepsilon$ and $|y-x_0|\leq \varepsilon$.
\end{lemma}
%\begin{proof}
%Relying on assumptions (A1) and (A4), for any $x_0\in \mathbb{R}^d$ there exists a small enough radius $\rho>0$ with $\mathbb{E}[\exp((\eps/\rho)^{1/C}/C)] \leq 2$ such that by classical Schauder theory we have for any $a$-harmonic function $u$ in $B_\rho(x_0)$
%\begin{align*}
%||\nabla u||_{C^\alpha(B_{\rho/2}(x_0))} \leq C \bigg(\fint_{B_\rho(x_0)} |\nabla u|^2 \dx \bigg).
%\end{align*}
%Applying this bound to the $a$-harmonic function $x_i+\phi_i$ and using a covering argument to cover an $\eps$-ball and inserting the bound \eqref{GNO_bound_balls}, we arrive at our estimates for $\phi_i$. Given the regularity bound for $\phi_i$, the argument for $\sigma_{ijk}$ is similar, though we now need to account for a right-hand side. The bounds for $\theta_i^\eps$ are analogous to the case of $\phi_i$, possibly (depending on the location of $x_0$) also making use of boundary regularity via Schauder theory as well as the bound on $\eps \nabla \theta_i^\eps$ from \cite[Proposition~2]{BFJR_1}.
%\end{proof}
Finally, we have used the following Poisson kernel type estimate on solutions to the linear elliptic PDE.
\begin{proposition}
Let $\Omega\subset \mathbb{R}^d$ be a bounded domain and let $a$ be a uniformly elliptic and bounded coefficient field for which a Lipschitz regularity theory for the operators $-\nabla \cdot a\nabla$ and $-\nabla \cdot a^\ast \nabla$ with homogeneous Dirichlet boundary conditions holds in the following sense: For any $x_0\in \overline{\Omega}$ there exists $\mathcal{C}(a,x_0)>0$ such that
\begin{align*}
|\nabla v(x_0)|\leq \mathcal{C}(a,x_0) \bigg(\fint_{B_r(x_0)} |\nabla v|^2 \dx\bigg)^{1/2}
\end{align*}
holds for any $a$-harmonic or $a^\transpose$-harmonic function $v$ in $B_r(x_0)$ with $v=0$ on $\partial\Omega \cap B_r(x_0)$. Then any $a$-harmonic function $u\in H^1(\Omega)$ satisfies
\begin{align}
\label{PoissonKernelEstimateAHarmonic}
|\nabla u(x_0)| \leq \mathcal{C}(a,x_0) \int_{\partial \Omega} \frac{\mathcal{C}(a,x) |u(x)|}{|x-x_0|^d} \,dS.
\end{align}
\end{proposition}
\begin{proof}
For any fixed $i\in \{1,\ldots,d\}$, (for now formally) define $h$ as the solution
\begin{align*}
-\nabla \cdot (a^\ast \nabla h) &= \partial_i \delta_{x_0}\quad\text{in }\Omega,
\\
h&=0\quad \text{on }\partial\Omega.
\end{align*}
Using the equations satisfied by $h$ and by $u$, we then have the representation
\begin{align*}
\partial_i u(x_0) &= \int_{\partial \Omega} u \vec{n}\cdot a^\ast \nabla h \,dS -\int_\Omega a^\ast \nabla h \cdot \nabla u \,dx
\\&
=\int_{\partial \Omega} u \vec{n}\cdot a^\ast \nabla h \,dS.
\end{align*}
To establish our proposition, it thus suffices to show
\begin{align}
\label{BoundForH}
|\nabla h(y)|\leq \frac{\mathcal{C}(a,x_0)\mathcal{C}(a,y)}{|y-x_0|^d}
\end{align}
for any $y,x_0$.

The latter estimate is a consequence of a duality estimate and the regularity estimate: Set $r:=|y-x_0|/2$. For any $g\in L^2(\Omega)$ with $\supp g\subset B_r(y)$ and $||g||_{L^2}\leq 1$, introduce the solution to the PDE $-\nabla \cdot (a\nabla v)=\nabla \cdot g$ with homogeneous Dirichlet boundary conditions. We obtain by an energy estimate
\begin{align*}
\int_\Omega |\nabla v|^2 \,dx
\leq C \int_\Omega |g|^2 \,dx \leq C.
\end{align*}
As $v$ is $a$-harmonic in $B_r(x_0)$, a regularity estimate implies
\begin{align*}
|\nabla v(x_0)|\leq \mathcal{C}(a,x_0) r^{-d/2}. 
\end{align*}
We therefore obtain
\begin{align*}
\int_\Omega g\cdot \nabla h \,dx
=-\int_\Omega a\nabla v \cdot \nabla h \,dx = (\partial_i v)(x_0) \leq \mathcal{C}(a,x_0) r^{d/2}.
\end{align*}
Taking the supremum with respect to all $g$ with $\supp g\subset B_r(y)$ and $||g||_{L^2}\leq 1$, we obtain
\begin{align*}
||\nabla h||_{L^2(B_r(y))} \leq \mathcal{C}(a,x_0) r^{-d/2}.
\end{align*}
Since $h$ is $a$-harmonic on $B_r(y)$, a regularity estimate now entails
\begin{align*}
|\nabla h|(y)\leq \mathcal{C}(a,y) \bigg(\fint_{B_r(y)} |\nabla h|^2 \,dx\bigg)^{1/2},
\end{align*}
yielding the conclusion \eqref{BoundForH}.

To justify the above considerations rigorously, one uses an approximation argument for the $\delta_{x_0}$ measure.
\end{proof}

\section*{Funding}

The work of PB was supported by the Deutsche Forschungsgemeinschaft (DFG, German Research Foundation) under the priority program SPP 2256, Project number 441469601.
This research was funded in whole or in part by the Austrian Science Fund (FWF) ESP4053024. For open access purposes, the author has applied a CC BY public copyright license to any author-accepted manuscript version arising from this submission.
This project has received funding from the European Research Council (ERC) under the European Union's Horizon 2020 research and innovation programme (grant agreement No 948819). \includegraphics[height=3.0mm]{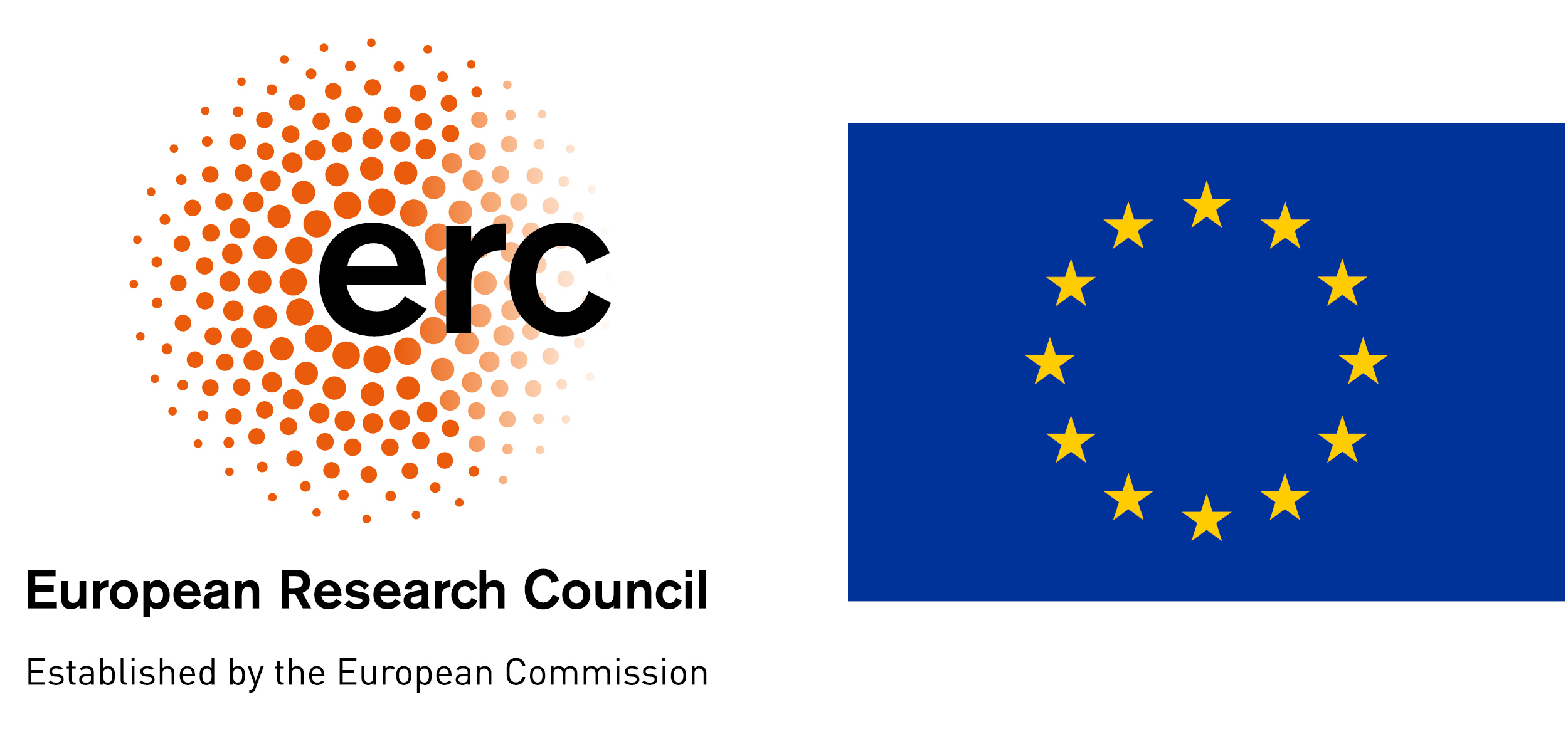}

\bibliographystyle{abbrv}
\bibliography{stochastic_homogenization}

\end{document}